\documentclass[12pt]{article}

\usepackage[utf8]{inputenc}

\usepackage{amsfonts,graphicx,amsmath,amssymb,amsthm,geometry,
hyperref,color,nicefrac,enumerate,bbm,verbatim,lmodern,url, mathrsfs}
\usepackage{xcolor} 

\usepackage[draft]{todonotes}   

\allowdisplaybreaks[1]

\newtheorem{lemma}{Lemma}[section]
\newtheorem{corollary}[lemma]{Corollary}

\newtheorem{theorem}[lemma]{Theorem}

\newtheorem{setting}[lemma]{Setting}

\renewcommand{\O}{\mathcal O}

\renewcommand{\L}{\mathcal{L}}

\providecommand{\X}{{\ensuremath{\mathcal{X}}}}

\providecommand{\N}{{\ensuremath{\mathbbm{N}}}}
\providecommand{\R}{{\ensuremath{\mathbbm{R}}}}
\providecommand{\E}{{\ensuremath{\mathbb{E}}}}
\providecommand{\B}{{\ensuremath{\mathcal{B}}}}
\providecommand{\F}{{\ensuremath{\mathcal{F}}}}
\providecommand{\f}{{\ensuremath{\mathbb{F}}}}
\providecommand{\M}{{\ensuremath{\mathcal{M}}}}
\renewcommand{\H}{{\ensuremath{\mathbb{H}}}}
\renewcommand{\P}{{\ensuremath{\mathbb{P}}}}
\providecommand{\1}{{\ensuremath{\mathbbm{1}}}}

\providecommand{\values}{{\ensuremath{\mathfrak{v}}}}
\providecommand{\N}{{\ensuremath{\mathbbm{N}}}}
\providecommand{\R}{{\ensuremath{\mathbbm{R}}}}

\providecommand{\E}{{\ensuremath{\mathbb{E}}}}
\renewcommand{\P}{{\ensuremath{\mathbb{P}}}}
\providecommand{\1}{{\ensuremath{\mathbbm{1}}}}
\providecommand{\HS}{{\ensuremath{\textup{HS}}}}

\allowdisplaybreaks 
\begin{document}

\title{Spatial Sobolev regularity for stochastic Burgers\\ equations with additive trace class noise}
\author{Arnulf Jentzen$^{1} $, Felix Lindner$^{2} $, and Primo\v{z} Pu\v{s}nik$^{3} $
	\bigskip 
	\\
	\small{$^1$ Seminar for Applied Mathematics, Department of Mathematics,}
	\\
	\small{ETH Zurich, Switzerland, 
		e-mail: arnulf.jentzen@sam.math.ethz.ch} 
	\smallskip
	\\
	\small{$^2$  Institute of Mathematics, Faculty of Mathematics and Natural Sciences,}
	\\
	\small{University of Kassel, Germany,
		e-mail: lindner@mathematik.uni-kassel.de}
	\smallskip
	\\
	\small{$^3$ Seminar for Applied Mathematics, Department of Mathematics,}
	\\
	\small{ETH Zurich, Switzerland,
		e-mail: primoz.pusnik@sam.math.ethz.ch}
}


\maketitle

\begin{abstract}
	In this article we investigate the spatial Sobolev regularity of mild solutions to stochastic Burgers equations with additive trace class noise. Our findings are based on a combination of suitable bootstrap-type arguments and a detailed analysis of the nonlinearity in the equation.
\end{abstract}

\tableofcontents
%
\newpage

\section{Introduction}
%
%
%
%
%

In the literature, there are nowadays various results on existence, uniqueness, and regularity of solutions
to 
stochastic Burgers equations. 
%
%
%
In particular, existence and uniqueness results for mild solutions to stochastic Burgers equations with additive space-time white noise and zero Dirichlet boundary conditions on the unit interval $ (0,1) $ 
taking values in the space 
$ L^p((0, 1), \R) $ for $ p \in [2,\infty) $, 
in the space
$ \mathcal{C}([0, 1], \R) $ of continuous functions, 
and in $ L^2((0, 1), \R) $-Sobolev-type spaces 
of order up to $ \nicefrac{1}{2} $
can be found, e.g., in
Da Prato et al.\ \cite{DaPratoDebusscheTemam1994}, 
Bl\"omker \& Jentzen~\cite{BloemkerJentzen2013}, 
Jentzen et al.\ \cite{JentzenSalimovaWelti2019},
and Mazzonetto \& Salimova~\cite{MazzonettoSalimova2019}.
Results on existence, uniqueness, and regularity of solutions to stochastic Burgers equations with multiplicative space-time white noise and zero Dirichlet boundary conditions  on the unit interval have been established, e.g., in 
Da Prato \& Gatarek~\cite{DaPratoGatarek1995} 
and Gy\"ongy~\cite{Gyongy1998}. 
Existence, uniqueness, and regularity results for solutions to stochastic Burgers equations on the whole real line  can be found, e.g., in Bertini et al.\ \cite{BertiniCancriniJonaLasinio1994},
Gy\"ongy \& Nualart~\cite{GyongyNualart1999},
Kim~\cite{Kim2006},
and
Lewis \& Nualart~\cite{LewisNualart2017}.
Results on  existence, uniqueness, and regularity of  mild solutions to stochastic Burgers equations driven by L\'evy noise are presented, e.g., in Dong \& Xu~\cite{DongXu2007} and Hausenblas \& Giri~\cite{HausenblasGiri2013}. 
We also refer to 
Brze{\'z}niak et al.\ 
\cite{brzezniak2011ergodic},
Da Prato \& Zabczyk~\cite[Section~14]{DaPratoZabczyk1996},  
Da Prato \& Zabczyk~\cite[Section~13.9]{DaPratoZabczyk2014}, 
R\"ockner et al.\ \cite{rrz2014},
and the references mentioned therein for further existence, uniqueness, 
and regularity results for stochastic Burgers-type
equations. 
%
%
%
%
In this paper, we 
present a higher order regularity result for stochastic Burgers equations with additive trace-class 
noise and zero Dirichlet boundary conditions on the unit interval $ (0,1) $. 
More specifically, 
in Theorem~\ref{theorem:existence_Burgers}, which is the main result of
this article, we establish the unique existence of mild solutions taking values in $ L^2((0, 1), \R) $-Sobolev-type
spaces of order up to $ 2 $. A slightly simplified version of our main result is given in the following
theorem.

\begin{theorem}
	\label{theorem:BurgersExistence}
	Let
	$ ( H, \langle \cdot, \cdot \rangle _H, \left \| \cdot \right\|_H ) $ 
	be the $ \R $-Hilbert space
	of equivalence classes
	of Lebesgue-Borel square-integrable functions
	from $ (0,1) $ to $ \R $,
	let
	$ A \colon D( A ) \subseteq H \to H $ 
	be the
	Laplacian with zero Dirichlet boundary conditions on $ H $,  
	let
	$ (H_r, \langle \cdot, \cdot \rangle_{H_r}, \left \| \cdot \right \|_{H_r} ) $, 
	$ r \in \R $, be a family of interpolation spaces associated to $ -A $,
	%
	%
	let   
	$ \beta \in ( - \nicefrac{1}{4}, \infty ) $,  
	$ \gamma \in ( \nicefrac{1}{4}, 
	\min \{ 1, \nicefrac{1}{2} + \beta \} ) $,
	$ T \in (0,\infty) $,
	$ \xi \in H_1 $, 
	$ B \in \HS(H, H_\beta) $,
	let
	$ ( \Omega, \F, \P ) $
	be a probability space, 
	and
	let
	$ (W_t)_{t\in [0,T]} $
	be an
	$ \operatorname{Id}_H $-cylindrical 
	Wiener process.
	Then 
	\begin{enumerate}[(i)] 
	\item \label{item:Nonlinearity F} there exists a unique
	continuous function 
	$ F \colon H_{ \nicefrac{1}{8} } \to H_{ - \nicefrac{1}{2} } $   
	which satisfies for every $ v \in H_{ \nicefrac{1}{2} } $
	that $ F(v) = - v'  v $
	and
	\item there exists an up to indistinguishability 
		unique stochastic process 
		$ X  \colon [0,T] \times \Omega \to H_\gamma $
		with continuous sample paths
		which satisfies that
		for every
		$ t \in [0,T] $ 
		it holds
		$ \P $-a.s.\ that
		\begin{equation} 
		\label{eq:model equation}
		X_t 
		= 
		e^{tA} \xi 
		+
		\int_0^t e^{(t-s)A} F(X_s) \, ds 
		+
		\int_0^t e^{(t-s)A} 
		B
		\, dW_s
		.
		\end{equation} 
\end{enumerate}
\end{theorem}
Theorem~\ref{theorem:BurgersExistence}
is a direct consequence of Theorem~\ref{theorem:existence_Burgers} 
(with
	$ T = T $,
	$ \varepsilon = 1 - \gamma $,
	$ c_0 = 1 $,
	$ c_1 = - 1 $,
	$ \beta = \beta $,
	$ \gamma = \gamma $, 
	$ A = A $,
	$ H_r = H_r $,
	$ ( \Omega, \F, \P ) 
	=
	( \Omega, \F, \P ) $,
	$ ( W_t )_{ t \in [0,T] } = ( W_t )_{ t \in [0,T] } $,
	$ B = B $,
	$ \xi = ( \Omega \ni \omega \mapsto \xi \in H_1 ) $ 
	for 
	$ r \in \R $,
	$ \gamma \in ( \nicefrac{1}{4}, 
	\min \{ 1, \nicefrac{1}{2} + \beta \} ) $ 
	in the notation of
	Theorem~\ref{theorem:existence_Burgers})
%
in 
Section~\ref{section:Existence}
below.
Note that the assumption in Theorem~\ref{theorem:BurgersExistence} above that 
$ (H_r, \langle \cdot, \cdot \rangle_{H_r}, \left \| \cdot \right \|_{H_r} ) $, 
$ r \in \R $,  
  is a family of interpolation spaces associated to 
  $ - A $ ensures that for every
  $ r \in [0, \infty) $ 
  	it holds that 
  	$ (H_r, \langle\cdot,\cdot\rangle_{H_r}, \left \| \cdot \right \|_{H_r} )
  	=
  	( D((-A)^r), 
  	\langle(-A)^r (\cdot),(-A)^r (\cdot )\rangle_H,
  	\left \| (-A)^r (\cdot) \right\|_H) $.
The equation in~\eqref{eq:model equation}
above is referred to as stochastic evolution equation (SEE)
or stochastic partial differential equation (SPDE) in the scientific literature
and,
roughly speaking,
there are mainly three common approaches for 
%
describing and
analyzing solutions of
SPDEs:
$ (i) $ the martingale measure approach (cf., e.g., Walsh~\cite{Walsh1986}), 
$ (ii) $ the variational (weak solution) approach 
(cf., e.g., 
Grecksch \& Tudor~\cite{GreckschTudor1996},
Liu \& R\"ockner~\cite{LiuRoeckner2015Book},
Pr\'ev\^ot \& R\"ockner~\cite{PrevotRoeckner2007},
and
Rozovski\u{\i}~\cite{r90}),
and 
$ (iii) $ the semigroup (mild solution)
approach
(cf., e.g., 
Da Prato \& Zabczyk~\cite{DaPratoZabczyk1996, DaPratoZabczyk2014},
Grecksch \& Tudor~\cite{GreckschTudor1996},
and Liu \& R\"ockner~\cite{LiuRoeckner2015Book})
in the literature.
Theorem~\ref{theorem:BurgersExistence}
and most of the other results in this article
are formulated within the semigroup approach.
The proof of 
Theorem~\ref{theorem:BurgersExistence}
and
Theorem~\ref{theorem:existence_Burgers},
respectively,
is mainly based on  
combining 
Corollary~\ref{corollary:Existence},
Lemma~\ref{lemma:Extension of Function F part1},
Corollary~\ref{corollary:Extension of Function F part2},
Lemma~\ref{lemma:AprioriBound3},
and
Lemma~\ref{lemma:PathwiseRates}.
Corollary~\ref{corollary:Existence}
establishes the unique existence of suitable spatial spectral Galerkin approximations of  stochastic Burgers equations 
(see the proof of 
Lemma~\ref{Lemma:Existence of approximation processes}
and~\eqref{eq:processEx}
in the proof of Theorem~\ref{theorem:existence_Burgers} below).
An existence and uniqueness result for stochastic differential equations (SDEs) similar to 
Corollary~\ref{corollary:Existence} 
can be found, e.g., in  
%
%
%
Liu \& R\"ockner~\cite[Theorem~3.1.1]{LiuRoeckner2015Book}.
Lemma~\ref{lemma:Extension of Function F part1}
and
Corollary~\ref{corollary:Extension of Function F part2}
(cf., e.g., Bl\"omker \& Jentzen~\cite[Lemma~4.7]{BloemkerJentzen2013})
prove
that the involved nonlinearity $ F $
(see item~\eqref{item:Nonlinearity F}
in Theorem~\ref{theorem:BurgersExistence} above)
satisfies specific local Lipschitz conditions
(see~\eqref{eq:Local lip}
in the proof of Theorem~\ref{theorem:existence_Burgers} below).
Lemma~\ref{lemma:AprioriBound3}
establishes appropriate  
pathwise uniform a priori bounds
for the  
spatial spectral Galerkin approximations of the considered stochastic Burgers equation
(see~\eqref{eq:unifBound} in
the proof of Theorem~\ref{theorem:existence_Burgers} below).
%
Its proof
is based on consecutive applications of
suitable bootstrap-type arguments
in Section~\ref{section:AprioriBound}
to establish appropriate a priori  
bounds for the solution processes
of the considered SDEs
in higher order smoothness spaces.
Related bootstrap-type arguments 
can be found, e.g., 
in 
Jentzen \& Pu\v{s}nik~\cite[Section~3]{JentzenPusnik2019Published},
Jentzen \& R\"ockner~\cite[Theorem~1]{jr12},
and
Zhang~\cite[Section~3]{Zhang2007}.
Lemma~\ref{lemma:PathwiseRates}
(cf., e.g., Bl\"omker \& Jentzen~\cite[Lemma~4.3]{BloemkerJentzen2013})
demonstrates  
pathwise uniform
convergence rates of
spatial spectral Galerkin approximations 
of the considered stochastic integral
(see~\eqref{eq:Galerking_Speed} in
the proof of Theorem~\ref{theorem:existence_Burgers} below).
Its proof is essentially based on an application of the factorization method for stochastic convolutions 
in Lemma~\ref{lemma:GalerkinRegularity}.
Combining these mentioned results
with
the existence and uniqueness result in Bl\"omker \& Jentzen~\cite[Theorem~3.1]{BloemkerJentzen2013}
proves Theorem~\ref{theorem:existence_Burgers}.
%
%
%
%

%
The remainder of this article is structured as follows.
In Section~\ref{section:strongApriori}
we recall some elementary
existence and uniqueness results for random 
ordinary differential equations (ODEs).
In 
Section~\ref{section:AprioriBound}
we
employ bootstrap-type arguments to establish    
suitable a priori bounds for certain approximation processes. 
In Subsection~\ref{subsection:Sobolev} 
we recall some elementary properties of  
Sobolev-Slobodeckij  
and interpolation
spaces. 
In Subsection~\ref{subsection:Nonlinearity} we  
recall and derive several auxiliary results on the regularity properties  
of the nonlinearity appearing in the
stochastic Burgers equation.
In Section~\ref{section:Existence}
we combine the results
in
Sections~\ref{section:strongApriori}--\ref{section:Properties of the nonlinearity}
to establish the main result of this article
in Theorem~\ref{theorem:existence_Burgers}
below.
%
%
%
%
 %
%
 %
%
 %
%
%
%
 %
%
%
%
\subsection{General setting}
\label{subsec:Main setting}
Throughout this article the following
setting is frequently used.
\begin{setting}
\label{setting:main}
%
%
\sloppy 
For every measurable space 
$ (\Omega_1, \F_1) $ and every 
measurable space   
$ (\Omega_2, \F_2) $ 
let $ \M( \F_1 , \F_2 ) $ be  
the set of all $ \F_1/\F_2 $-measurable functions from $ \Omega_1 $ to $ \Omega_2 $,
let
$ ( H, \langle \cdot, \cdot \rangle _H, \left \| \cdot \right\|_H ) $  
be a separable $ \R $-Hilbert space, 
let 
$ \H \subseteq H $
be a non-empty orthonormal basis of $ H $,  
let
$ \values \colon \H \to \R $
be a function which
satisfies
$ \sup_{h\in \H} \values_h < 0 $,
let
$ A \colon D( A ) \subseteq H \to H $ 
be the linear operator which satisfies
$ D(A) = \{
v \in H \colon \sum_{ h \in \H}  | \values_h \langle h, v \rangle_H  |^2 <  \infty
 \} $ 
and  
$ \forall \, v \in D(A) \colon 
A v = \sum_{ h \in \H} \values_h \langle h, v \rangle _H h $,
and
let
$ (H_r, \langle \cdot, \cdot \rangle_{H_r}, \left \| \cdot \right \|_{H_r} ) $, 
$ r\in \R $, be a family of interpolation spaces associated to $ -A $
(cf., e.g., \cite[Section~3.7]{SellYou2002}).
\end{setting}
Note that the assumption in Setting~\ref{setting:main} above that 
$ (H_r, \langle \cdot, \cdot \rangle_{H_r}, \left \| \cdot \right \|_{H_r} ) $, 
$ r \in \R $,  
is a family of interpolation spaces associated to 
$ - A $ ensures that for every
$ r \in [0, \infty) $ 
it holds that 
$ (H_r, \langle\cdot,\cdot\rangle_{H_r}, \left \| \cdot \right \|_{H_r} )
=
( D((-A)^r), 
\langle(-A)^r (\cdot),(-A)^r (\cdot )\rangle_H,
\left \| (-A)^r (\cdot) \right\|_H) $.
\section{Pathwise solvability for a class of random ODEs}
\label{section:strongApriori}

 In this section we analyze in Corollary~\ref{corollary:Existence} the solvability of a specific class of abstract random ODEs. The considered equations can  be thought of as spectral Galerkin discretizations in space of an underlying stochastic Burgers equation. Corollary~\ref{corollary:Existence} is based on an elementary and essentially well-known   
  pathwise existence and uniqueness result for random ODEs with non-globally Lipschitz continuous coefficient functions presented in Lemma~\ref{lemma:ProveExistence} 
(cf., e.g., Liu \& R\"ockner~\cite[Theorem~3.1.1]{LiuRoeckner2015Book}).
In addition,
we also recall
elementary results on  
measurability
in
Lemma~\ref{lemma:JointMeasurability}
(see, e.g., 
Aliprantis \& Border~\cite[Lemma~4.51]{AliprantisBorder2006})
and Lemma~\ref{lemma:Tonelli}
(cf., e.g., 
in Klenke~\cite[Theorem~14.16]{Klenke2008}).
For the sake of completeness we include the proof of Lemma~\ref{lemma:Tonelli}.
\begin{lemma}
	\label{lemma:JointMeasurability}
	%
	Let $ ( \Omega, \F) $ be a measurable space,
	let $ (X, d_X) $ be a separable metric space,
	let $ (Y, d_Y) $ be a metric space, 
	let
	$ f \colon X \times \Omega \to Y $
	be a function,
	assume for every 
	$ x \in X $
	that  
	$ \Omega \ni \omega \mapsto f( x, \omega ) \in Y $
	is
    $ \F / \B(Y) $-measurable,
    and assume for every 
    $ \omega \in \Omega $ 
    that  
	$ ( X \ni x \mapsto f( x, \omega) \in Y )\in \mathcal{C}( X, Y ) $.
	Then it holds that   
	$ f \colon X \times \Omega \to Y $
	is 
	$ ( \mathcal{B}(X) \otimes \F ) / \B(Y) $-measurable.
\end{lemma}
Note that for every topological space $ (X, \tau ) $ 
it holds that $ \mathcal B(X) $ is the smallest sigma-algebra 
on $ X $ which contains all elements of $ \tau $. 
\begin{lemma}
	\label{lemma:Tonelli}
	%
	Let
	$ (X, \left\| \cdot \right\|_X) $ 
	be an $ \R $-Banach space,
	let $ ( \Omega, \F ) $
	be a measurable space,
	let
	$ a \in \R $, $ b \in (a, \infty) $,
	let 
	$ f \colon [a,b] \times \Omega \to X $
	be a
	strongly 
	$ ( \B( [a,b] ) \otimes \F ) / (X, \left \| \cdot \right \|_X ) $-measurable function,
	assume for every  
	$ \omega \in \Omega $
	that
	$ \int_a^b \| f( s, \omega ) \|_X \, ds < \infty $,
	and let
	$ F \colon \Omega \to X $
	be the function which satisfies
	for every
	$ \omega \in \Omega $
	that
	$ F(\omega) = \int_a^b f(s, \omega) \, ds $.
	Then it holds that
	$ F $  is 
	strongly 
	$ \F / (X, \left \| \cdot \right \|_X ) $-measurable. 
\end{lemma}
\begin{proof}[Proof of Lemma~\ref{lemma:Tonelli}]
	Throughout this proof
	let
	$ \lambda 
	\colon
	\B( \R )
	\rightarrow [0,\infty] $
	be the Lebesgue-Borel
	measure on
	$ \R $, 
	let
	$ \mathcal{C} \subseteq ( \B([a,b]) \otimes \F ) $
	be the set given by
	\begin{equation} 
	\label{eq:DefineC}
	\mathcal{C} = \Big\{ C \in  ( \B([a,b]) \otimes \F )
	\colon \Big( \Omega \ni \omega \mapsto \int_a^b \1_C(s, \omega) \, ds \in \R \Big)
	\text{ is }  \F / \B( \R ) \text{-measurable} \Big\} 
	,
	\end{equation}
	for every set $ S $ 
	let $ \mathcal{P}( S ) $
	be the power set of $ S $,	
	for every set
	$ S $
	and every
	$ \mathcal{A} \subseteq \mathcal{P}( S ) $
	let
	$ \sigma_S ( \mathcal{A} ) $
	be the smallest
	sigma-algebra on $ S $
	which contains $ \mathcal{A} $, 
	and 
	for every set
	$ S $
	and every 
	$ \mathcal{A} \subseteq \mathcal{P}( S ) $
	let
	$ \delta_S( \mathcal{A} ) $
	be
	the smallest
	Dynkin system on $ S $
	which contains $ \mathcal{A} $.
	First, we intend to prove that 
	\begin{equation} 
	\label{eq:ProveDynkin}
	\mathcal{C} = \B( [a,b] ) \otimes \F 
	.
	\end{equation} 
	For this note that for every
	$ A \in \B( [a,b] ) $,
	$ B \in \F $,
	$ \omega \in \Omega $ 
	it holds that
	\begin{equation}
	\int_a^b
	\1_{ A \times B }( s, \omega ) 
	\,ds
	= 
	\int_a^b
	\1_A(s) \, \1_B(\omega) \, ds
	=
	\lambda (A)  \, \1_B(\omega)
	.
	\end{equation}
	This ensures that
	\begin{equation} 
	\label{eq:Dynkin1}
	\{ A \times B \colon A \in \B( [a,b] ), B \in \F \}
	\subseteq 
	\mathcal{C}  
	\qquad
	\text{and}
	\qquad
	( [a,b] \times \Omega ) \in \mathcal{C} 
	.
	\end{equation}
	The fact that 
	$ \{ A \times B  \colon A \in \B( [a,b] ), B \in \F \} $ is $ \cap $-stable
	and
	Dynkin's Lemma 
	therefore 
	prove that 
	\begin{equation}
	\begin{split}
	\label{eq:ProveForIndicator} 
	\B([a,b]) \otimes \F
	&
	=
	\sigma_{[a,b] \times \Omega} 
	(
	\{ A \times B  \colon A \in \B( [a,b] ), B \in \F \} 
	)
	\\ 
	&
	=
	\delta_{ [a,b] \times \Omega }
	(
	\{ A \times B  \colon A \in \B( [a,b] ), B \in \F \}
	)
	\\
	&
	\subseteq 
	\delta_{ [a,b] \times \Omega }
	( \mathcal{C} ) 
	\subseteq 
	\delta_{ [a,b] \times \Omega }
	( \B([a,b]) \otimes \F )
	= 
	\B( [a,b] ) \otimes \F 
	.
	\end{split}
	\end{equation}
	This
	shows that
	\begin{equation} 
	\label{eq:ProvedDynkin}
	\delta_{ [a,b] \times \Omega }(\mathcal{C}) 
	= 
	\B([a,b] ) \otimes \F  
	.
	\end{equation}
	Moreover, note that for every 
	$ C \in \mathcal{C} $,
	$ \omega \in \Omega $ 
	it holds that
	\begin{equation}
	\begin{split} 
	\int_a^b \1_{ ( [a,b] \times \Omega ) \backslash C}(s, \omega) \, ds
	&=
	\int_a^b ( \1_{ [a,b] \times \Omega }(s, \omega) 
	-
	\1_{ C}(s, \omega) ) \,ds
	\\
	&=
	\int_a^b \1_{ [a,b] \times \Omega }(s, \omega) \,ds
	-
	\int_a^b
	\1_{ C}(s, \omega) \,ds
	.
	\end{split} 
	\end{equation}
	This and~\eqref{eq:Dynkin1} imply that
	for
	every $ C \in \mathcal{C} $ 
	it holds that
	\begin{equation}
	\label{eq:Dynkin2}
	  ( ( [a,b] \times \Omega ) \backslash C )
	\in \mathcal{C}
	.
	\end{equation}
	Furthermore,
	note that the monotone convergence theorem
	proves that for all 
	pairwise disjoint sets
	$ C_n \in \mathcal{C} $,
	$ n \in \N $,
	it holds that
	\begin{equation}
	\begin{split}
	\int_a^b
	\1_{ \cup_{n \in \N } C_n } ( s, \omega) \, ds
	&=
	\int_a^b 
	\sum_{ n= 1 }^\infty
	\1_{ C_n }( s, \omega ) \, ds
	\\
	&
	=
	\int_a^b
	\lim_{ k \to \infty}
	\sum_{ n= 1 }^k 
	\1_{ C_n }( s, \omega ) \, ds
	= 
	\lim_{ k \to \infty}
	\int_a^b
	\sum_{ n= 1 }^k 
	\1_{ C_n }( s, \omega ) \, ds
	.
	\end{split}
	\end{equation}
	Therefore, we obtain that
	for all
	pairwise disjoint sets
	$ C_n \in \mathcal{C} $,
	$ n \in \N $,
	it holds that
	$
	\cup_{n \in \N } C_n \in \mathcal{C} 
	$.
	Combining this, \eqref{eq:Dynkin1}, and~\eqref{eq:Dynkin2}
	implies that
	$ \mathcal{C} $ 
	is a Dynkin system on 
	$ [a,b] \times \Omega $.
	Combining this
	and~\eqref{eq:ProvedDynkin} 
	establishes~\eqref{eq:ProveDynkin}.
	Next we intend to establish the statement of Lemma~\ref{lemma:Tonelli}.
	For this observe that the fact that
	$ f \colon [a,b] \times \Omega \to X $
	is strongly
	$ ( \B([a,b]) \otimes \F) / (X, \left \| \cdot \right \|_X) $-measurable and, e.g., 
	Pr\'ev\^ot \& R\"ockner~\cite[Lemma~A.1.4]{PrevotRoeckner2007}
	imply that there exist
	$ ( \B( [a,b] ) \otimes \F) / \B(X) $-measurable functions
	$ f_n \colon [a,b] \times \Omega \to X $,
	$ n \in \N $,
	which satisfy that
	\begin{enumerate}[(a)]
	\item \label{item:necessary 1} 
	it holds
	for every $ n \in \N $ that 
	$ f_n( [a, b]  \times \Omega) $
	is a finite set and
	\item \label{item:necessary 2} 
	it holds for every  
	$ \omega \in \Omega $
	that
	\begin{equation}
	\begin{split}
	\label{eq:Limsup}
	\limsup_{ n \to \infty }
	\Big\| 
	\int_a^b f_n(\omega, s) \, ds
	-
	\int_a^b f(\omega, s) \, ds
	\Big\|_X
	\leq
	\limsup_{ n \to \infty }
	\int_a^b
	\| f_n(\omega, s) - f(\omega,s) \|_X
	\, ds
	=
	0.
	\end{split}
	\end{equation}
	\end{enumerate} 
	Note that item~\eqref{item:necessary 1}  
	shows that
	for every
	$ n \in \N $,
	$ s \in [a,b] $,
	$ \omega \in \Omega $
	it holds that
	\begin{equation}
	f_n(s, \omega) 
	=
	\sum_{ x \in f_n([a,b] \times \Omega ) }
	x
	\1_{ (f_n)^{-1}(\{x\})}
	(s, \omega)
	.
	\end{equation}
	The fact that for every $ n \in \N $
	it holds that
	$ f_n( [a,b] \times \Omega) $
	is a finite set, the fact that
	for every $ n \in \N $,
	$ x \in X $
	it holds that
	$ (f_n)^{-1}(\{x\}) \in ( \B([a,b]) \otimes \F ) $,
	and~\eqref{eq:ProveDynkin}  
	hence prove that
	for every $ n \in \N $
	it holds that
	$ \Omega \ni \omega \mapsto
	\int_a^b f_n(s, \omega) \, ds \in X $  
	is strongly 
	$ \F / (X, \left \| \cdot \right \|_X ) $-measurable.
	%
	%
	%
	Combining 
	item~\eqref{item:necessary 1},
	item~\eqref{item:necessary 2},
	and, e.g.,
	Pr\'ev\^ot \& R\"ockner~\cite[item~(i)
	of Proposition~A.1.3]{PrevotRoeckner2007}
	therefore establishes that
	$ F $
	is strongly
	$ \F / (X, \left \| \cdot \right \|_X ) $-measurable.
    The proof of 
	Lemma~\ref{lemma:Tonelli}
	is thus completed.
\end{proof} 
\begin{lemma}
\label{lemma:ProveExistence}
	Let
	$ (H, \left\| \cdot \right\|_H, 
	\langle \cdot, \cdot \rangle_H ) $ 
	be a separable $ \R $-Hilbert space,
	let
	$ T \in (0, \infty) $, 
	$ s \in [0,T) $,
	let $ ( \Omega, \F, \P, ( \f_t )_{t \in [s,T]} ) $
	be a filtered probability space,
let 
$ \xi \colon \Omega \to H $
be an $ \f_s / \B(H) $-measurable function, 
let
$ f \colon [s,T] \times H \times \Omega \to H $
and
$ K \colon [s,T] \times (0,\infty) \times \Omega 
\to [0, \infty) $ 
	be functions,
	assume for every
	$ t \in [s,T] $,
	$ x \in H $
	that
	$ \Omega 
	\ni \omega  
	\mapsto 
	f(t, x, \omega) \in H $
	is $ \f_t / \B(H) $-measurable, 
	assume for every $ \omega \in \Omega $,
	$ r \in (0, \infty) $
	that
	$ ( [s,T] \times H \ni (t, x) \mapsto f(t, x, \omega) \in H ) \in \mathcal{C}( [s,T] \times H, H ) $, 
	$ ( [s,T] \ni t \mapsto K_t(r,\omega) \in [0, \infty) ) 
	\in \mathcal{C}( [s,T], [0, \infty) ) $,
	and
	$ \sup_{ t \in [s,T] }
	\sup_{x\in H,\|x\|_H\leq r} \|f(t, x, \omega)\|_H
	< \infty $,
	and assume for every 
	$ t \in [s,T] $,
	$ x, y \in H $,
	$ \omega \in \Omega $,
	$ r \in (0, \infty) $ 
	with 
	$ \max \{ \| x \|_H, \| y \|_H \} \leq r $
	that 
	%
	%
	%
	%
	%
	%
	%
	%
	$ 2 \langle x, f(t, x, \omega) \rangle_H 
	\leq K_t(1, \omega) (1 + \| x \|_H^2) $
	and
	\begin{equation} 
	\label{eq:AssMonotone} 
	2
	\langle x - y, f(t, x, \omega) - f(t, y, \omega) \rangle_H
	\leq
	K_t(r, \omega) \| x - y \|_H^2
	.
	\end{equation}
	%
	%
	%
%
	%
	%
	%
	Then 
	\begin{enumerate}[(i)]
		\item \label{item:Estistence}
	there exists a unique  
	function
	$ X \colon [s,T] \times \Omega \to H $
	which satisfies
	for every 
	$ t \in [s,T] $, $ \omega \in \Omega $
	that
	$ ( [s,T] \ni u \mapsto X_u(\omega) \in H ) 
	\in \mathcal{C}( [s,T], H ) $
	and
	\begin{equation}
	\label{eq:Result}
	X_t(\omega) = \xi(\omega) + \int_s^t f( u, X_{ u}(\omega), \omega ) \, d u  
	\end{equation} 
	and
	\item \label{item:StochasticProcess}
	it holds that 
	$ X \colon [s,T] \times \Omega \to H $
	is $ ( \f_t )_{ t \in [s,T] } $-adapted.
	\end{enumerate}
\end{lemma}
\begin{proof}[Proof of Lemma~\ref{lemma:ProveExistence}]
	Throughout this proof
	let
	$ X^n \colon [s,T] \times \Omega \to H $, $ n \in \N $,
	be the functions
	which satisfy for every
	$ n \in \N $,
	$ k \in \{ 0, 1, \ldots, n-1 \} $,
	$ t \in  (s+\frac{k(T-s)}{n},
	s + \frac{(k+1)(T-s)}{n} ] $, 
	$ \omega \in \Omega $
	that
	$ X_s^n( \omega ) = \xi( \omega ) $
	and
	\begin{equation}
	\begin{split} 
	\label{eq:Start}
	X_t^n( \omega )
	=
	X_{ s + (\nicefrac{k(T-s)}{n} ) }^n( \omega )
	+
	\int_{ s + ( \nicefrac{k(T-s)}{n} ) }^t
	f
	\big( 
	u, 
	X_{ s + ( \nicefrac{k(T-s)}{n} ) }^n(  \omega ), 
	\omega 
	\big) \, du 
	,
	\end{split} 
	\end{equation}
	%
	%
	let
	$ L \colon (0, \infty) \times \Omega \to [0, \infty) $
	be the function which satisfies for every
	$ r \in (0, \infty) $,
	$ \omega \in \Omega $
	that
	\begin{equation}
L_r(\omega) =
	\sup\nolimits_{ t \in [s,T] }
	\sup\nolimits_{ h \in H, \| h \|_H \leq r }
	\| f ( t, h, \omega ) \|_H
	,
	\end{equation}
	let
	$ \kappa \colon \N \times [s,T] \to [s, T] $ 
	be the function which
	satisfies for every
	$ n \in \N $, 
	$ k \in \{ 0, 1, \ldots, n-1 \} $,
	%
%
	%
	$ t \in ( s + \frac{k(T-s)}{n}, 
	s + \frac{(k+1)(T-s)}{n} ] $  
	that
$ \kappa(n, s) = s $
and
\begin{equation}
\kappa(n,t) =
s
+
\tfrac{ k ( T - s ) }{ n }
,
\end{equation}
	let
	$ \mathcal{K} \colon [s,T] \times (0, \infty) \times \Omega \to [0, \infty) $
	and
	$ \alpha \colon [s,T] \times (0, \infty) \times \Omega \to [0, \infty) $
	be the functions which satisfy for every
	$ t \in [s,T] $,
	$ r \in (0, \infty) $, 
	$ \omega \in \Omega $
	that
\begin{equation} 
\mathcal{K}_t(r, \omega) 
=
\max \{ 
K_t(r, \omega)
,
L_r(\omega) 
\} 
\qquad 
%
\text{and} 
\qquad
\alpha_t(r, \omega)
=
\int_s^t \mathcal{K}_u(r, \omega) \, du,
\end{equation}
let
$ \tau^n \colon (0, \infty) \times \Omega \to [0,T] $, 
$ n \in \N $,
be the functions which satisfy for every
$ n \in \N $,
$ r \in (0,\infty) $,
$ \omega \in \Omega $
that
\begin{equation} 
 \tau_r^n(\omega) = \inf ( \{T\} \cup 
\{ t \in [0,T] \colon  \| X_t^n(\omega) \|_H \geq r \} ),
\end{equation}
and
	let
	$ p^n \colon [s,T] \times \Omega \to H $, $ n \in \N $,
	be the functions 
	which satisfy for every
	$ n \in \N $,
	$ t \in [s,T] $,
	$ \omega \in \Omega $
	that
	\begin{equation}
	p^n_t(\omega) =
	X^n_{ \kappa(n,t) }(\omega) - X_t^n(\omega)
	.
	\end{equation}
	First, we 
	establish 
	item~\eqref{item:Estistence}. 
	For this note
	that
	for every
	$ r \in (0, \infty) $,
	$ n \in \N $,
	$ \omega \in \Omega $,
	$ t \in [s, \tau_r^n(\omega)] $
	it holds that
	\begin{equation}
	\begin{split}
	\label{eq:uniformEstimate}
	&
	\|p_t^n (\omega) \|_H 
	\leq
	\int_{ \kappa(n,t) }^t
	\| f ( u, X^n_{ \kappa(n, u) } (\omega), \omega ) \|_H \, du
	\\
	&
	\leq
	\int_{ \kappa(n,t) }^t 
	L_r(\omega)
	\, du
	\leq 
	( t - \kappa(n,t) )
	L_r(\omega)
	\leq
	\tfrac{ (T-s) }{ n } L_r( \omega ) 
	<
	\infty
	.
	\end{split}
	\end{equation}
	This ensures
	for every  
	$ r \in (0, \infty) $,
	$ \omega \in \Omega $ 
	that 
	\begin{equation}
	\label{eq:Convergence}
	\limsup\nolimits_{ n \to \infty }
	\sup\nolimits_{ t \in [s,T] }
	\1_{ [s, \tau_r^n(\omega) ]}(t)
	\,
	\| p^n_t(\omega) \|_H
	=
	0
	.
	\end{equation}
	The dominated convergence theorem hence
	shows that for every
	$ r \in (0, \infty) $,
	$ \omega \in \Omega $ it holds that
	\begin{equation}
	\label{eq:IMportantLimit}
	\limsup\nolimits_{ n \to \infty }
	\int_s^T
	\1_{ [s, \tau_r^n(\omega) ] }
	(u)
	\,
	\| p^n_u(\omega) \|_H
	\mathcal{K}_u(r, \omega) \, du
	=
	0
	.
	\end{equation}
	In the next step we observe that
	for every
	$ t \in [s,T] $,
	$ n \in \N $,
	$ \omega \in \Omega $ it holds that
	\begin{equation}
	\label{eq:Needed}
	X_t^n(\omega)=
	\xi(\omega) 
	+
	\int_s^t 
	f( u, X^n_{\kappa(n,u)}(\omega), \omega )
	\, 
	du
	.
	\end{equation}
	Furthermore, note that the fact that
	for every
	$ \omega \in \Omega $,
	$ x \in H $
	it holds that
	$ ( [s,T] \ni u \mapsto f(u, x, \omega) \in H ) \in \mathcal{C}( [s,T], H ) $
	and, e.g., \cite[Corollary~2.7]{JentzenLindnerPusnik2017a} 
	(with
	$ V = H $,
	$ W = \R $,
	$ a = s $,
	$ b = T $,
	$ \phi = ( [s,T] \times H \ni (t, x) \mapsto \| x \|_H^2 e^{ - \alpha_t(1,\omega)} \in \R ) $,
	$ f = ( [s,T] \ni t \mapsto 
	f(t, X_{ \kappa(n, t) }^n(\omega), \omega) \in H ) $,
	$ F = ( [s,T] \ni t \mapsto  X_t^n(\omega) \in H ) $
	for 
	$ n \in \N $, 
	$ \omega \in \Omega $
	in the notation of~\cite[Corollary~2.7]{JentzenLindnerPusnik2017a}) 
	prove that for every
	$ r \in (0, \infty) $,
	$ n \in \N $,
	$ \omega \in \Omega $,
	$ t \in [s, \tau_r^n(\omega)] $ 
	it holds that
	\begin{equation}
	\begin{split}
	&
	\| X_t^n(\omega) \|_H^2 
	e^{ - \alpha_t(1, \omega) }
	\\
	&= 
	\| \xi (\omega)\|_H^2 
	+ 
	\int_s^t
	e^{ - \alpha_u(1, \omega) } 
	\big[
	2 
	\langle 
	X_u^n(\omega), 
	f(u, X_{ \kappa(n,u) }^n(\omega), \omega) 
	\rangle_H 
	-
	\mathcal{K}_u(1, \omega) 
	\| X_{ u}^n(\omega) \|_H^2
	\big]
	\, du 
	\\
	& 
	=
	\| \xi (\omega)\|_H^2 
	+
	\int_s^t
	e^{ - \alpha_u(1, \omega) } 
	\big[
	2
	\langle 
	X_{ \kappa(n,u) }^n(\omega), 
	f(u, X_{ \kappa(n,u) }^n(\omega), \omega) 
	\rangle_H  
	\\
	&
	\quad
	-
	2 
	\langle 
	p_u^n( \omega ), 
	f(u, X_{ \kappa(n,u) }^n(\omega), \omega) 
	\rangle_H 
	-
	\mathcal{K}_u(1, \omega) 
	\| X_{ u}^n(\omega) \|_H^2
	\big]
	\, du 
	. 
	\end{split}
	\end{equation}
	\sloppy 
	Combining this, 
	the assumption that
	for every
	$ t \in [s,T] $,
	$ x \in H $,  
	$ \omega \in \Omega $
	it holds that
	$ 2 \langle x, f(t, x, \omega) \rangle_H 
	 \leq K_t(1, \omega) (1 + \| x \|_H^2) $,
	the Cauchy-Schwarz inequality,
	and~\eqref{eq:Needed} 
	implies that for every
	$ r \in (0, \infty) $,
	$ n \in \N $,
	$ \omega \in \Omega $,
	$ t \in [s, \tau_r^n(\omega)] $
	it holds that
	\begin{equation}
	\begin{split}
	&
	\| X_t^n(\omega) \|_H^2 
	e^{ - \alpha_t(1, \omega) }
\leq 
\| \xi( \omega ) \|_H^2 
\\
&
\quad
+
\int_s^t
e^{ - \alpha_u(1, \omega) } 
\big[
\mathcal{K}_u(1, \omega) ( 1 + \| X_{ \kappa(n,u) }^n (\omega) \|_H^2 ) 
+
2
\| p_u^n(\omega) \|_H 
\| f(u, X_{ \kappa(n,u) }^n(\omega), \omega) \|_H
\big]
\, du 
\\
&
\leq 
\| \xi( \omega ) \|_H^2 
+
\int_s^t
e^{ - \alpha_u(1, \omega) } 
\big[
\mathcal{K}_u(1, \omega) ( 1 + \| X_{ \kappa(n,u) }^n (\omega) \|_H^2 ) 
+
2
\mathcal{K}_u(r, \omega)
\| p_u^n(\omega) \|_H 
\big]
\, du
. 
	\end{split}
	\end{equation}
	The fact that
	for every
	$ r \in (0, \infty) $, 
	$ n \in \N $,
	$ \omega \in \Omega $,
	$ u \in [s, T] $
	it holds that
	$ \1_{  [s, \tau_r^n(\omega)]  }(u)
	\,
	\| X_u^n(\omega) \|_H^2 
	\leq r^2 $,
	Fatou's Lemma,
	and~\eqref{eq:IMportantLimit}
	hence assure that
	for every
	$ r \in (0, \infty) $,
	$ \omega \in \Omega $,
	$ t \in [s,T] $ 
	it holds that
	\begin{equation}
	\begin{split}
	&
	\limsup\nolimits_{n \to \infty}
	\sup\nolimits_{ u \in [s,t] }
	\big( 
	\1_{  [s, \tau_r^n(\omega)]  }(u)
	\,
	\| X_u^n(\omega) \|_H^2 
	\big)
	e^{ - \alpha_t(1, \omega) }
	\\
	&
	\leq 
	\| \xi( \omega ) \|_H^2 
	+
	\int_s^t
	e^{ - \alpha_u(1, \omega) } 
	\mathcal{K}_u(1, \omega)
	\big[ 
	1 
	+  
	\limsup\nolimits_{n \to \infty}
	\sup\nolimits_{ v \in [s,u] }
	\big( 
	\1_{  [s, \tau_r^n(\omega)]  }(v)
	\,
	\| X_v^n ( \omega ) \|_H^2  
	\big) 
	\big]
	\, du
	\\
	&
	\quad
	+
	2
	\limsup\nolimits_{n \to \infty} 
	\int_s^t 
	\1_{  [s, \tau_r^n(\omega)]  }( u )
	\,
	\mathcal{K}_u(r, \omega)
	\| p_u^n(\omega) \|_H   
	\, du
	\\
	&
	=
	\| \xi( \omega ) \|_H^2 
	+
	\int_{ s}^{ t }
	e^{ - \alpha_u(1, \omega) }  
	\mathcal{K}_u(1, \omega) 
	\, du
	\\
	&
	\quad
	+
	\int_s^t   
	\mathcal{K}_u(1, \omega)
	\limsup\nolimits_{n \to \infty}
	\sup\nolimits_{ v \in [s,u] }
	\big( 
	\1_{ [s, \tau_r^n(\omega)]  }(v)
	\,
	\| X_v^n ( \omega ) \|_H^2   
	\big) 
		e^{ - \alpha_u(1, \omega) } 
	\, du
	. 
	\end{split}
	\end{equation}
	Gronwall's lemma  
	therefore demonstrates 
	that
	for every
	$ r \in (0, \infty) $, 
	$ \omega \in \Omega $,
	$ t \in [s, T] $
	it holds that
	\begin{equation}
	\begin{split}
	&
	\limsup\nolimits_{n \to \infty}
	\sup\nolimits_{ u \in [s,t] }
	\big( 
	\1_{ [s, \tau_r^n(\omega)] }(u)
	\,
	\| X_u^n(\omega) \|_H^2 
	\big)
	e^{ - \alpha_t(1, \omega) }
	\\
	&
	\leq 
	\Big[ 
	\| \xi( \omega ) \|_H^2
	+
	\int_{ s}^{ t }
	e^{ - \alpha_u(1, \omega) }  
	\mathcal{K}_u(1, \omega) 
	\, du
	\Big]
	\exp
	\Big(
	\int_s^t  
	\mathcal{K}_u(1, \omega)
	\, du
	\Big)
	.
	\end{split} 
	\end{equation} 
	The change of variables formula
	hence establishes 
	that
	for every
	$ r \in (0, \infty) $, 
	$ \omega \in \Omega $,
	$ t \in [s, T] $
	it holds that
	\begin{equation}
	\begin{split}
	&
	\limsup\nolimits_{n \to \infty}
	\sup\nolimits_{ u \in [s,t] }
	\big( 
	\1_{  [s, \tau_r^n(\omega)]  }(u)
	\,
	\| X_u^n(\omega) \|_H^2 
	\big)
	\\
	&
	\leq
	e^{ 2 \alpha_t(1, \omega) }
	\Big[ 
	\| \xi( \omega ) \|_H^2
	+
	\int_{ \alpha_s(1, \omega) }^{ \alpha_t(1, \omega) } 
	e^{ - v } 
	\, dv
	\Big]
	\leq
	e^{ 2 \alpha_t(1, \omega) }
	[ 
	\| \xi( \omega ) \|_H^2
	+
	1
	]
	. 
	\end{split}
	\end{equation}
	This shows for every
		$ r \in (0, \infty) $, 
	$ \omega \in \Omega $ 
	that
	\begin{equation}
	\begin{split}
	\label{eq:LimSupEstimate}
	&
	\limsup\nolimits_{n \to \infty}
	\sup\nolimits_{ u \in [s, T] }
	\big( 
	\1_{  [s, \tau_r^n(\omega)]  }(u)
	\,
	\| X_u^n(\omega) \|_H^2 
	\big)
	\leq
	e^{ 2 \alpha_T(1, \omega) }
	[ 
	\| \xi( \omega ) \|_H^2
	+
	1
	]
	. 
	\end{split}
	\end{equation}
	Therefore, we obtain that 
	there exist functions
	$ N \colon \Omega \to \N $
	and
	$ M \colon \Omega \to (0,\infty) $
	which satisfy that 
	for every
	$ \omega \in \Omega $, 
	$ n \in [ N(\omega), \infty ) \cap \N $
	it holds that
	$ M(\omega) =
	1
	+
	e^{ \alpha_T(1, \omega) }
	\sqrt{ 
	\| \xi( \omega ) \|_H^2
	+
	1 
} $
and 
	\begin{equation}
	\begin{split}
	\label{eq:Strict}
	\sup\nolimits_{ u \in [s,T] }
	\big( 
	\1_{  [s, \tau_{ M(\omega)  }^n(\omega)]  }(u)
	\,
	\| X_u^n(\omega) \|_H^2 
	\big)
	\leq
	 [M( \omega ) - 1]^2   
	 +
	 1
	<
	[ M(\omega) ]^2
	.
	\end{split}
	\end{equation}
	%
	%
	Note that~\eqref{eq:Strict} shows that
	for every
	$ \omega \in \Omega $,
	$ n \in [ N(\omega), \infty ) \cap \N $
	it holds that
	$ \tau_{ M(\omega) }^n ( \omega ) = T $
	and
	\begin{equation}
	\begin{split}
	\label{eq:Crucial}
	\sup\nolimits_{ u \in [s,T] } 
	\| X_u^n(\omega) \|_H 
	\leq
	\sqrt{ [ M(\omega) - 1 ]^2 + 1 }
	\leq
	M(\omega)
	.
	\end{split}
	\end{equation}
	Furthermore,
	note that
	for every
	$ t \in [s,T] $, 
	$ r \in (0, \infty) $,
	$ \omega \in \Omega $,
	$ m, n \in \N $
	it holds that
	\begin{equation}
	\begin{split}
	&
	\| X_t^n( \omega ) - X_t^m( \omega ) \|_H^2 
	e^{ - 2 \alpha_t(r, \omega) } 
	\\
	&
	=
	2
	\int_s^t  
	\big[
	\langle X_u^n(\omega) - X_u^m(\omega),
	f( u, X_u^n(\omega) + p_u^n(\omega), \omega )
	-
	f( u, X_u^m(\omega) + p_u^m(\omega), \omega )
	\rangle_H
	\\
	&
	-
    \mathcal{K}_u(r, \omega) 
	\| X_u^n(\omega) - X_u^m(\omega) \|_H^2 
	\big]
	e^{ - 2 \alpha_u(r, \omega) } 
	\, du
	\\
	&
	=
	2
	\int_s^t  
	\big[
	\langle
	p_u^m(\omega) 
	- 
	p_u^n(\omega),
	f( u, X_u^n(\omega) + p_u^n(\omega), \omega )
	-
	f( u, X_u^m(\omega) + p_u^m(\omega), \omega )
	\rangle_H
	\\
	&
	+
	\langle  X_u^n(\omega) + p_u^n(\omega)  
	- 
	X_u^m(\omega) - p_u^m(\omega)   
	,
	f( u, X_u^n(\omega) + p_u^n(\omega), \omega )
	-
	f( u, X_u^m(\omega) + p_u^m(\omega), \omega )
	\rangle_H
	\\
	&
	-
    \mathcal{K}_u(r, \omega) 
	\| X_u^n(\omega) - X_u^m(\omega) \|_H^2 
	\big]
	e^{ - 2 \alpha_u(r, \omega)  } 
	\, du
	.
	\end{split}
	\end{equation}
	Combining~\eqref{eq:AssMonotone},
	\eqref{eq:Crucial}, 
	and the Cauchy-Schwarz inequality
	hence
	ensures that for every
	$ t \in [s, T] $,
	$ \omega \in \Omega $,
	$ m, n \in [ N(\omega), \infty) \cap \N $ 
	it holds that
	\begin{equation}
	\begin{split}
	&
	\| X_t^n(\omega) - X_t^m(\omega) \|_H^2 
	e^{ - 2 \alpha_t( M(\omega), \omega)  }
	\leq
	\int_s^t 
	e^{ - 2 \alpha_u( M(\omega), \omega)  }
	\\
	&
	\quad
	\cdot
	\big[
	2
	\langle 
	p_u^m(\omega) - p_u^n(\omega),
	f( u, X_u^n(\omega) + p_u^n(\omega), \omega )
	-
	f( u, X_u^m(\omega) + p_u^m(\omega), \omega )
	\rangle_H 
	\\ 
	&
	\quad
	+
	\mathcal{K}_u( M(\omega), \omega )
	\| 
	( X_u^n(\omega) - X_u^m(\omega) ) 
	+
	( p_u^n(\omega) - p_u^m(\omega) )
	\|_H^2
	\\
	&
	\quad 
	-
	2 \mathcal{K}_u( M(\omega), \omega) 
	\| X_u^n(\omega) - X_u^m(\omega) \|_H^2
	\big]
	\, du
	\\
	&
	\leq
	\int_s^t 
	e^{ - 2 \alpha_u( M(\omega) , \omega)   }
	%
	\big[
	2
	\|
	p_u^m(\omega) - p_u^n(\omega)
	\|_H
	\|
	f( u, X_u^n(\omega) + p_u^n(\omega), \omega )
	-
	f( u, X_u^m(\omega) + p_u^m(\omega), \omega )
	\|_H 
	\\
	&
	\quad
	+
	2
	\mathcal{K}_u( M(\omega), \omega )
	\big(
	\| X_u^n(\omega) - X_u^m(\omega) \|_H^2 
	+
	\| p_u^n(\omega) - p_u^m(\omega) \|_H^2
	\big)
	\\
	&
	\quad
	-
	2 \mathcal{K}_u( M(\omega), \omega) 
	\| X_u^n(\omega) - X_u^m(\omega) \|_H^2
	\big]
	\, d u
	.
	\end{split} 
	\end{equation} 
	This implies
	for every
	$ t \in [s, T] $, 
	$ \omega \in \Omega $,
	$ m, n \in [N(\omega), \infty) \cap \N $
	that
	\begin{equation}
	\begin{split}
	\label{eq:BigEstimate}
	&
	\| X_t^n(\omega) - X_t^m(\omega) \|_H^2 
	e^{ - 2 \alpha_t( M(\omega), \omega)  }
	\\
	&
	\leq
	2
	\int_s^t 
	e^{ - 2 \alpha_u( M(\omega), \omega) }
	\big[
	2
	\|
	p_u^m(\omega) - p_u^n(\omega)
	\|_H 
	L_{ M(\omega)  }(\omega)
	%
	%
	+ 
	\mathcal{K}_u(  M(\omega),\omega)
	\| p_u^n(\omega) - p_u^m(\omega) \|_H^2 
	\big]
	\, du
	\\
	&
	\leq
	2
	\int_s^t 
	e^{ - 2 \alpha_u(  M(\omega), \omega)  }
	\mathcal{K}_u( M(\omega), \omega) 
	\big[
	2 \| p_u^m(\omega) - p_u^n(\omega) \|_H
	+
	\| p_u^m(\omega) - p_u^n(\omega) \|_H^2 
	\big]
	\, du
	.
	\end{split}
	\end{equation}
	Moreover, 
	note that~\eqref{eq:Crucial}
	establishes for every
	$ u \in [s, T ] $, 
	$ \omega \in \Omega $,
	$ m, n \in [ N(\omega), \infty) \cap \N $
	that
	\begin{equation}
	\| p_u^m(\omega) - p_u^n(\omega) \|_H^2 
	\leq
	4 M (\omega) ( \| p_u^m(\omega) \|_H 
	+
	\| p_u^n(\omega) \|_H )
	.
	\end{equation}
	Combining
	this and~\eqref{eq:BigEstimate}
	shows that for every
	$ t \in [s, T ] $, 
	$ \omega \in \Omega $,
		$ m, n \in [ N(\omega), \infty) \cap \N $
	it holds that
	\begin{equation}
	\begin{split} 
	\label{eq:CauchyEstimate}
	&
	\| X_t^n(\omega) - X_t^m(\omega) \|_H^2 
	e^{ - 2 \alpha_t( M(\omega), \omega) }
	\\
	&
	\leq
	4 ( 1 + 2 M(\omega) )
	\Big(
	\int_s^T
	\mathcal{K}_u( M(\omega), \omega) 
	( 
	\| p^n_u(\omega) \|_H
	+
	\| p^m_u(\omega) \|_H
	)
	\, du
	\Big)
	.
	\end{split}
	\end{equation}
	In addition, observe that
	the fact that
	for every $ \omega \in \Omega $, 
	$ n \in [ N(\omega), \infty) \cap \N $
	it holds that
	$ \tau_{ M(\omega) }^n(\omega) = T $,
	\eqref{eq:IMportantLimit},
	and~\eqref{eq:Crucial} 
	assure that for every $ \omega \in \Omega $ it holds that
	\begin{equation}
	\begin{split}
	\label{eq:IMportantLimit2}
	&
	\limsup\nolimits_{ n \to \infty }
	\int_s^T
	\| p^n_u(\omega) \|_H
	\,
	\mathcal{K}_u( M(\omega), \omega) \, du
	\\
	&
	=
	\limsup\nolimits_{ n \to \infty }
	\int_s^T
	\1_{ [s, \tau_{ M(\omega) }^n(\omega) ] }
	(u)
	\,
	\| p^n_u(\omega) \|_H
	\,
	\mathcal{K}_u( M(\omega), \omega) \, du
	=
	0
	.
	\end{split} 
	\end{equation}
	This and~\eqref{eq:CauchyEstimate}
	demonstrate that for every
	$ \omega \in \Omega $
	it holds
	that
	$ ( [s,T] \ni t \mapsto X_t^n(\omega)\in H ) \in \mathcal{C}( [s,T], H ) $,
	$ n \in \N $,
	is a Cauchy sequence.
	The fact that the space 
	$ \mathcal{C}( [s,T], H ) $
	with the supremum norm
	is complete
	hence ensures that
	there exists
	a function
	$ X \colon [s,T] \times \Omega \to H $
	which satisfies for every 
	$ \omega \in \Omega $
	that  
	$ ( [s,T] \ni t \mapsto X_t(\omega) \in H ) 
	\in \mathcal{C}( [s,T], H ) $
	and
	\begin{equation}
	\label{eq:Existence}
	\limsup\nolimits_{ n \to \infty }
	\sup\nolimits_{ t \in [s,T] }\| X^n_t(\omega) - X_t(\omega) \|_H 
	=
	0
	.
	\end{equation} 
	\sloppy 
	Observe that
	the assumption that
	for every
	$ \omega \in \Omega $
	it holds that 
	$ ( [s,T] \times H \ni (t,x) \mapsto f(t,x,\omega) \in H ) \in \mathcal{C}( [s,T] \times H, H ) $,
the assumption that
for every
$ r \in (0, \infty) $,
$ \omega \in \Omega $
it holds that
$ \sup_{ t \in [s,T] }
\sup_{ x \in H, \| x \|_H \leq r } \| f ( t, x, \omega ) \|_H
< \infty $,
%
%
\eqref{eq:Crucial},
\eqref{eq:Existence},
	and
	the
	dominated convergence theorem
	prove that for every
	$ t \in [s,T] $,
	$ \omega \in \Omega $ 
	it holds that
	\begin{equation} 
	\begin{split} 
	\label{eq:Part1}
	&
	\limsup\nolimits_{n \to \infty}
	\Big\|
	\int_s^t 
	f( u, X^n_u(\omega), \omega ) \, du 
	-
	\int_s^t 
	f(u, X_u(\omega), \omega) \, du
	\Big\|_H
	\\
	&
	\leq
	\limsup\nolimits_{n \to \infty}
	\int_s^t 
	\| 
	f(u, X^n_u(\omega), \omega ) 
	-  
	f(u, X_u(\omega), \omega) 
	\|_H
	\, du
	\\
	&
	=
	\int_s^t
	\limsup\nolimits_{n \to \infty} 
	\| 
	f(u, X^n_u(\omega), \omega ) 
	-  
	f(u, X_u(\omega), \omega) 
	\|_H 
	\, du
	=
	0.
	\end{split} 
	\end{equation}
	Moreover,
	observe that~\eqref{eq:Existence}
	assures that for every
	$ \omega \in \Omega $
	it holds that 
	the sequence 
	$ X^n(\omega) \in \mathcal{C}( [s,T], H ) $,
	$ n \in \N $,
	is uniformly equicontinuous.
	This implies for every
	$ \omega \in \Omega $ that
	\begin{equation}
	\begin{split}
	\limsup\nolimits_{ n \to \infty }
	\sup\nolimits_{ u \in [s,T] }\| X^n_{\kappa(n, u)}(\omega) - X_u^n(\omega) \|_H
	=
	0.
	\end{split}
	\end{equation}
	\sloppy 
	The assumption that
	for every 
	$ \omega \in \Omega $
	it holds that
	$ ( [s,T] \times H \ni (t,x) \mapsto f(t,x,\omega) \in H ) \in \mathcal{C}( [s,T] \times H, H ) $,
%
the assumption that
for every
$ r \in (0, \infty) $,
$ \omega \in \Omega $
it holds that
$ \sup_{ t \in [s,T] }
\sup_{ x \in H, \| x \|_H \leq r  }
\| f ( t, x, \omega ) \|_H
< \infty $,
	\eqref{eq:Crucial}, 
	and the dominated convergence theorem therefore show that
	for every
	$ t \in [s,T] $,
	$ \omega \in \Omega $ 
	it holds that
	\begin{equation}
	\begin{split}
	\label{eq:Part2}
	&
	\limsup\nolimits_{n \to \infty}
	\Big\|
	\int_s^t 
	f(  u, X^n_{ \kappa(n, u) }(\omega), \omega ) \, du
	-
	\int_s^t 
	f(u, X^n_u(\omega), \omega) \, du
	\Big\|_H
	\\
	&
	\leq
	\limsup\nolimits_{n \to \infty} 
	\int_s^t  
	\|
	f( u, X^n_{ \kappa(n,u) }(\omega), \omega )   
	- 
	f( u, X^n_u(\omega), \omega) 
	\|_H 
	\, du 
	\\
	&
	\leq	
	\int_s^t 
	\limsup\nolimits_{n \to \infty} 
	\|
	f( u, X^n_{ \kappa(n,u) }(\omega), \omega )   
	- 
	f( u, X^n_u(\omega), \omega) 
	\|_H 
	\, du 
	=
	0.
	\end{split}
	\end{equation}
	The triangle inequality and~\eqref{eq:Part1}
	hence
	ensure
	that
	for every
	$ t \in [s, T ] $,
	$ \omega \in \Omega $ 
	it holds that
	\begin{equation}
	\begin{split} 
	& 
	\limsup\nolimits_{n \to \infty}
	\Big\|
	\int_s^t 
	f(  u, X^n_{ \kappa(n, u) }(\omega), \omega ) \, du 
	-
	\int_s^t 
	f( u, X_u(\omega), \omega) \, du
	\Big\|_H
	=
	0
	.
	\end{split} 
	\end{equation}
	Combining this, \eqref{eq:Needed}, 
	and~\eqref{eq:Existence}
	implies
	that for every
	$ t \in [s,T] $, 
	$ \omega \in \Omega $ 
	it holds that
	\begin{equation}
	\label{eq:ProvedExistence}
	X_t(\omega)
	=
	\xi(\omega)
	+
	\int_s^t f( u, X_u(\omega), \omega) \, du
	.
	\end{equation}
	%
	%
%
	%
	%
	%
	Next note that~\eqref{eq:AssMonotone}  
	proves that for every function
	$ \mathbf{X} \colon [s,T] \times \Omega \to H $
	with
	$ \forall \, \omega\in\Omega\colon ([s,T]\ni t\mapsto \mathbf X_t(\omega)\in H)\in \mathcal{C}([s,T],H) $
	and
%
%
	$ \forall \,  t \in [s, T] $,
	$ \omega \in \Omega \colon \mathbf{X}_t(\omega)
	=
	\xi(\omega)
	+
	\int_s^t f( u, \mathbf{X}_{u }(\omega), \omega) \, du
	$
	and every 
    $ t \in [s, T] $,
	$ \omega \in \Omega $,
	$ r \in (\sup_{u \in [s,T] } 
	\max \{ \| X_u(\omega) \|_H, \| \mathbf{X}_u(\omega) \|_H \}, \infty) $
    it holds that
	\begin{equation}
	\begin{split}
	&
	e^{ - 2 \alpha_t(r, \omega) } 
	\| X_t( \omega ) - \mathbf{X}_t( \omega ) \|_H^2  
	\\ 
	&
	=
	2 
	\int_s^t   e^{ - 2 \alpha_u(r, \omega) } 
	\big[
	\langle X_u(\omega) - \mathbf{X}_u(\omega),
	f( u, X_u(\omega), \omega )
	-
	f( u, \mathbf{X}_u(\omega), \omega )
	\rangle_H
	\\
	&
	\quad
	-
	\mathcal{K}_u(r, \omega) 
	\| X_u(\omega) - \mathbf{X}_u(\omega) \|_H^2 
	\big] 
	\, du
	\\
	&
	\leq
	\int_s^t 
	\mathcal{K}_u(r, \omega)
	e^{ - 2 \alpha_u(r, \omega) }
	\| X_u(\omega) - \mathbf{X}_u(\omega) \|_H^2
	\, du
	\\
	&
	\leq
	\sup\nolimits_{ u \in [s,T] }
	| \mathcal{K}_u ( r, \omega ) |
	\int_s^t 
	e^{ - 2 \alpha_u(r, \omega) }
	\| X_u(\omega) - \mathbf{X}_u(\omega) \|_H^2
	\, du
	< \infty
	.
	\end{split}
	\end{equation}
	Gronwall's lemma hence implies that
	for every function
	$ \mathbf{X} \colon [s,T] \times \Omega \to H $
	with
	$ \forall \, \omega\in\Omega\colon ([s,T]\ni t\mapsto \mathbf X_t(\omega)\in H)\in \mathcal{C}([s,T],H) $
	and
%
	$ \forall \, t \in [s,T] $,
	$ \omega \in \Omega \colon
	\mathbf{X}_t(\omega)
	=
	\xi(\omega)
	+
	\int_s^t f( u, \mathbf{X}_u(\omega), \omega) \, du
	$
	and every 
	$ t \in [s,T] $, $ \omega \in \Omega $ 
	it holds
    that
    \begin{equation} 
    X_t(\omega) = \mathbf{X}_t(\omega) 
    .
    \end{equation}
	Combining this
	and~\eqref{eq:ProvedExistence} establishes item~\eqref{item:Estistence}.
	In addition, 
	note that
	the assumption that
	for every 
	$ \omega \in \Omega $
	it holds that
	$ ( [s,T] \times H \ni ( t, x ) \mapsto 
	f( t, x, \omega) \in H ) \in
	\mathcal{C}( [s,T] \times H, H ) $,
	the assumption that
	for every 
	$ t \in [s,T] $, 
	$ u \in [s,t] $, 
	$ x \in H $
	it holds that  
	$ \Omega \ni \omega \mapsto f(u, x, \omega) \in H $
	is $ \f_t / \B(H) $-measurable,
	and 
	Lemma~\ref{lemma:JointMeasurability}
	(with
	$ ( \Omega, \F ) = ( \Omega, \f_t ) $,
	$ X = [s,T] \times H $,
	$ d_X = ( [s,T] \times H \times [s,T] \times H  \ni (t_1, x_1, t_2, x_2) \mapsto 
	| t_1 - t_2 | + \| x_1 - x_2 \|_H \in [0, \infty) ) $,
	$ Y = H $,
	$ d_Y = ( H \times H \ni (x_1, x_2) \mapsto
	\| x_1 - x_2 \|_H \in [0, \infty) )$,
	$ f = ( [s,t] \times H \times \Omega 
	\ni ( u, x, \omega) \mapsto f(u, x, \omega) \in H ) $
	for
	$ t \in [s,T] $
	in the notation of Lemma~\ref{lemma:JointMeasurability})
	show that
	for every
	$ t \in [s,T] $ 
	it holds that
	\begin{equation} 
	[s,t] \times H \times \Omega \ni ( u, x, \omega) \mapsto f(u, x, \omega) \in H 
	\end{equation} 
	is $ ( \B( [s, t] ) \otimes \B(H) \otimes \f_t ) / \B(H) $-measurable.
	The fact that
	for every $ t \in [s,T] $
	and every $ \f_t / \B(H) $-measurable 
	function $ \zeta \colon \Omega \to H $ 
	it holds that
	$ [s,t] \times \Omega \ni ( u, \omega ) \mapsto ( u, \zeta(\omega), \omega) \in [s,t] \times H \times \Omega $ 
	is
	$ ( \B( [s,t] ) \otimes \f_t ) / ( \B( [s,t] ) \otimes \B(H) \otimes \f_t ) $-measurable
	hence assures that 
	for every
	$ t \in [s,T] $ 
	and every $ \f_t / \B(H) $-measurable 
	function $ \zeta \colon \Omega \to H $ 
	it holds that
	\begin{equation} 
	 [s,t] \times \Omega \ni (u, \omega) \mapsto f( u, \zeta(\omega), \omega) \in H 
	\end{equation}  
	is
	$ ( \B( [s,t] ) \otimes \f_t ) / \B(H) $-measurable. 
			The assumption that
			$ \xi \colon \Omega \to H $
			is $ \f_s / \B(H) $-measurable, 
			\eqref{eq:Start},
			and  
			Lemma~\ref{lemma:Tonelli}
%
%
			(with
			$ X = H $,
			$ \Omega = \Omega $,
			$ \F = \f_t $,
			$ a = s + ( \nicefrac{k(T-s)}{n} ) $,
			$ b = t $, 
			$ f =  (  [ s + ( \nicefrac{k(T-s)}{n} ), t  ] \times \Omega \ni (u, \omega) \mapsto 
			f (u, X_{ s + ( \nicefrac{k(T-s)}{n} )  }^n(\omega), \omega  ) \in H  ) $, 
			$ F =  ( \Omega
			\ni \omega 
			\to 
			\int_{ s + ( \nicefrac{k(T-s)}{n} ) }^t
			f  ( u, X_{ s + ( \nicefrac{k(T-s)}{n} )  }^n(\omega), \omega  ) \, du \in H  ) $
			for 
			$ t \in ( s + (\nicefrac{k(T-s)}{n}), s + (\nicefrac{(k+1)(T-s)}{n} ) ] $,
			$ k \in \{ 0, 1, \ldots, n-1 \} $,
			$ n \in \N $
			in the notation of
			Lemma~\ref{lemma:Tonelli}) 
			therefore
			imply that
		for every $ n \in \N $ it holds that
		$ (X_t^n)_{t \in [s,T]} $
		is $ ( \f_t )_{t \in [s,T] } $-adapted. 
		Combining this and~\eqref{eq:Existence}
		establishes item~\eqref{item:StochasticProcess}. 
	The proof of Lemma~\ref{lemma:ProveExistence}
	is thus completed.
\end{proof}
\begin{corollary}
	\label{corollary:Existence} 
	%
	%
	Assume Setting~\ref{setting:main},
	assume that
	$ \dim(H) < \infty $,
	let 
	$ T \in (0,\infty) $,  
	$ s \in [0,T] $,
	$ C, c \in [0,\infty) $,   
	$ \delta, \kappa \in \R $,  
	$ F \in \mathcal{C}( H, H ) $,  
	$ \Phi \in \mathcal{C}(H, [0, \infty) ) $,
	let 
	$ ( \Omega, \F, \P, ( \f_t )_{t \in [0,T]} ) $
	be a filtered probability space,
	let
	$ \xi \in \M( \f_s, \B(H) ) $,
	let 
	$ O \colon [0,T] \times \Omega \to H $
	be an $ ( \f_t )_{ t \in [0,T] } $-adapted
	stochastic 
	process with continuous sample paths,  
	and assume
	for every  
	$ x, y \in H $
	that   
	%
	%
	$
	\| F( x )- F( y )  \|_H 
	\leq 
	C  \| x - y \|_{H_\delta} ( 1 + \| x \|_{H_\kappa}^c + \| y \|_{H_\kappa}^c ) $
	and
	$ \langle x, Ax +  F( x + y) \rangle_H
	\leq
	\Phi( y ) ( 1 + \| x \|_H^2 ) $.
	Then
	\begin{enumerate}[(i)]
		\item \label{item:Existence0} 
		there exists a unique function
		$ X \colon [s,T] \times \Omega \to H $
%
		which satisfies for every
		$ t \in [s,T] $,
		$ \omega \in \Omega $ 
		that
		$ ( [s,T]\ni u \mapsto X_u (\omega) \in H )
		\in \mathcal{C}( [s,T], H ) $
		and
		\begin{equation} 
		X_t (\omega)
		= 
		e^{ ( t - s ) A } \xi(\omega)
		+
		\int_s^t e^{ ( t - u ) A }  F( X_u (\omega)   ) \, du + O_t(\omega) - e^{(t-s)A} O_s(\omega) 
		%
		\end{equation}
		and
		\item \label{item:Adaptedness0} 
		it holds that
		$ X \colon [s,T] \times \Omega \to H $
		is $ ( \f_t )_{ t \in [s,T] } $-adapted.
	\end{enumerate}
\end{corollary}
\begin{proof}[Proof of Corollary~\ref{corollary:Existence} ]
	\sloppy 
	Throughout this proof let
	$ K \colon (0, \infty) \times \Omega \to [0, \infty) $
	be the function which satisfies for every
	$ r \in (0, \infty) $,
	$ \omega \in \Omega $
	that 
	$ K(r, \omega) =  
	\max \{  
	C  \| (-A)^\delta \|_{L(H)} 
	\max \{  \| (-A)^\kappa \|_{L(H)}^c, 1 \} 
		( 1 
	+ 2 ( r + \sup\nolimits_{ u \in [0,T] } \| O_u(\omega) \|_H )^c ),
	\sup_{ u \in [0,T] } \Phi( O_u ( \omega ) )
	\} $.
	Note that 
	the assumption that
	for every
	$ x, y \in H $
	it holds that 
	$ \| F( x )- F( y )  \|_H 
	\leq 
	C  \| x - y \|_{H_\delta} ( 1 + \| x \|_{H_\kappa}^c + \| y \|_{H_\kappa}^c ) $
	implies that	
	for every 
	$ t \in [0, T] $, 
	$ x, y \in H $,
	$ r \in (0,\infty) $,
	$ \omega \in \Omega $
	with 
	$ \min \{ \| x \|_H, \| y \|_H \} \leq r $
	it holds that
	\begin{equation}
	\begin{split}
	\label{eq:Monotone}
	&
	\langle x - y,
	A( x - y ) + F( x + O_t(\omega)  ) 
	- F( y + O_t(\omega)  ) \rangle_H
	\\
	&
	\leq
	\langle x-y, A(x-y) \rangle_H
	+
	\| x - y \|_H
	\| F ( x + O_t(\omega) ) - F( y + O_t( \omega ) ) \|_H
	\\
	&
	\leq
	C
	\| x - y \|_H
	\| x - y \|_{H_\delta} 
	( 1 
	+ 
	\| x + O_t(\omega) \|_{H_\kappa}^c
	+ 
	\| y + O_t(\omega) \|_{H_\kappa}^c 
	)
	\\
	&
	\leq
	C 
	\| ( -A )^\delta \|_{L(H)}
	\max \{ \| ( - A )^\kappa \|_{L(H)}^c, 1 \}
	( 1 
	+ 2( r + \sup\nolimits_{ u \in [0,T] } \| O_u(\omega) \|_H )^c )
	\| x - y \|_H^2 
	\\
	&
	\leq 
	K(r, \omega) 
	\| x - y \|_H^2
	<
	\infty
	.
	\end{split}
	\end{equation}
	In addition,
	observe that
	the assumption that for every
	$ x, y \in H $
	it holds that
	$ \langle x, Ax +  F( x + y) \rangle_H
	\leq
	\Phi( y ) ( 1 + \| x \|_H^2 ) $
	shows that for every
	$ t \in [0,T] $,
	$ x, y \in H $,
	$ \omega \in \Omega $
	it holds that	
	\begin{equation}
	\begin{split} 
	\label{eq:coercivity0}
	\langle x, A x + F( x + O_t(\omega) ) \rangle_H
	&
	\leq
	\Phi( O_t( \omega) ) ( 1 + \| x \|_H^2 )
	%
	\leq
	\sup\nolimits_{ u \in [0,T] } 
	\Phi( O_u(\omega) ) ( 1 + \| x \|_H^2 ) 
	.
	\end{split}
	\end{equation} 
	Moreover, note that the 
	assumption that
	$ \dim(H) < \infty $,
	the assumption that $ F \in \mathcal{C}(H, H) $,
	and
	the assumption that
	$ O \colon [0,T] \times \Omega \to H $
	has continuous sample paths
	ensure that
	for every 
	$ r \in (0, \infty) $,
	$ \omega \in \Omega $
	it holds that 
	$ ( [s,T] \times H \ni (u ,x) \mapsto  
	( A x 
	+
	F(x + O_u(\omega) )
	)
	\in H
	)
	\in
	\mathcal{C}([s,T] \times H, H ) $
	and  
	\begin{equation}
	\begin{split}
	\sup\nolimits_{ u \in [s,T] }
	\sup\nolimits_{ x \in H, \| x \|_H \leq r }
	\| A x + F( x + O_u(\omega) ) \|_H 
	<
	\infty
	. 
	\end{split}
	\end{equation}
	%
	%
	The assumption that
	$ (O_t)_{ t \in [0,T]} $ is 
	$ ( \f_t )_{ t \in [0,T] } $-adapted,
	\eqref{eq:Monotone},
	\eqref{eq:coercivity0},
	and Lemma~\ref{lemma:ProveExistence}
	(with
	$ H = H $,
	$ T = T $,
	$ s = s $,
	$ ( \Omega, \F, \P, ( \f_u )_{u \in [s,T]} )
	=
	( \Omega, \F, \P, ( \f_u )_{u \in [s,T]} ) $,
	$ \xi = \xi - O_s $,
	$ f = ( [s,T] \times H \times \Omega \ni (u, h, \omega) 
	\mapsto 
	A h + F( h + O_u( \omega ) ) \in H ) $,
	$ K_t(r, \omega) = 2 K(r, \omega) $
	for
	$ t \in [s,T] $,
	$ r \in (0, \infty) $ 
	in the notation of Lemma~\ref{lemma:ProveExistence})
	therefore prove that 
	\begin{enumerate}[(a)]
		\item \label{item:Uniqueness}
		there exists a unique function
		$  \X \colon [s,T] \times \Omega 
		\to H $
		which satisfies for every
		$ t \in [s,T] $,
		$ \omega \in \Omega $
		that
		$ ( [s,T] \ni u \mapsto
		\X_u (\omega) \in H )
		\in \mathcal{C}( [s,T], H ) $
		and
		\begin{equation} 
		\label{eq:EquationSatisfied}
		\X_t
		(\omega)
		=
		\xi(\omega) - O_s(\omega)
		+
		\int_s^t [ A \X_u (\omega) + F ( \X_u (\omega) + O_u(\omega) ) ] \, du
		%
		\end{equation} 
		and
		\item \label{item:Adaptedness1}
		it holds that
		$  \X \colon [s,T] \times \Omega 
		\to H $
		is $ ( \f_t )_{ t \in [s,T] } $-adapted.
	\end{enumerate} 
Next let $ X \colon [s,T] \times \Omega \to H $ 
be the stochastic process with continuous sample paths 
	which satisfies for every
	$ t \in [s,T] $,
	$ \omega \in \Omega $ 
	that
	\begin{equation} 
	\label{eq:defineHelp} 
	X_t(\omega)
	= 
	\X_t
	( \omega )
	+
	O_t(\omega) 
	.
	\end{equation}
	In addition, observe that~\eqref{eq:EquationSatisfied} 
	implies for every 
	$ t \in [s,T] $,
	$ \omega \in \Omega $
	that
	\begin{equation}
		\begin{split} 
			\label{eq:FirstHelp1}
			\X_t(\omega)
			=
			e^{(t-s)A} ( \xi(\omega) - O_s(\omega)  )
			+
			\int_s^t
			e^{(t-u)A}
			F( \X_u(\omega)  + O_u(\omega) ) \, du  
			.
		\end{split}
	\end{equation}
    This and~\eqref{eq:defineHelp}   
    show that for every
	$ t \in [s,T] $,
	$ \omega \in \Omega $ 
	it holds that
	\begin{equation}
	\label{eq:Uniqueness argument}
	X_t(\omega) 
	= 
	e^{ ( t - s ) A } \xi(\omega)
	+
	\int_s^t e^{ ( t - u ) A } F( X_u(\omega)   ) \, du + O_t(\omega) - e^{(t-s)A} O_s(\omega) 
	.
	\end{equation}
	Moreover, observe that for every function 
	$ Y \colon [s,T] \times \Omega \to H $
	%
	with
	$ \forall \, \omega\in\Omega \colon ([s,T] \ni t \mapsto Y_t (\omega))\in \mathcal{C}([s,T],H) $ 
	and
	$ \forall \, t \in [s,T] $, 
	$ \omega \in \Omega \colon Y_t(\omega)  
	= 
	e^{ ( t - s ) A } \xi(\omega)
	+
	\int_s^t e^{ ( t - u ) A } F( Y_u(\omega)   ) \, du + O_t(\omega) - e^{(t-s)A} O_s(\omega) $
	%
	and every
	$ t\in[s,T] $, 
	$ \omega \in \Omega $  
	it holds that
%
	$ Y_t(\omega) - O_t(\omega)
	=
	\xi(\omega) - O_s(\omega)
	+
	\int_s^t [ A ( Y_u(\omega) - O_u(\omega) ) 
	+ F ( [ Y_u(\omega) - O_u(\omega) ] + O_u(\omega) ) ] \, du $.
	The fact that
	for every  
	$ \omega \in \Omega $ 
	it holds that
	$ ( [s, T] \ni t \mapsto 
	[ X_t(\omega) - O_t(\omega) ] \in H )
	\in \mathcal{C}( [s,T], H ) $,
	item~\eqref{item:Uniqueness}, 
	and~\eqref{eq:Uniqueness argument} 
	therefore establish   
	item~\eqref{item:Existence0}. 
	Furthermore, 
	note that 
	item~\eqref{item:Adaptedness1},
	the fact that
	$ ( O_t )_{ t \in [0,T] } $ is $ ( \f_t )_{ t \in [0,T] } $-adapted,
	and~\eqref{eq:defineHelp} 
	establish item~\eqref{item:Adaptedness0}.
	The proof of  
    Corollary~\ref{corollary:Existence} 
	is thus completed.
\end{proof}
\section{Strong a priori bounds based on bootstrap-type arguments}
\label{section:AprioriBound}
In this section we provide in
Lemmas~\ref{lemma:F_Burgers_bootstrap20}--\ref{lemma:F_Burgers_bootstrap220}
appropriate a priori bounds for the  
 approximation process $ ( Y_t )_{ t \in [0,T] } $ introduced in 
 Setting~\ref{setting:AprioriBound} below.
The considered equations can,
in particular, be thought of as discretizations in space and time of an underlying stochastic Burgers equation.
The proofs of Lemmas~\ref{lemma:F_Burgers_bootstrap20}--\ref{lemma:F_Burgers_bootstrap220}
are based on suitable bootstrap-type arguments,
which have been intensively used in the literature to establish regularity properties of solutions to
(stochastic) evolution equations
(cf., e.g., 
\cite{JentzenPusnik2019Published, jr12}
and the references mentioned therein).
%
%
\sloppy 
\begin{setting}
\label{setting:AprioriBound} 
Assume Setting~\ref{setting:main}, 
let $ ( \Omega, \F, \P ) $
be a probability space,
let
$ T \in (0,\infty) $, 
$ \beta \in [0, 1) $, 
$ \gamma \in [0, \beta] $,
$ \xi \in \M( \F, \B(H_{ \beta } ) ) $,
$ F \in \M( \B( H_\gamma), \B(H) ) $,
$ \kappa \in \M( \B([0,T]), \B([0,T]) ) $,
$ Z \in \M( \B( [0,T] ) \otimes \F, \B( H_\gamma) ) $
satisfy for every
$ t \in [0,T] $ that
$ \kappa(t) \leq t $ 
and
$ \sup_{ u \in [0,T] } \| Z_u \|_H +
\int_0^t \| e^{(t-\kappa(s)) A}
F( Z_s ) \|_H \, ds < \infty $,
and let 
$ O \colon [0,T] \times \Omega \to H_\beta $
and 
$ Y \colon [0,T] \times \Omega \to H $ 
be  
stochastic processes
with continuous sample paths
which satisfy for every $ t \in [0,T] $ that
$ Y_t = e^{tA} \xi 
+
\int_0^t e^{(t-\kappa(s)) A}
F( Z_s ) \, ds
+
O_t $.
\end{setting}
%
%
%
%
%
%
\begin{lemma}
	\label{lemma:F_Burgers_bootstrap20}
	%
Assume Setting~\ref{setting:AprioriBound},
	let 
	$ p \in [1, \infty ) $, 
	$ \rho \in [0, \beta] $, 
	$ \alpha \in [0, 1 - \rho ) $,
	and assume that
	\begin{equation} 
	\label{eq:AssumptionFinite}
	\qquad
	\Big( \sup\nolimits_{ v \in H_{\gamma} } 
	\tfrac{ \| F(v) \|_{H_{-\alpha} } }{ 1 + \| v \|_H^2 }
	\Big) < \infty 
	.
	\end{equation} 
	Then 
	\begin{enumerate}[(i)]
	\item \label{item:statement 1} it holds for every
	$ t \in [0,T] $ that $ Y_t(\Omega) \subseteq H_\rho $,
	\item \label{item:statement 2} it holds for every $ t \in [0,T] $ that
	\begin{equation}
	\begin{split}
	\label{eq:esta1}
	\| Y_t \|_{H_\rho} 
	&
	\leq
	\|
	\xi
	\|_{ H_\rho}
	+ 
	\| O_t \|_{ H_\rho }
	+
	\tfrac{ T^{1-\alpha - \rho} }{ 1 - \alpha - \rho }
	\Big(
	\sup\nolimits_{ v \in H_{\gamma}  } 
	\tfrac{ \| F(v) \|_{H_{-\alpha} } }{ 1 + \| v \|_H^2 } 
	\Big) 
	\big( 
	1
	+
	\sup\nolimits_{ u \in [0,T] }
	\| Z_u \|_{H}^2 
	\big)
	< \infty
	,
	\end{split}
	\end{equation}
	and
	\item \label{item:statement 3} it holds for every $ t \in [0,T] $ that 
	\begin{equation}
	\begin{split}
	\label{eq:esta2}
	\| Y_t \|_{\L^{ p}(\P; H_\rho)} 
	&
	\leq
	\|
	\xi
	\|_{\L^{p}(\P; H_\rho)}
	+ 
	\| O_t \|_{ \L^p( \P; H_\rho )}
	\\
	&
	\quad
	+
	\tfrac{   T^{1-\alpha - \rho} }{ 1 - \alpha - \rho }
	\Big(
	\sup\nolimits_{ v \in H_{\gamma} } 
	\tfrac{ 1 + \| F(v) \|_{H_{-\alpha} } }{ 1 +  \| v \|_H^2 } 
	\Big) 
	\big( 
	1
	+
	\sup\nolimits_{ u \in [0,T] }
	\| Z_u
	\|_{\L^{2p}(\P; H ) }^2 
	\big) 
	.
	\end{split}
	\end{equation}
	\end{enumerate}
\end{lemma}
\begin{proof}[Proof of Lemma~\ref{lemma:F_Burgers_bootstrap20}]
	Throughout this proof assume w.l.o.g.\ that
	$ \sup_{ v \in H_\gamma } \| F(v) \|_H > 0 $.
	Note that 
	the assumption that 
	$ \forall \, t \in [0,T] \colon \kappa(t) \leq t $
	implies 
	that
	for every
	$ t \in (0, T] $ 
	it holds 
	that
	\begin{equation}
	\begin{split} 
	\label{eq:SomeImport}
	&
	\int_0^t
	\| 
	e^{(t-\kappa(u))A}  
	F( Z_u )
	\|_{ H_\rho } 
	\, 
	du
	\leq
	\int_0^t
	\| 
	(-A)^{ \alpha + \rho }
	e^{(t-\kappa(u))A}  
	\|_{ L(H) }
	\| 
	F( Z_u )
	\|_{ H_{- \alpha} } 
	\, 
	du
	\\
	& \leq
	\int_0^t 
	(t - \kappa(u) )^{-\alpha-\rho} 
	\| F( Z_u )
	\|_{ H_{ - \alpha } } 
	\, 
	du
	\\
	&
	\leq 
	\Big(
	\sup\nolimits_{ v \in H_\gamma } 
	\tfrac{ \| F(v) \|_{H_{-\alpha} } }{ 1 + \| v \|_H^2 } 
	\Big) 
	\int_0^t
	( t - \kappa(u))^{- \alpha - \rho}
	\big( 
	1
	+
	\| Z_u
	\|_{  H }^2
	\big) 
	\, 
	du
	\\
	&
	\leq 
	\Big(
	\sup\nolimits_{ v \in H_\gamma } 
	\tfrac{ \| F(v) \|_{H_{-\alpha} } }{ 1 +  \| v \|_H^2 } 
	\Big) 
	\int_0^t
	( t - u )^{- \alpha - \rho} 
	\big(
	1
	+
	\| Z_u
	\|_H^2 
	\big)
	\, du
	.
	\end{split} 
	\end{equation} 
	Hence, we obtain that for every $ t \in (0,T] $ it holds that
	\begin{equation}
	\begin{split} 
	&
	\int_0^t
	\| 
	e^{(t-\kappa(u))A}  
	F( Z_u )
	\|_{ H_\rho } 
	\, 
	du
	\leq 
	\tfrac{   t^{1-\alpha - \rho} }{ 1 - \alpha - \rho }
	\Big(
	\sup\nolimits_{ v \in H_\gamma } 
	\tfrac{ \| F(v) \|_{H_{-\alpha} } }{ 1 + \| v \|_H^2 } 
	\Big) 
	\big( 
	1
	+
	\sup\nolimits_{ u \in [0,T] }
	\| Z_u
	\|_H^2 
	\big) 
	.
	\end{split}
	\end{equation}
	%
	%
	%
	%
	%
	The triangle inequality,
	the assumption that 
	$ \sup_{ u \in [0,T] } \| Z_u \|_H < \infty $,
	and~\eqref{eq:AssumptionFinite} 
	therefore 
	prove that for every
	$ t \in [0,T] $ it holds that 
	$ Y_t(\Omega) \subseteq H_\rho $
	and
	\begin{equation}
	\begin{split}
	\label{eq:Tria10}
	\| Y_t \|_{H_\rho}
	&
	\leq
	\|  
	\xi
	\|_{H_\rho}
	+
	\int_0^t
	\| 
	e^{(t-\kappa(u) )A}  
	F( Z_u )
	\|_{ H_\rho } 
	\, 
	du
	+
	\| O_t \|_{ H_\rho } 
	\\
	&
	\leq 
	\|  
	\xi
	\|_{H_\rho}
	+
	\| O_t \|_{ H_\rho } 
	+
	\tfrac{   T^{1-\alpha - \rho} }{ 1 - \alpha - \rho }
	\Big(
	\sup\nolimits_{ v \in H_\gamma } 
	\tfrac{ \| F(v) \|_{H_{-\alpha} } }{ 1 + \| v \|_H^2 } 
	\Big) 
	\big( 
	1
	+
	\sup\nolimits_{ u \in [0,T] }
	\| Z_u
	\|_H^2 
	\big) 
	< \infty 
	.
	\end{split}
	\end{equation}
	This establishes items~\eqref{item:statement 1}
	and~\eqref{item:statement 2}.
	%
	%
%
%
Next note that~\eqref{eq:SomeImport}
and
Minkowski's integral inequality  
(see, e.g., \cite[Proposition~8 in A.1]{jk12})
ensure that
for every $ t \in (0, T] $ it holds that
\begin{equation}
\begin{split} 
\label{eq:InLp}
&
\int_0^t
\| 
e^{(t-\kappa(u))A}  
F( Z_u )
\|_{ \L^p(\P; H_\rho) } 
\, 
du
\leq 
\tfrac{   t^{1-\alpha - \rho} }{ 1 - \alpha - \rho }
\Big(
\sup\nolimits_{ v \in H_\gamma } 
\tfrac{ 1 + \| F(v) \|_{H_{-\alpha} } }{ 1 + \| v \|_H^2 } 
\Big) 
\big( 
1
+
\sup\nolimits_{ u \in [0,T] }
\| Z_u \|_{ \L^{2p}(\P; H) }^2  
\big) 
.
\end{split}
\end{equation}
%
	The triangle inequality 
	therefore
	establishes
	item~\eqref{item:statement 3}. 
	This completes the proof of Lemma~\ref{lemma:F_Burgers_bootstrap20}.
\end{proof}
\begin{lemma}
	\label{lemma:F_Burgers_bootstrap0}
	%
Assume Setting~\ref{setting:AprioriBound},
	let  
	$ p \in [1,\infty) $,  
	$ \rho \in [ 0, \beta ] $,
	$ \eta \in [ \rho, \beta ] $, 
	$ \alpha_1 \in [0, 1 - \rho ) $,
	$ \alpha_2 \in [0, 1 - \eta ) $,
	and
	assume for every
	$ t \in [0,T] $ that 
	$ Z_t( \Omega ) \subseteq H_\rho $,   
	$ \sup\nolimits_{ u \in [0,T] }
	\| Z_u
	\|_{ H_\rho }
	\leq
	\sup\nolimits_{ u \in [0,T] }
	\| Y_u
	\|_{ H_\rho } $, 
	$ \sup\nolimits_{ u \in [0,T] }
	\| Z_u
	\|_{\L^{2p}(\P; H_\rho ) }
	\leq
	\sup\nolimits_{ u \in [0,T] }
	\| Y_u
	\|_{\L^{2p}(\P; H_\rho ) } $,
	and
	\begin{equation} 
	\label{eq:Assumption1}
	\qquad
	\Big( 
	\sup\nolimits_{ v \in H_{ \max \{ \gamma, \rho \} } } 
	\tfrac{ \| F(v) \|_{H_{-\alpha_2} } }{ 1 + \| v \|_{H_\rho}^2 } 
	\Big)
	+
	\Big( 
	\sup\nolimits_{ v \in H_{\gamma}  } 
	\tfrac{ \| F(v) \|_{H_{-\alpha_1} } }{ 1 + \| v \|_{H}^2 }
	\Big) 
	< \infty 
	.
	\end{equation} 
	Then 
	\begin{enumerate}[(i)]
	\item \label{item:statement 1 strong}
	it holds for every $ t \in [0,T] $ that
	$ Y_t(\Omega) \subseteq H_\eta $,
	\item \label{item:statement 2 strong}
	it holds for every $ t \in [0,T] $ that
	\begin{equation} 
	\begin{split}
	\label{eq:Result1}  
	\|
	Y_t
	\|_{ H_\eta }  
	&
	\leq
	\|  
	\xi
	\|_{ H_\eta }
	+ 
	\| O_t \|_{ H_\eta }
	+
	\tfrac{ T^{1 - \alpha_2 - \eta} }{ 1 - \alpha_2 - \eta}
	\Big(
	\sup\nolimits_{ v \in H_{ \max \{ \gamma, \rho \} } }
	\tfrac{  \| F(v) \|_{H_{ - \alpha_2 } } }
	{ 1 + \| v \|_{ H_\rho }^2 }
	\Big) 
	\\
	&
	\cdot 
	\Big[
	1
	+
	\|
	\xi
	\|_{ H_\rho }
	+
	\sup\nolimits_{ u \in [0,T] }
	\| O_u \|_{ H_\rho }
	\\
	&
	\quad 
	+
	\tfrac{   T^{1-\alpha_1 - \rho} }{ 1 - \alpha_1 - \rho }
	\Big(
	\sup\nolimits_{ v \in H_{\gamma} } 
	\tfrac{ \| F(v) \|_{H_{-\alpha_1} } }{ 1 + \| v \|_H^2 } 
	\Big) 
	\big( 
	1
	+
	\sup\nolimits_{ u \in [0,T] }
	\| Z_u
	\|_H^2 
	\big)
	\Big]^2
	< \infty 
	,
	\end{split}
	\end{equation}
	and
	\item \label{item:statement 3 strong} 
	it holds for every $ t \in [0,T] $ that
	\begin{equation} 
	\begin{split} 
	\label{eq:Result2} 
	\|
	Y_t
	\|_{\L^p(\P; H_\eta ) }  
	&
	\leq
	\|  
	\xi
	\|_{\L^{p}(\P; H_\eta)}
	+
	\| O_t \|_{ \L^{p}( \P; H_\eta )}
	+
	\tfrac{ T^{1 - \alpha_2 - \eta} }{ 1 - \alpha_2 - \eta}
	\Big(
	\sup\nolimits_{ v \in H_{ \max \{ \gamma, \rho \} } }
	\tfrac{ 1 + \| F(v) \|_{H_{ - \alpha_2 } } }
	{ 1 + \| v \|_{ H_\rho }^2 }
	\Big) 
	\\
	&
	\cdot 
	\Big[
	1
	+
	\|
	\xi
	\|_{\L^{2 p}(\P; H_\rho)}
	+
	\sup\nolimits_{ u \in [0,T] }
	\| O_u \|_{ \L^{2p}( \P; H_\rho )}
	\\
	&
	\quad
	+
	\tfrac{   T^{1-\alpha_1 - \rho} }{ 1 - \alpha_1 - \rho }
	\Big(
	\sup\nolimits_{ v \in H_{\gamma} } 
	\tfrac{ 1 + \| F(v) \|_{H_{-\alpha_1} } }{  1 + \| v \|_H^2 } 
	\Big) 
	\big( 
	1
	+
	\sup\nolimits_{ u \in [0,T] }
	\| Z_u
	\|_{\L^{4 p}(\P; H ) }^2 
	\big) 
	\Big]^2
	.
	\end{split}
	\end{equation}
	\end{enumerate} 
\end{lemma}
\begin{proof}[Proof of Lemma~\ref{lemma:F_Burgers_bootstrap0}]
	Throughout this proof assume w.l.o.g.\ that
	$ \sup_{ v \in H_\gamma } \| F(v) \|_H > 0 $.
	Note that 
	the assumption that 
	$ \forall \, t \in [0,T] \colon \kappa(t) \leq t $
	implies 
	that
	for every
	$ t \in (0,T] $
	it holds
	that
	\begin{equation}
	\begin{split}
	\label{eq:more_regularity_drift0}
	&
	\int_0^t
	\|
	 e^{(t- \kappa(u) )A}  
	F( Z_u )
	\|_{ H_\eta } 
	\, 
	du
	\leq 
	\int_0^t
	\| 
	(-A)^{\alpha_2 + \eta}
	e^{(t-\kappa(u))A}
	\|_{L(H)}
	\|  
	F( Z_u )
	\|_{ H_{ - \alpha_2 } } 
	\, 
	du
	\\
	&
	\leq
	\int_0^t 
	(t - \kappa(u) )^{-\alpha_2-\eta} 
	\|   F( Z_u )
	\|_{ H_{ - \alpha_2 } } 
	\, 
	du
	\\
	&
	\leq   
	\Big(
	\sup\nolimits_{ v \in H_{ \max \{ \gamma, \rho \} } }
	\tfrac{ \| F(v) \|_{H_{ - \alpha_2 } } }
	{ 1 + \| v \|_{ H_\rho }^2 }
	\Big)
	\int_0^t
	( t - \kappa(u) )^{- \alpha_2 - \eta }
	\big( 
	1
	+
	\| 
	Z_u
	\|_{ H_\rho }^2
	\big)
	\, 
	du
	\\
	&
	\leq   
	\Big(
	\sup\nolimits_{ v \in H_{ \max \{ \gamma, \rho \} } }
	\tfrac{ \| F(v) \|_{H_{ - \alpha_2 } } }
	{ 1 + \| v \|_{ H_\rho }^2 }
	\Big)
	\int_0^t
	( t - u )^{- \alpha_2 - \eta} 
	\big( 
	1
	+
	\| Z_u
	\|_{ H_\rho }^2 
	\big) 
	\, du
	.
	\end{split} 
	\end{equation} 
	Hence, we obtain that for every $ t \in (0,T] $ it holds that
	\begin{equation}
	\begin{split} 
	\label{eq:Next step}
	\int_0^t
	\|
	e^{(t- \kappa(u) )A}  
	F( Z_u )
	\|_{ H_\eta } 
	\, 
	du
	&
	\leq    
	\Big(
	\sup\nolimits_{ v \in H_{ \max \{ \gamma, \rho \} } }
	\tfrac{ \| F(v) \|_{H_{ - \alpha_2 } } }
	{ 1 + \| v \|_{ H_\rho }^2 }
	\Big)
	\tfrac{ t^{1-\alpha_2 - \eta} }{ 1 - \alpha_2 - \eta}
	\big( 
	1
	+
	\sup\nolimits_{ u \in [0,T] } 
	\| 
	Z_u
	\|_{ H_\rho }^2 
	\big)
	\\
	&
	\leq    
	\Big(
	\sup\nolimits_{ v \in H_{ \max \{ \gamma, \rho \} } }
	\tfrac{ \| F(v) \|_{H_{ - \alpha_2 } } }
	{ 1 + \| v \|_{ H_\rho }^2 }
	\Big)
	\tfrac{ t^{1-\alpha_2 - \eta} }{ 1 - \alpha_2 - \eta}
	\big( 
	1
	+
	\sup\nolimits_{ u \in [0,T] } 
	\| 
	Z_u
	\|_{ H_\rho }
	\big)^2
	.
	\end{split}
	\end{equation}
	%
	%
	%
Next observe that~\eqref{eq:more_regularity_drift0}
and 
	Minkowski's integral inequality  
	(see, e.g., \cite[Proposition~8 in A.1]{jk12}) 
	ensure that 
	for every $ t \in (0,T] $ it holds that
	\begin{equation}
	\begin{split}
	\label{eq:more_regularity_drift0b}
	& 
	\int_0^t
	\|
	e^{(t- \kappa(u) )A}  
	F( Z_u )
	\|_{\L^p(\P; H_\eta ) } 
	\, 
	du
	\\
	&	  
	\leq    
	\Big(
	\sup\nolimits_{ v \in H_{ \max \{ \gamma, \rho \} } }
	\tfrac{ 1 + \| F(v) \|_{H_{ - \alpha_2 } } }
	{ 1 + \| v \|_{ H_\rho }^2 }
	\Big)
	\tfrac{ t^{1-\alpha_2 - \eta} }{ 1 - \alpha_2 - \eta}
	\big( 
	1
	+
	\sup\nolimits_{ u \in [0,T] }
	\| Z_u
	\|_{\L^{2p}(\P; H_\rho ) }^2 
	\big) 
	\\
	&
	\leq 
	\Big(
	\sup\nolimits_{ v \in H_{ \max \{ \gamma, \rho \} } }
	\tfrac{ 1 + \| F(v) \|_{H_{ - \alpha_2 } } }
	{ 1 + \| v \|_{ H_\rho }^2 }
	\Big)
	\tfrac{ t^{1-\alpha_2 - \eta} }{ 1 - \alpha_2 - \eta}
	\big( 
	1
	+
	\sup\nolimits_{ u \in [0,T] }
	\| Z_u
	\|_{\L^{2p}(\P; H_\rho ) }  
	\big)^2 
	.
	\end{split}
	\end{equation}
	Moreover, 
	note that~\eqref{eq:Assumption1}
	and
	Lemma~\ref{lemma:F_Burgers_bootstrap20}
	(with 
	$ p = 2p $,
	$ \rho = \rho $,
	$ \alpha = \alpha_1 $
	in the notation of Lemma~\ref{lemma:F_Burgers_bootstrap20})
	imply that
	\begin{enumerate}[(a)]
		\item  it holds for every
		$ t \in [0,T] $ that $ Y_t(\Omega) \subseteq H_\rho $,
	\item \label{item:66}
	it holds
	that 
	\begin{equation}
	\begin{split}
	\label{eq:Apply1}
	&  
	\!\!\!\!\! 
	\sup\nolimits_{ u \in [0,T] }
	\| Z_u \|_{H_\rho} 
	\leq
	\sup\nolimits_{ u \in [0,T] }
	\| Y_u
	\|_{ H_\rho }
	\\
	& 
		\!\!\!\!\!
	\leq
	\|
	\xi
	\|_{ H_\rho}
	+
	\sup\nolimits_{ u \in [0,T] } \| O_u \|_{ H_\rho }
	+
	\tfrac{ T^{1-\alpha_1 - \rho} }{ 1 - \alpha_1 - \rho }
	\Big(
	\sup\nolimits_{ v \in H_{\gamma} } 
	\tfrac{ \| F(v) \|_{H_{-\alpha_1} } }{ 1 + \| v \|_H^2 } 
	\Big) 
	\big( 
	1
	+
	\sup\nolimits_{ u \in [0,T] }
	\| Z_u \|_{H}^2 
	\big)
	<
	\infty,
	\end{split}
	\end{equation}
	and
	\item \label{item:67}
	it holds 
	that 
	\begin{equation}
	\begin{split}
	\label{eq:Apply2}
	\sup\nolimits_{ u \in [0,T] }
	\| Z_u \|_{\L^{ 2p}(\P; H_\rho)} 
	&
	\leq
	\sup\nolimits_{ u \in [0,T] }
	\| Y_u
	\|_{\L^{2p}(\P; H_\rho ) }
	\leq
	\|
	\xi
	\|_{\L^{2p}(\P; H_\rho)}
	+
	\sup\nolimits_{ u \in [0,T] }\| O_u \|_{ \L^{2p}( \P; H_\rho )}
	\\
	&
	\quad
	+
	\tfrac{   T^{1-\alpha_1 - \rho} }{ 1 - \alpha_1 - \rho }
	\Big(
	\sup\nolimits_{ v \in H_\gamma } 
	\tfrac{ 1 + \| F(v) \|_{H_{-\alpha_1} } }{ 1 + \| v \|_H^2 } 
	\Big) 
	\big( 
	1
	+
	\sup\nolimits_{ u \in [0,T] }
	\| Z_u
	\|_{\L^{4p}(\P; H ) }^2 
	\big) 
	.
	\end{split}
	\end{equation}
	\end{enumerate} 
	Observe that the triangle inequality,
	\eqref{eq:Assumption1}, 
	\eqref{eq:Next step}
	and
	item~\eqref{item:66}
	ensure that for every $ t \in [0,T] $ it holds that
	$ Y_t ( \Omega) \subseteq H_\eta $
	and	
	\begin{equation}
	\begin{split} 
	\label{eq:some_integral0}
	\| Y_t  \|_{ H_\eta }
	&
	\leq
	\|  
	\xi
	\|_{ H_\eta}
	+ 
	\int_0^t
	\|
	e^{(t-\kappa(u))A}  
	F( Z_u )
	\|_{ H_\eta } 
	\, 
	du
	+  
	\| O_t \|_{ H_\eta }
	\\
	&
	\leq 
	\|  
	\xi
	\|_{ H_\eta}
	+
	\| O_t \|_{ H_\eta }
	+
	\Big(
	\sup\nolimits_{ v \in H_{ \max \{ \gamma, \rho \} } }
	\tfrac{ \| F(v) \|_{H_{ - \alpha_2 } } }
	{ 1 + \| v \|_{ H_\rho }^2 }
	\Big)
	\tfrac{ T^{1-\alpha_2 - \eta} }{ 1 - \alpha_2 - \eta}
	\big( 
	1
	+
	\sup\nolimits_{ u \in [0,T] } 
	\| 
	Z_u
	\|_{ H_\rho }^2 
	\big)
	\\
	&
	\leq
	\|  
	\xi
	\|_{ H_\eta }
	+ 
	\| O_t \|_{ H_\eta }
	+
	\tfrac{ T^{1 - \alpha_2 - \eta} }{ 1 - \alpha_2 - \eta}
	\Big(
	\sup\nolimits_{ v \in H_{ \max \{ \gamma, \rho \} } }
	\tfrac{  \| F(v) \|_{H_{ - \alpha_2 } } }
	{ 1 + \| v \|_{ H_\rho }^2 }
	\Big) 
	\\
	&
	\quad
	\cdot 
	\Big[
	1
	+
	\|
	\xi
	\|_{ H_\rho }
	+
	\sup\nolimits_{ u \in [0,T] }
	\| O_u \|_{ H_\rho }
	\\
	&
	\qquad 
	+
	\tfrac{   T^{1-\alpha_1 - \rho} }{ 1 - \alpha_1 - \rho }
	\Big(
	\sup\nolimits_{ v \in H_{\gamma} } 
	\tfrac{ \| F(v) \|_{H_{-\alpha_1} } }{ 1 + \| v \|_H^2 } 
	\Big) 
	\big( 
	1
	+
	\sup\nolimits_{ u \in [0,T] }
	\| Z_u
	\|_H^2 
	\big)
	\Big]^2
	< \infty 
	.
	\end{split}
	\end{equation}
	This
	establishes items~\eqref{item:statement 1 strong}
 	and~\eqref{item:statement 2 strong}.
	%
	%
	%
	Furthermore,
	observe that
	the triangle inequality
	and~\eqref{eq:more_regularity_drift0b}
	prove that for every $ t \in [0,T] $ it holds that
	\begin{equation}
	\begin{split}
	\| Y_t  \|_{ \L^p(\P; H_\eta) }
	&
	\leq
	\|  
	\xi
	\|_{ \L^p(\P; H_\eta) }
	+ 
	\int_0^t
	\|
	e^{(t-\kappa(u))A}  
	F( Z_u )
	\|_{ \L^p(\P; H_\eta) } 
	\, 
	du
	+  
	\| O_t \|_{ \L^p(\P; H_\eta) }
	\\
	&
	\leq 
	\|  
	\xi
	\|_{ \L^p(\P; H_\eta) } 
	+  
	\| O_t \|_{ \L^p(\P; H_\eta) }
	\\
	&
	\quad 
	+
	\Big(
	\sup\nolimits_{ v \in H_{ \max \{ \gamma, \rho \} } }
	\tfrac{ 1 + \| F(v) \|_{H_{ - \alpha_2 } } }
	{ 1 + \| v \|_{ H_\rho }^2 }
	\Big)
	\tfrac{ T^{1-\alpha_2 - \eta} }{ 1 - \alpha_2 - \eta}
	\big( 
	1
	+
	\sup\nolimits_{ u \in [0,T] }
	\| Z_u
	\|_{\L^{2p}(\P; H_\rho ) }  
	\big)^2 
	.
	\end{split}
	\end{equation}
	Combining this
	and 
	item~\eqref{item:67} 
	establishes
	item~\eqref{item:statement 3 strong}.
	The proof of Lemma~\ref{lemma:F_Burgers_bootstrap0}
	is thus completed.
\end{proof}
\begin{lemma}
	\label{lemma:F_Burgers_bootstrap220}
	%
Assume Setting~\ref{setting:AprioriBound},
	let 
	$ p \in [1,\infty) $,
	$ \rho \in [0, \beta ] $,    
	$ \eta \in [\rho, \beta ] $,
	$ \iota \in [ \eta, \beta ] $,
	$ \alpha_1 \in [0, 1 - \rho ) $,
	$ \alpha_2 \in [0, 1 - \eta ) $,
	and assume for every
	$ t \in [0,T] $ that 
	$ Z_t( \Omega ) \subseteq H_\eta $, 
	$ \sup\nolimits_{ u \in [0,T] }
	\| Z_u
	\|_{ H_\rho   }
	\leq
	\sup\nolimits_{ u \in [0,T] }
	\| Y_u
	\|_{ H_\rho } $,
	$ \sup\nolimits_{ u \in [0,T] }
	\|
	Z_u
	\|_{ H_\eta  }  
	\leq \sup\nolimits_{ u \in [0,T] }
	\|
	Y_u
	\|_{ H_\eta } $,
	\sloppy 
	$ \sup\nolimits_{ u \in [0,T] }
	\| Z_u
	\|_{\L^{4p}(\P; H_\rho ) }
	\leq
	\sup\nolimits_{ u \in [0,T] }
	\| Y_u
	\|_{\L^{4p}(\P; H_\rho ) } $,
	$ \sup\nolimits_{ u \in [0,T] }
	\|
	Z_u
	\|_{\L^{2p}(\P; H_\eta ) }  
	\leq \sup\nolimits_{ u \in [0,T] }
	\|
	Y_u
	\|_{\L^{2p}(\P; H_\eta ) } $,
	and
	\begin{equation}
	\label{eq:AssumptionFinite2}
	\Big[ 
	\sup\nolimits_{ v \in H_{ \max \{ \gamma, \eta \} }  }
	\tfrac{ \| F(v) \|_{H  } }
	{ 1 + \| v \|_{H_{ \eta } }^2 }
	\Big]
	+
	\Big[
	\sup\nolimits_{ v \in H_{ \max \{ \gamma, \rho \} } }
	\tfrac{ \| F(v) \|_{H_{ - \alpha_2 } } }
	{ 1 + \| v \|_{ H_\rho }^2 }
	\Big]
	+
	\Big[ 
	\sup\nolimits_{ v \in H_{\gamma}  } 
	\tfrac{ \| F(v) \|_{H_{-\alpha_1} } }{  1 + \| v \|_H^2 }
	\Big]
	< \infty 
	.
	\end{equation} 
	Then 
	\begin{enumerate}[(i)]
	\item \label{item:statement 1 stronger} it holds for every
	$ t \in [0,T] $ that
	$ Y_t(\Omega) \subseteq H_\iota $,
	\item \label{item:statement 2 stronger} it holds for every
	$ t \in [0,T] $ that
	\begin{equation}
	\begin{split} 
	\label{eq:result1}
	&  
	\| Y_t  \|_{ H_\iota }
	\leq
	\| 
	\xi
	\|_{ H_\iota }
	+ 
	\| 
	O_t
	\|_{ H_\iota }
	+ 
	\tfrac{ T^{1-\iota} }{ 1 - \iota }
	\Big(
	\sup\nolimits_{ v \in H_{ \max \{ \gamma, \eta \} } }
	\tfrac{ \| F(v) \|_{H  } }
	{ 1 + \| v \|_{H_{ \eta } }^2 }
	\Big) 
	\\
	&
	\cdot
	\bigg[
	1
	+
	\|  
	\xi
	\|_{ H_\eta }
	+
	\sup\nolimits_{ u \in [0,T] }
	\| O_u \|_{ H_\eta }
	+
	\tfrac{ T^{1 - \alpha_2 - \eta} }{ 1 - \alpha_2 - \eta}
	\Big(
	\sup\nolimits_{ v \in H_{ \max \{ \gamma, \rho \} } }
	\tfrac{  \| F(v) \|_{H_{ - \alpha_2 } } }
	{ 1 + \| v \|_{ H_\rho }^2 }
	\Big) 
	\\
	&
	\cdot 
	\Big[
	1
	+
	\|
	\xi
	\|_{ H_\rho }
	+
	\sup\nolimits_{ u \in [0,T] }
	\| O_u \|_{ H_\rho }
	\\
	&
	\quad 
	+
	\tfrac{   T^{1-\alpha_1 - \rho} }{ 1 - \alpha_1 - \rho }
	\Big(
	\sup\nolimits_{ v \in H_{\gamma} } 
	\tfrac{ \| F(v) \|_{H_{-\alpha_1} } }{  1 + \| v \|_H^2 } 
	\Big) 
	\big( 
	1
	+
	\sup\nolimits_{ u \in [0,T] }
	\| Z_u
	\|_{H}^2 
	\big) 
	\Big]^2
	\bigg]^2
	< \infty
	,
	\end{split}
	\end{equation}
	and
	\item \label{item:statement 3 stronger} it holds for every
	$ t \in [0,T] $ that
	\begin{equation}
	\begin{split} 
	\label{eq:result2}
	&  
	\| Y_t  \|_{\L^{ p}(\P; H_\iota )}
	\leq
	\| 
	\xi
	\|_{\L^{ p}(\P; H_\iota )}
	+ 
	\| 
	O_t
	\|_{\L^{ p}(\P; H_\iota )}
	+ 
	\tfrac{ T^{1-\iota} }{ 1 - \iota }
	\Big(
	\sup\nolimits_{ v \in H_{ \max \{ \gamma, \eta \} } }
	\tfrac{ 1 + \| F(v) \|_{H  } }
	{ 1 + \| v \|_{H_{ \eta } }^2 }
	\Big) 
	\\
	& \cdot
	\bigg[
	1
	+
	\|  
	\xi
	\|_{\L^{2p}(\P; H_\eta)}
	+
	\sup\nolimits_{ u \in [0,T] }
	\| O_u \|_{ \L^{2p}( \P; H_\eta )}
	+
	\tfrac{ T^{1 - \alpha_2 - \eta} }{ 1 - \alpha_2 - \eta}
	\Big(
	\sup\nolimits_{ v \in H_{ \max \{ \gamma, \rho \} } }
	\tfrac{ 1 + \| F(v) \|_{H_{ - \alpha_2 } } }
	{ 1 + \| v \|_{ H_\rho }^2 }
	\Big) 
	\\
	&
	\cdot 
	\Big[
	1
	+
	\|
	\xi
	\|_{\L^{4 p}(\P; H_\rho)}
	+
	\sup\nolimits_{ u \in [0,T] }
	\| O_u \|_{ \L^{4p}( \P; H_\rho )}
	\\
	&
	\quad 
	+
	\tfrac{   T^{1-\alpha_1 - \rho} }{ 1 - \alpha_1 - \rho }
	\Big(
	\sup\nolimits_{ v \in H_{\gamma}  } 
	\tfrac{ 1 + \| F(v) \|_{H_{-\alpha_1} } }{  1 + \| v \|_H^2 } 
	\Big) 
	\big( 
	1
	+
	\sup\nolimits_{ u \in [0,T] }
	\| Z_u
	\|_{\L^{8 p}(\P; H ) }^2 
	\big)
	\Big]^2
	\bigg]^2
	.
	\end{split}
	\end{equation}
	\end{enumerate} 
\end{lemma}
\begin{proof}[Proof of Lemma~\ref{lemma:F_Burgers_bootstrap220}]
	Throughout this proof 
	assume w.l.o.g.\ that
	$ \sup_{ v \in H_\gamma } \| F(v) \|_H > 0 $.
	Observe that 
	the assumption that
	$ \forall \, t \in [0,T] \colon \kappa(t) \leq t $
	implies 
	that
	for every
	$ t \in (0,T] $
	 it holds 
	that
	\begin{equation}
	\begin{split} 
	\label{eq:estimateTT0}
	&
	\int_0^t
	\|
	e^{(t-\kappa(u))A}  
	F( Z_u )
	\|_{ H_\iota } 
	\, 
	du
	\leq
	\int_0^t 
	(t - \kappa(u))^{-\iota} 
	\|   F( Z_u )
	\|_{H} 
	\, 
	du
	\\
	&
	\leq
	\Big(
	\sup\nolimits_{ v \in H_{ \max \{ \gamma, \eta \} } }
	\tfrac{ \| F(v) \|_{H  } }
	{ 1 + \| v \|_{H_{ \eta } }^2 }
	\Big)
	\int_0^t 
	(t - \kappa(u))^{-\iota} 
	\big( 
	1
	+
	\|    
	Z_u
	\|_{ H_{ \eta } }^2
	\big) 
	\, 
	du
	\\
	&
	\leq 
	\Big(
	\sup\nolimits_{ v \in H_{ \max \{ \gamma, \eta \} } }
	\tfrac{ \| F(v) \|_{H  } }
	{ 1 + \| v \|_{H_{ \eta } }^2 }
	\Big)
	\int_0^t 
	(t - u)^{-\iota} 
	\big( 
	1
	+
	\|    
	Z_u
	\|_{ H_{ \eta } }^2
	\big) 
	\, 
	du
	.
	%
	%
	\end{split} 
	\end{equation} 
	Hence, we obtain that for every $ t \in (0, T] $ it holds that
	\begin{equation}
	\begin{split} 
	\label{eq:Next step again}
	\int_0^t
	\|
	e^{(t-\kappa(u))A}  
	F( Z_u )
	\|_{ H_\iota } 
	\, 
	du
	& 
	\leq
	\Big(
	\sup\nolimits_{ v \in H_{ \max \{ \gamma, \eta \} }  }
	\tfrac{ \| F(v) \|_{H  } }
	{ 1 + \| v \|_{H_{ \eta } }^2 }
	\Big)
	\tfrac{ t^{1-\iota} }{ 1 - \iota }
	\big(
	1
	+
	\sup\nolimits_{ u \in [0,T] }
	\|    Z_u 
	\|_{  H_{ \eta }   }^2
	\big)
	\\
	&
	\leq
	\Big(
	\sup\nolimits_{ v \in H_{ \max \{ \gamma, \eta \} }  }
	\tfrac{  \| F(v) \|_{H  } }
	{ 1 + \| v \|_{H_{ \eta } }^2 }
	\Big)
	\tfrac{ t^{1-\iota} }{ 1 - \iota }
	\big(
	1
	+
	\sup\nolimits_{ u \in [0,T] }
	\|    Z_u 
	\|_{  H_{ \eta }   }
	\big)^2
	.
	\end{split} 
	\end{equation} 
	%
 %
Moreover, note that~\eqref{eq:estimateTT0}
and 
	Minkowski's integral inequality  
	(see, e.g., \cite[Proposition~8 in A.1]{jk12}) 
	prove  that for every $ t \in (0,T] $ it holds that
	\begin{equation}
	\begin{split} 
	\label{eq:estimateTT0b}
	&
	\int_0^t
	\|
	e^{(t-\kappa(u))A}  
	F( Z_u )
	\|_{ \L^p(\P;H_\iota )} 
	\, 
	du
	\\
	&
	\leq
	\Big(
	\sup\nolimits_{ v \in H_{ \max \{ \gamma, \eta \} } }
	\tfrac{ \| F(v) \|_{H  } }
	{ 1 + \| v \|_{H_{ \eta } }^2 }
	\Big)
	\tfrac{ t^{1-\iota} }{ 1 - \iota }
	\big( 
	1
	+
	\sup\nolimits_{ u \in [0,T] }
	\|    Z_u 
	\|_{\L^{2p}(\P; H_{ \eta } ) }^2
	\big) 
	\\
	&
	\leq 
	\Big(
	\sup\nolimits_{ v \in H_{ \max \{ \gamma, \eta \} } }
	\tfrac{ \| F(v) \|_{H  } }
	{ 1 + \| v \|_{H_{ \eta } }^2 }
	\Big)
	\tfrac{ t^{1-\iota} }{ 1 - \iota }
	\big( 
	1
	+
	\sup\nolimits_{ u \in [0,T] }
	\|    Z_u 
	\|_{\L^{2p}(\P; H_{ \eta } ) }
	\big)^2
	.
	\end{split} 
	\end{equation} 
	%
	Next observe that~\eqref{eq:AssumptionFinite2},
	the assumption that 
	$ \sup\nolimits_{ u \in [0,T] }
	\| Z_u
	\|_{ H_\rho   }
	\leq
	\sup\nolimits_{ u \in [0,T] }
	\| Y_u
	\|_{ H_\rho } $,
	the assumption that 
	$ \sup\nolimits_{ u \in [0,T] }
	\| Z_u
	\|_{\L^{4p}(\P; H_\rho ) }
	$
	$ \leq
	\sup\nolimits_{ u \in [0,T] }
	\| Y_u
	\|_{\L^{4p}(\P; H_\rho ) } $,
	and
	Lemma~\ref{lemma:F_Burgers_bootstrap0}
	(with $ p = 2 p $, 
	$ \rho = \rho $,
	$ \eta = \eta $, 
	$ \alpha_1 = \alpha_1 $,
	$ \alpha_2 = \alpha_2 $
	in the notation of 
	Lemma~\ref{lemma:F_Burgers_bootstrap0})
	show that
	\begin{enumerate}[(a)]
	%
%
\item it holds for every $ t \in [0,T] $ that
$ Y_t(\Omega) \subseteq H_\eta $,
	\item \label{item:74} 
	it holds
	 that 
	\begin{equation} 
	\begin{split}  
	\label{eq:Result1b} 
	& 
	\sup\nolimits_{u \in [0,T] }
	\|
	Z_u
	\|_{ H_\eta }  
	\leq \sup\nolimits_{ u \in [0,T] }
	\|
	Y_u
	\|_{ H_\eta }
	\\
	&
	\leq
	\|  
	\xi
	\|_{ H_\eta }
	+
	\sup\nolimits_{ u \in [0,T] }
	\| O_u \|_{ H_\eta }
	+
	\tfrac{ T^{1 - \alpha_2 - \eta} }{ 1 - \alpha_2 - \eta}
	\Big(
	\sup\nolimits_{ v \in H_{ \max \{ \gamma, \rho \} } }
	\tfrac{  \| F(v) \|_{H_{ - \alpha_2 } } }
	{ 1 + \| v \|_{ H_\rho }^2 }
	\Big) 
	\\
	&
	\cdot 
	\Big[
	1
	+
	\|
	\xi
	\|_{ H_\rho }
	+
	\sup\nolimits_{ u \in [0,T] }
	\| O_u \|_{ H_\rho }
	\\
	&
	\quad 
	+
	\tfrac{   T^{1-\alpha_1 - \rho} }{ 1 - \alpha_1 - \rho }
	\Big(
	\sup\nolimits_{ v \in H_{\gamma} } 
	\tfrac{ \| F(v) \|_{H_{-\alpha_1} } }{  1 + \| v \|_H^2 } 
	\Big) 
	\big( 
	1
	+
	\sup\nolimits_{ u \in [0,T] }
	\| Z_u
	\|_H^2 
	\big) 
	\Big]^2
	< \infty 
	,
	\end{split}
	\end{equation}
	and
	\item \label{item:75} 
	it holds 
	that 
	\begin{equation} 
	\begin{split} 
	\label{eq:Result2b} 
	& 
	\sup\nolimits_{u \in [0,T] }
	\|
	Z_u
	\|_{\L^{2p}(\P; H_\eta ) }  
	\leq \sup\nolimits_{ u \in [0,T] }
	\|
	Y_u
	\|_{\L^{2p}(\P; H_\eta ) }
	\\
	&
	\leq
	\|  
	\xi
	\|_{\L^{2 p}(\P; H_\eta)}
	+
	\sup\nolimits_{ u \in [0,T] }
	\| O_u \|_{ \L^{2 p}( \P; H_\eta )}
	+
	\tfrac{ T^{1 - \alpha_2 - \eta} }{ 1 - \alpha_2 - \eta}
	\Big(
	\sup\nolimits_{ v \in H_{ \max \{ \gamma, \rho \} } }
	\tfrac{ 1 + \| F(v) \|_{H_{ - \alpha_2 } } }
	{ 1 + \| v \|_{ H_\rho }^2 }
	\Big) 
	\\
	&
	\quad 
	\cdot 
	\Big[
	1
	+
	\|
	\xi
	\|_{\L^{4 p}(\P; H_\rho)}
	+
	\sup\nolimits_{ u \in [0,T] }
	\| O_u \|_{ \L^{4 p}( \P; H_\rho )}
	\\
	&
	\qquad
	+
	\tfrac{   T^{1-\alpha_1 - \rho} }{ 1 - \alpha_1 - \rho }
	\Big(
	\sup\nolimits_{ v \in H_{\gamma} } 
	\tfrac{ 1 + \| F(v) \|_{H_{-\alpha_1} } }{  1 + \| v \|_H^2 } 
	\Big) 
	\big( 
	1
	+
	\sup\nolimits_{ u \in [0,T] }
	\| Z_u
	\|_{\L^{8 p}(\P; H ) }^2 
	\big) 
	\Big]^2
	.
	\end{split}
	\end{equation}
	\end{enumerate} 
	Note that
	the triangle inequality,
	\eqref{eq:AssumptionFinite2},
	\eqref{eq:Next step again},
	and item~\eqref{item:74}  
	ensure that for every $ t \in [0,T] $ it holds that
	$ Y_t (\Omega) \subseteq H_\iota $
	and
	\begin{equation}
	\begin{split}
	\label{eq:bootstrap240}
	\| Y_t  \|_{H_\iota }
	&
	\leq
	\| 
	\xi
	\|_{ H_\iota }
	+ 
	\int_0^t
	\|
	e^{(t-\kappa(u))A}  
	F( Z_u )
	\|_{ H_\iota } 
	\, 
	du 
	+
	\| 
	O_t
	\|_{ H_\iota }
	\\
	&
	\leq 
	\|  
	\xi
	\|_{ H_\iota }
	+
	\| 
	O_t
	\|_{ H_\iota }
	+
	\Big(
	\sup\nolimits_{ v \in H_{ \max \{ \gamma, \eta \} }  }
	\tfrac{  \| F(v) \|_{H  } }
	{ 1 + \| v \|_{H_{ \eta } }^2 }
	\Big)
	\tfrac{ T^{1-\iota} }{ 1 - \iota }
	\big(
	1
	+
	\sup\nolimits_{ u \in [0,T] }
	\|    Z_u 
	\|_{  H_{ \eta }   }
	\big)^2 
	\\
	&
	\leq 
	\|  
	\xi
	\|_{ H_\iota }
	+
	\| 
	O_t
	\|_{ H_\iota }
	+ 
	\tfrac{ T^{1-\iota} }{ 1 - \iota }
	\Big(
	\sup\nolimits_{ v \in H_{ \max \{ \gamma, \eta \} } }
	\tfrac{  \| F(v) \|_{H  } }
	{ 1 + \| v \|_{H_{ \eta } }^2 }
	\Big) 
	\\
	&
	\quad
	\cdot
	\bigg[
	1
	+
	\|  
	\xi
	\|_{ H_\eta }
	+
	\sup\nolimits_{ u \in [0,T] }
	\| O_u \|_{ H_\eta }
	+
	\tfrac{ T^{1 - \alpha_2 - \eta} }{ 1 - \alpha_2 - \eta}
	\Big(
	\sup\nolimits_{ v \in H_{ \max \{ \gamma, \rho \} } }
	\tfrac{ \| F(v) \|_{H_{ - \alpha_2 } } }
	{ 1 + \| v \|_{ H_\rho }^2 }
	\Big) 
	\\
	&
	\quad
	\cdot 
	\Big[
	1
	+
	\|
	\xi
	\|_{ H_\rho }
	+
	\sup\nolimits_{ u \in [0,T] }
	\| O_u \|_{ H_\rho }
	\\
	&
	\qquad 
	+
	\tfrac{   T^{1-\alpha_1 - \rho} }{ 1 - \alpha_1 - \rho }
	\Big(
	\sup\nolimits_{ v \in H_{\gamma} } 
	\tfrac{ \| F(v) \|_{H_{-\alpha_1} } }{  1 + \| v \|_H^2 } 
	\Big) 
	\big( 
	1
	+
	\sup\nolimits_{ u \in [0,T] }
	\| Z_u
	\|_{H}^2 
	\big) 
	\Big]^2
	\bigg]^2
	<
	\infty 
	.
	\end{split}
	\end{equation}
	This establishes   
	items~\eqref{item:statement 1 stronger}
	and~\eqref{item:statement 2 stronger}.
	Furthermore, observe that the triangle inequality
	and~\eqref{eq:estimateTT0b}
	prove that for every $ t \in [0,T] $ it holds that
	\begin{equation}
	\begin{split}
	\| Y_t  \|_{ \L^p(\P; H_\iota) }
	&
	\leq
	\|  
	\xi
	\|_{ \L^p(\P; H_\iota) }
	+ 
	\int_0^t
	\|
	e^{(t-\kappa(u))A}  
	F( Z_u )
	\|_{ \L^p(\P; H_\iota) } 
	\, 
	du
	+  
	\| O_t \|_{ \L^p(\P; H_\iota) }
	\\
	&
	\leq 
	\|  
	\xi
	\|_{ \L^p(\P; H_\iota) } 
	+  
	\| O_t \|_{ \L^p(\P; H_\iota) }
	\\
	&
	\quad 
	+
	\Big(
	\sup\nolimits_{ v \in H_{ \max \{ \gamma, \eta \} } }
	\tfrac{ 1 + \| F(v) \|_{H  } }
	{ 1 + \| v \|_{H_{ \eta } }^2 }
	\Big)
	\tfrac{ T^{1-\iota} }{ 1 - \iota }
	\big( 
	1
	+
	\sup\nolimits_{ u \in [0,T] }
	\|    Z_u 
	\|_{\L^{2p}(\P; H_{ \eta } ) }
	\big)^2
	.
	\end{split} 
	\end{equation} 
	Combining
	this 
	and
	item~\eqref{item:75}
establishes item~\eqref{item:statement 3 stronger}.
	%
	%
	%
	The proof of Lemma~\ref{lemma:F_Burgers_bootstrap220}
	is thus completed.
\end{proof} 
\section{Properties of the nonlinearity}
\label{section:Properties of the nonlinearity}
In this section  
we recall and derive
in
Subsection~\ref{subsection:Sobolev}
and
in Subsection~\ref{subsection:Nonlinearity} some 
partially well-known properties of
certain Sobolev spaces
and the nonlinearity
appearing in the stochastic Burgers equations,
respectively.
We employ these results
to establish
in
Theorem~\ref{theorem:existence_Burgers}
in 
Section~\ref{section:Existence}
below
the main result of this article.
%
%
%
\begin{setting} 
\label{setting:Examples}
Assume Setting~\ref{setting:main},   
%
%
let
$ \lambda
\colon
\mathcal{B}( (0,1) )
\rightarrow [0, 1] $
be the Lebesgue-Borel
measure on
$ (0,1) $, 
for every measure space
$ ( \Omega, \F, \mu) $,
every measurable space $ ( S, \mathcal{S} ) $,
every set $ R $,
and every function
$ f \colon \Omega \to R $
let
	$ [f]_{\mu, \mathcal{S}} =
	\{
	g \in \M(\F,
	\mathcal{S})
	\colon
	(
	\exists\, D \in \F \colon
	\mu( D ) = 0 \,\,\text{and}\,\,
	\{\omega \in \Omega \colon
	f(\omega) \neq g(\omega)
	\}
	\subseteq D
	)
	\} $,  
let 
$ c_0 \in (0, \infty) $,
$ c_1 \in \R $,
assume that
$ (H, \langle \cdot, \cdot \rangle_H,
\left \| \cdot \right \|_H) 
=
(L^2( \lambda; \R), 
\langle \cdot, \cdot 
\rangle_{L^2( \lambda;\R)} $,
$\left \| \cdot \right \|_{L^2( \lambda; \R)}
) $,
let       
$ ( e_n )_{ n \in \N } \subseteq H $
satisfy
for every
$ n \in \N $ 
that  
$ e_n = [ ( \sqrt{2} \sin( n \pi x ) )_{x \in (0,1) } ]_{ \lambda, \B(\R ) } $,
assume that
$ \H = \{ e_n \colon n \in \N \} $,
assume for every $ n \in \N $ that 
$ \values_{e_n } = - c_0 \pi^2 n^2 $,  
%
%
%
%
for every  
$ v \in W^{1,2}( (0,1), \R ) $  
let
$ \partial v    
\in H $  
satisfy for every
$ \varphi  
\in 
\mathcal{C}_{cpt}^\infty( (0,1), \R ) $ 
that 
$ \langle \partial  v, 
[ \varphi ]_{ \lambda, \mathcal{B}(\R )} \rangle_{H}
=
-
\langle v , 
[ \varphi' ]_{ \lambda, \mathcal{B}(\R )}
\rangle_{H} $,
and
let  
$ F \colon H_{\nicefrac{1}{2}} \to H $
be the function which
satisfies
for every
$ w \in H_{ \nicefrac{1}{2} } $ 
that
$ F(w) = 
c_1
w \partial w 
$.
\end{setting}
Note that for every $ s\in[0,\infty) $,  $ p\in[1,\infty) $ it holds that  
$ (W^{s, p}((0,1),\R), 
\left\| \cdot \right\|_{W^{s, p}((0,1),\R)}) $ 
is the Sobolev-Slobodeckij space with smoothness parameter $ s $ and integrability parameter $ p $ 
of equivalence classes
of $ \B( (0,1) ) / \B(\R) $-measurable
functions.
%
%
%
\subsection{Auxiliary results on Sobolev and interpolation spaces}
\label{subsection:Sobolev}
In this subsection we recall some elementary
 properties of the involved Sobolev
and interpolation spaces.
Lemmas~\ref{Lemma:trivial}--\ref{lemma:Equivalence of H1 norm}, 
Lemma~\ref{lemma:NormEquivalency}
(cf., e.g., Fujiwara~\cite{Fujiwara1967}),
Lemmas~\ref{lemma:Sobolev_multiplication}--\ref{lemma:InfEstimate},
Lemma~\ref{lemma:Gagliardo}
(cf., e.g., Brezis~\cite[Exercise~8.15 and~(42)
in the section 
\textit{Comments on Chapter 8}]{Brezis2011}
and Nirenberg~\cite{Nirenberg1959}),
and
Lemma~\ref{lemma:Gagliardo2}
(see, e.g., Sell \& You~\cite[Theorem B.2]{SellYou2002})   
below
are used for the regularity analysis of the  
considered 
nonlinearity 
in Subsection~\ref{subsection:Nonlinearity} below.
\begin{lemma}
\label{Lemma:trivial}
Assume Setting~\ref{setting:Examples}.
Then it holds
for every
$ \rho \in [ \nicefrac{1}{2}, \infty) $
that
$
\sum\nolimits_{ h \in \H }
	| \values_h |^{-2 \rho} 
	\leq \nicefrac{ | c_0 |^{ - 2 \rho } }{ 6 } $, 
$
\sup_{h \in \H }\| \partial h 
	\|_H
	| \values_{h} |^{- \rho} 
	\leq | c_0 |^{ - \rho } $,
	and
	$ \sup\nolimits_{h\in \H} \| h
	\|_{L^\infty( \lambda; \R)}
	= \sqrt{2} $.
\end{lemma}   
\begin{proof}[Proof of Lemma~\ref{Lemma:trivial}]
%
%
First, observe that 
	\begin{equation} 
	\begin{split} 
	\label{eq:sum_conv}
	\sum\nolimits_{ h \in \H }
	| \values_h |^{-2 \rho} 
	& =
	\sum\nolimits_{n\in \N}
	| c_0 \pi^2 n^2 |^{-2 \rho}
	=
	| c_0 |^{ - 2 \rho }
	\pi^{-4 \rho}
	\sum\nolimits_{n \in \N}
	 n^{-4 \rho}
	 \\
	 &
	 \leq
	 | c_0 |^{ - 2 \rho }
	 \pi^{-2}
	 \sum\nolimits_{n \in \N}
	 n^{-2}
	 =
	 | c_0 |^{ - 2 \rho }
	 \pi^{-2}
	 \tfrac{ \pi^2}{6}
	 .
	 \end{split}
	\end{equation}
	Moreover,
	note that for every
	$ n \in \N $
	it holds that
	\begin{equation}
	\begin{split}
	\label{eq:eig_decrease}
	\| \partial e_n 
	\|_H
	| \values_{e_n} |^{- \rho} 
	&
	=
	\| [ ( \pi n \sqrt{2} \cos( n \pi  x )  )_{x\in (0,1)} ]_{ \lambda, \mathcal{B}(\R)} \|_H
	| c_0 \pi^2 n^2 |^{-\rho}
	\\
	&
	=
	\pi n | c_0 \pi^2 n^2 |^{-\rho}
	=
	\tfrac{ | c_0 |^{ - \rho } }{ (\pi n)^{2 \rho - 1 } } 
	\leq
	| c_0 |^{ - \rho }
	.
	\end{split}
	\end{equation}
	In addition,
	observe that for every 
	$ n \in \N $, 
	$ x \in (0,1) $ 
	it holds
	that
	\begin{equation}
	| \sqrt{2} \sin( \pi n x ) | \leq \sqrt{2} .
	\end{equation}
This completes the proof of Lemma~\ref{Lemma:trivial}.
\end{proof}
\begin{lemma}
	\label{lemma:local_lipA}
	%
	Assume Setting~\ref{setting:Examples}.
	Then
	\begin{enumerate}[(i)]
	\item \label{item:01 A}
	it holds  		
	that 
	$ W_0^{1,2}((0,1), \R) \subseteq H_{ \nicefrac{1}{2} } $
	continuously,
	\item \label{item:01 B}
	it holds  		
	that
	$ H_{ \nicefrac{1}{2} } \subseteq W_0^{1,2}((0,1), \R) $ 
	continuously,
	\item \label{item:01b}
	it holds that $ W_0^{1,2}((0,1), \R) \subseteq L^\infty( \lambda; \R ) $
	continuously, 
	\item \label{item:norm_estimate}
	it holds for every 
	$ v \in H_{ \nicefrac{1}{2} } $ 
	that
	$ \| \partial v \|_H
	= 
	| c_0 |^{ - \nicefrac{1}{2} }
	\| v \|_{H_{\nicefrac{1}{2}}} $,
	and
	\item \label{item:norm_estimate_infty}
	it holds for every 
	$ v \in H_{ \nicefrac{1}{2} } $ 
	that
	$ \| v \|_{L^\infty( \lambda; \R ) } 
	\leq
	| 3 c_0 |^{ - \nicefrac{1}{2} }
	\| v \|_{H_{\nicefrac{1}{2}}} $.
	\end{enumerate}
\end{lemma}
\begin{proof}[Proof of Lemma~\ref{lemma:local_lipA}]
	Note that, e.g.,
	Lunardi~\cite[Example~4.34]{Lunardi2018}
	ensures that
	\begin{equation} 
	\label{eq:Equal space} 
	H_{\nicefrac{1}{2}} = W_0^{1,2}((0,1), \R) 
	.
	\end{equation} 
	This,
	the fact that
	for every 
	$ v \in W^{1,2}((0,1), \R) $
	it holds that
	\begin{equation}  
	\label{eq:Again something}
	\| v \|_{ W^{1,2}((0,1), \R) }^2 
	=
	\| v \|_H^2
	+
	\| \partial v \|_H^2
	,
	\end{equation}
	and 
	the fact that
	for every 
	$ v \in H_{ \nicefrac{1}{2}} $
	it holds that
	\begin{equation}
	\label{eq:Needed later}
	\| v \|_{H_{ \nicefrac{1}{2} } } = \sqrt{c_0} \| \partial v \|_H 
	\end{equation} 
	(see, e.g., \cite[Lemma~6.1]{JentzenSalimovaWelti2019})
	establish item~\eqref{item:01 A}.
	Moreover, observe 
	that~\eqref{eq:Equal space}--\eqref{eq:Needed later}
	and
	Poincar\'e's inequality
	(see,
	e.g., Brezis~\cite[Proposition~8.13]{Brezis2011}) 
	show
	item~\eqref{item:01 B}. 
	Next note that 
	Lemma~\ref{Lemma:trivial}
	(with $ \rho = \nicefrac{1}{2} $
	in the notation of Lemma~\ref{Lemma:trivial})
	and, 
	e.g., \cite[Lemma~4.3]{JentzenPusnik2018Published}
	(with
	$ d = 1 $,
	$ \H = \H $,
	$ \rho = \nicefrac{1}{2} $,
	$ v = v $
	for 
	$ v \in H_{ \nicefrac{1}{2} } $
	in the notation of~\cite[Lemma~4.3]{JentzenPusnik2018Published})
	prove that for every
	$ v \in H_{  \nicefrac{1}{2} } $ it holds that
	%
	%
	%
	%
	\begin{equation}
	\begin{split}
	\label{eq:estimate_infiniti} 
	\| v \|_{L^\infty( \lambda; \R )} 
	&
	\leq
	\| v \|_{H_{ \nicefrac{1}{2} }}
	\bigg(
	\sup_{ h \in \H } 
	\| h
	\|_{L^\infty( \lambda; \R )}
	\bigg) 
	\Bigg[ 
	\sum_{h \in \H }
	| \values_h |^{- 1 }
	\Bigg]^{\nicefrac{1}{2}}
	\leq
	| 3 c_0 |^{ - \nicefrac{1}{2} }
	\| v \|_{ H_{ \nicefrac{1}{2} } }
	.
	\end{split}
	\end{equation} 
	This
	and	
	item~\eqref{item:01 A} 
	establish item~\eqref{item:01b}.
	Moreover, note that~\eqref{eq:Needed later}
	shows item~\eqref{item:norm_estimate}.
	In addition, observe that~\eqref{eq:estimate_infiniti} 
	establishes item~\eqref{item:norm_estimate_infty}.
	%
	%
	%
	%
	%
	%
	The proof of Lemma~\ref{lemma:local_lipA}
	is thus completed.
\end{proof}
\begin{lemma}
	\label{lemma:int_parts}
	Assume Setting~\ref{setting:Examples}
	and let  
	$ u \in W_0^{1,2}( (0,1) , \R) $,  
	$ v \in W^{1,2}( (0,1) , \R) $. 
	Then it holds that 
	\begin{equation}
	\left< \partial u, v \right>_H 
	=
	\!
	- 
	\!
	\left< u, \partial v \right>_H 
	.
	\end{equation}
\end{lemma}
\begin{proof}[Proof of Lemma~\ref{lemma:int_parts}]
	%
	Throughout this proof let
	$ ( \mathbf{u}_n )_{n \in \N} 
	\subseteq  
	\mathcal{C}_c^\infty( \R, \R) $,
	$ ( \mathbf{v}_n )_{n \in \N} 
	\subseteq  
	\mathcal{C}_c^\infty( \R, \R) $, 
	$ ( u_n )_{n \in \N} \subseteq 
	W_0^{1,2}((0,1), \R ) $,
	$ ( v_n )_{n \in \N} \subseteq 
	W^{1,2}((0,1), \R ) $
	satisfy for every 
	$ n \in \N $,
	$ x \in ( (-\infty, 0] \cup [1, \infty) ) $
	that
	$ \mathbf{u}_n(x) = 0 $,
	$ u_n = 
	[ \mathbf{u}_n |_{(0,1)}]_{\lambda, \mathcal{B}(\R)} $,
	$ v_n = 
	[ \mathbf{v}_n |_{(0,1)}]_{\lambda, \mathcal{B}(\R)} $,
	and
	$ \limsup_{m \to \infty} 
	( 
	\| u  - 
	u_m \|_{W^{1,2}(  (0,1) , \R)} 
	+
	\| v  - 
	v_m \|_{W^{1,2}( (0,1), \R)} 
	) = 0 $.
	Observe that
	integration by parts
	and the fact that
		for every
	$ n \in \N $	 
	it holds that
	$ \mathbf{u}_n(0) =  \mathbf{u}_n(1) = 0  $
	demonstrate that
	\begin{equation}
	\begin{split}
	& 
	\langle \partial u, v \rangle_H
	=
	\lim_{n \to \infty}
	\langle \partial u_n, v \rangle_H
	= 
	\lim_{n \to \infty}
	\left(
	\lim_{m \to \infty}
	\langle \partial u_n, v_m \rangle_H
	\right)
	\\
	&
	= 
	\lim_{n \to \infty}
	\left(
	\lim_{m \to \infty}
	\int_{(0,1)} 
	( 
	\mathbf{u}_n
	)'(x) 
	\,
	\mathbf{v}_m(x)
	\, dx
	\right)
	=
	-
	\lim_{n \to \infty}
	\left(
	\lim_{m \to \infty}
	\int_{(0,1)}
	\mathbf{u}_n(x)
	\,
	( 
	\mathbf{v}_m
	)'(x) 
	\, dx
	\right)
	\\
	&
	=
	-
	\lim_{n \to \infty}
	\left(
	\lim_{m \to \infty}
	\langle  u_n, \partial v_m \rangle_H
	\right)
	=
	-
	\lim_{n \to \infty} 
	\langle  u_n, \partial v \rangle_H  
	=
	-
	\langle  u, \partial v \rangle_H 
	.
	\end{split}
	\end{equation}
	The proof of Lemma~\ref{lemma:int_parts}
	is thus completed.
\end{proof}
\begin{lemma}
	\label{lemma:Equivalence of H1 norm}
	Assume Setting~\ref{setting:Examples}.
	Then 
	\begin{enumerate}[(i)]
	\item \label{item:continuous embedding H1} it holds that
	$ H_1 \subseteq W^{2,2}((0,1), \R) $
	continuously
	and
	\item \label{item:equivalent norms} it holds that
	\begin{equation} 
	\sup_{ v \in H_1\backslash \{ 0 \}  } 
	\bigg[ 
	\frac{ \| v \|_{H_1} }{ \| v \|_{ W^{2,2}( (0,1), \R ) }  }
	+
	\frac{ \| v \|_{ W^{2,2}( (0,1), \R ) }  }{ \| v \|_{H_1} }
	\bigg]
	<
	\infty 
	.
	\end{equation} 
\end{enumerate}
	%
\end{lemma}
\begin{proof}[Proof of Lemma~\ref{lemma:Equivalence of H1 norm}]
First, observe that the fact that
\begin{equation} 
\label{eq:Interpo2}
D ( - A ) 
= 
\big( W_0^{1,2}( (0,1), \R) \big) \cap \big( W^{2,2}( (0,1), \R ) \big)
\end{equation}
(cf., e.g., 
Lunardi~\cite[Example~4.34]{Lunardi2018}
and
Sell \& You~\cite[Section 3.8.1]{SellYou2002})
and 
the fact that
$ D(-A) = H_1 $
prove that
\begin{equation} 
\label{eq:Initial step}
H_1 \subseteq W^{2,2}( (0,1), \R )
.
\end{equation} 
%
%
%
%
Hence, we obtain
that
for every $ v \in H_1 $
it holds that
$ \partial v \in W^{1,2}((0,1), \R) $.
The fact that
	for every
$ n \in \N $
it holds that
$ e_n \in W_0^{1,2}((0,1), \R) $
and
Lemma~\ref{lemma:int_parts} 
(with
$ u = e_n $,
$ v = \partial v $
for 
$ n \in \N $,
$ v \in H_1 $
in the notation of 
Lemma~\ref{lemma:int_parts})
therefore prove 
that for every $ v \in H_1 $ 
it holds that
\begin{equation}
\begin{split}
\label{eq:PerPartes 1}
\sum_{ n = 1 }^\infty 
| \langle e_n, \partial^2 v \rangle_H |^2
=
\sum_{ n = 1 }^\infty 
| - \langle \partial e_n, \partial v \rangle_H |^2
.
\end{split}
\end{equation} 
Furthermore,
note that  
item~\eqref{item:01 B} 
of Lemma~\ref{lemma:local_lipA}
assures that for every
$ v \in H_1 $
it holds that
$ v \in W_0^{1,2}((0,1), \R) $.
Combining~\eqref{eq:PerPartes 1},
the fact that 
	for every
$ n \in \N $
it holds that
$ \partial e_n \in W^{1,2}((0,1), \R) $,
\eqref{eq:Initial step},
and
Lemma~\ref{lemma:int_parts} 
(with
$ u = v $,
$ v = \partial e_n $
for 
$ n \in \N $, 
$ v \in H_1 $
in the notation of 
Lemma~\ref{lemma:int_parts}) 
hence shows that
for every $ v \in H_1 $ it holds that
\begin{equation}
\begin{split} 
\sum_{ n = 1 }^\infty 
| \langle e_n, \partial^2 v \rangle_H |^2
=
\sum_{ n = 1 }^\infty 
| \langle \partial^2 e_n, v \rangle_H |^2
=
\sum_{ n = 1 }^\infty
|\pi n|^4
| \langle e_n, v \rangle_H |^2
=
\tfrac{ 1 }{ |c_0|^2 }
\| v \|_{H_1}^2
< \infty.
\end{split}
\end{equation}
This proves that for every 
$ v \in H_1 $
it holds that
$ \partial^2 v \in H $
and
\begin{equation} 
\label{eq:Main identity}
\| v \|_{H_1}
=
c_0
\| \partial^2 v \|_H 
.
\end{equation} 
%
%
%
%
%
%
The fact that
	for every
$ v \in W^{2,2}((0,1), \R ) $
it holds that
$ \| v \|_{W^{2,2}((0,1), \R )}^2
=
\| v \|_H^2
+
\| \partial v \|_H^2
+
\| \partial^2 v \|_H^2 $
and~\eqref{eq:Initial step}
hence
ensure that for every $ v \in H_1 $
it holds that
\begin{equation} 
\label{eq:Main estimate}
\| v \|_{H_1} 
= 
c_0 
\| \partial^2 v \|_H
\leq 
c_0 
\| v \|_{W^{2,2}((0,1), \R)}
.
\end{equation} 
Next note that
item~\eqref{item:01 B} 
of Lemma~\ref{lemma:local_lipA}
and
Poincar\'e's inequality
(see,
e.g., Brezis~\cite[Proposition~8.13]{Brezis2011}) 
imply that 
there exists $ C \in (0, \infty) $
such that
for every
$ v \in H_{\nicefrac{1}{2}} $ it holds that
$ 
\| v \|_{W^{1,2}((0,1), \R)} \leq C \| \partial v \|_H 
$.
Combining this,
\eqref{eq:Initial step}, 
and~\eqref{eq:Main identity} 
proves that
there exists $ C \in (0, \infty) $
such that
for every 
$ v \in H_1 $ it holds that
\begin{equation}
\begin{split}
\| v \|_{W^{2,2}((0,1), \R)}^2
&
=
\| v \|_{W^{1,2}((0,1), \R )}^2
+
\| \partial^2 v \|_H^2
\leq
C^2 \| \partial v\|_H^2 
+
\frac{1}{|c_0|^2} 
\|  v \|_{H_1}^2
. 
\end{split}
\end{equation}
Item~\eqref{item:norm_estimate}
of Lemma~\ref{lemma:local_lipA}
hence shows that 
there exists $ C \in (0, \infty) $
such that 
for every $ v \in H_1 $ it holds that
$ v \in W^{2,2}((0,1), \R) $
and
\begin{equation}
\begin{split}
\| v \|_{W^{2,2}((0,1), \R)}^2
& 
\leq
\frac{ C^2 }{ c_0 } 
\| v\|_{H_{ \nicefrac{1}{2} }}^2 
+
\frac{1}{|c_0|^2} 
\| v \|_{ H_1 }^2
\leq 
\bigg[ 
\frac{ C^2 }{ |c_0|^2 }  
+
\frac{1}{|c_0|^2} 
\bigg] 
\| v \|_{ H_1 }^2
. 
\end{split}
\end{equation}
This 
establishes item~\eqref{item:continuous embedding H1}.
%
%
%
Moreover, observe that item~\eqref{item:continuous embedding H1}
and~\eqref{eq:Main estimate} 
imply
item~\eqref{item:equivalent norms}.
%
%
The proof of Lemma~\ref{lemma:Equivalence of H1 norm}
is thus completed.
\end{proof} 
\begin{lemma}
	\label{lemma:NormEquivalency}
	Assume Setting~\ref{setting:Examples}.
	Then 
	\begin{enumerate}[(i)]
    \item \label{item:Inclusion}
	it holds for every $ s \in [0, 1] $
	that
	$ H_s \subseteq W^{2s, 2 }( (0,1), \R ) $
	continuously,
	\item \label{item:Equivalence}
	it holds for every
	$ s \in [0, \nicefrac{1}{2}] \backslash \{ \nicefrac{1}{4} \} $
	that 
	$ H_s \subseteq W_0^{2 s, 2 }( (0,1), \R) $
	continuously,
	and
	\item \label{item:Equivalence 2}
	it holds for every
	$ s \in [0, \nicefrac{1}{2}] \backslash \{ \nicefrac{1}{4} \} $
	that
	$ W_0^{2 s, 2 }( (0,1), \R) \subseteq H_s $
	continuously.
	\end{enumerate}
\end{lemma}
\begin{proof}[Proof of Lemma~\ref{lemma:NormEquivalency}]
\sloppy 
Throughout this proof consider the notation in
Triebel~\cite[Section~1.3.2 on page~24]{Triebel1978}
(cf., e.g., Lunardi~\cite[Definition~1.2]{Lunardi2018}).
Note that
item~\eqref{item:continuous embedding H1} of 
Lemma~\ref{lemma:Equivalence of H1 norm}
ensures that
\begin{equation}
\label{eq:H1_Continuously}
H_1
\subseteq 
W^{2,2}( (0,1), \R)
\end{equation}
continuously.
Furthermore,
observe that,
e.g.,
Triebel~\cite[the theorem in Section~1.18.10 on page~142]{Triebel1978}
(cf., e.g., Lunardi~\cite[Theorem~4.36]{Lunardi2018})
and the fact that 
$ \forall \, s \in [0, \infty) \colon 
( D( (-A)^s ), \| (-A)^s \left( \cdot \right) \|_H ) 
= 
( H_s, \left\| \cdot \right\|_{H_s} ) $
prove that for every
$ s \in (0, 1) $ it holds that
\begin{equation}
\label{eq:Interpo4}
(H, H_1 )_{ s, 2 }
=
( H, D(-A) )_{ s, 2 }
=
D ( (-A)^s )
=
H_s  
\end{equation} 
and
\begin{equation} 
\label{eq:Interpo4 equivalent}
\sup_{ x \in H_s \backslash \{ 0 \}  }
\bigg(   
\frac{ \| x \|_{ ( H, H_1)_{s,2} }  }
{ \| x \|_{ H_s } }
+
\frac{ \| x \|_{ H_s } }
{ \| x \|_{ ( H, H_1)_{s,2} }  }
\bigg) 
<
\infty 
.
\end{equation}
%
%
%
%
%
%
This,
\eqref{eq:H1_Continuously},
and, e.g., 
Lunardi~\cite[Theorem~1.6]{Lunardi2018}
imply that for every
$ s \in (0,1) $ it holds that
\begin{equation}
\begin{split} 
H_s 
%
\subseteq 
( H,
W^{2,2}( (0,1), \R )
)_{ s, 2 } 
\end{split} 
\end{equation} 
continuously.
The fact that
for every $ s \in (0,1) $ it holds that
\begin{equation} 
( H,
W^{2,2}( (0,1), \R )
)_{ s, 2 } 
\subseteq
W^{ 2 s, 2}( (0,1), \R ) 
\end{equation}
continuously
(cf., e.g., 
Triebel~\cite[Definition~1 
in Section~4.2.1 
on page~310,
Theorem~1 in Section~4.3.1 
on page~317, 
item~(a) in Theorem~1 
in Section~4.4.2
on page~323,
 and 
Remark~2 in Section~4.4.2 
on page~324]{Triebel1978}) 
hence establishes item~\eqref{item:Inclusion}.
%
%
%
%
%
Moreover,
note that,
e.g., 
Triebel~\cite[the theorem in Section~1.18.10 on page~142]{Triebel1978}
(cf., e.g.,
Lunardi~\cite[Theorem~4.36]{Lunardi2018}) 
and the fact that 
$ \forall \, s \in [0, \infty) \colon 
( D( (-A)^s ), \| (-A)^s \left( \cdot \right) \|_H ) 
= 
( H_s, \left\| \cdot \right\|_{H_s} ) $
prove that for every
$ s \in (0, 1) $ it holds that
\begin{equation}
\label{eq:Interpo4b}
(H, H_{\nicefrac{1}{2}} )_{ s, 2 }
=
( H, D( (-A)^{\nicefrac{1}{2}} ) )_{ s, 2 }
=
D ( (-A)^{\nicefrac{s}{2}} )
=
H_{\nicefrac{s}{2}}  
\end{equation} 
and
\begin{equation} 
\label{eq:Interpo4 equivalent b}
\sup_{ x \in H_s \backslash \{ 0 \}  }
\bigg(   
\frac{ \| x \|_{ ( H, H_{ \nicefrac{1}{2} })_{s,2} }  }
{ \| x \|_{ H_{ \nicefrac{s}{2} } } }
+
\frac{ \| x \|_{ H_{ \nicefrac{s}{2} } } }
{ \| x \|_{ ( H, H_{ \nicefrac{1}{2} } )_{s,2} }  }
\bigg) 
<
\infty 
.
\end{equation}
The fact that
for every
$  s \in (0,1)
\backslash \{ \nicefrac{1}{2} \} $
it holds that
\begin{equation} 
( H,
W_0^{1,2}( (0,1), \R )
)_{ s, 2 } 
=
W_0^{ s, 2}( (0,1), \R ), 
%
%
\end{equation}
(cf., e.g., 
Triebel~\cite[Definition~1 and Definition~2 in Section~4.2.1 on page~310,
the definition in Section~4.3.2 on page~317,
item~(c) in Theorem~1 
and
Theorem~2 in Section~4.3.2 on page~318,
item~(a) in Theorem~1 
in Section~4.4.2
on page~323,
and
Remark~2 in Section~4.4.2 on page~324]{Triebel1978}),
%
%
%
%
items~\eqref{item:01 A} and~\eqref{item:01 B} 
of Lemma~\ref{lemma:local_lipA},  
and, e.g.,  
Lunardi~\cite[Theorem~1.6]{Lunardi2018} 
therefore
assure that for every
$ s \in (0,1) \backslash \{ \nicefrac{1}{2} \} $ it holds that 
\begin{equation}
H_{ \nicefrac{s}{2} }
=
( H, H_{ \nicefrac{1}{2} } )_{s, 2}
=
( H, W_0^{ 1, 2}( (0,1), \R )   )_{s, 2}
=
W_0^{ s, 2}( (0,1), \R )  
\end{equation}
and
\begin{equation} 
\sup_{ x \in H_{ \nicefrac{s}{2} } \backslash \{ 0 \}  }
\bigg(   
\frac{ \| x \|_{ H_{ \nicefrac{s}{2} } }  }
{ \| x \|_{W^{ s, 2}( (0,1), \R ) } }
+
\frac{ \| x \|_{ W^{ s, 2}( (0,1), \R )  }  }
{ \| x \|_{H_{ \nicefrac{s}{2} } } }
\bigg) 
<
\infty 
.
\end{equation}
This establishes items~\eqref{item:Equivalence}
and~\eqref{item:Equivalence 2}. 
The proof of Lemma~\ref{lemma:NormEquivalency}
is thus completed.
\end{proof} 
\begin{lemma} 
	\label{lemma:Sobolev_multiplication}
	Let 
	$ s \in [0,\infty) $,
	$ q, r \in [s,\infty) $ 
	satisfy  
	$ r + q - s > \nicefrac{1}{2} $.
	%
	%
	Then 
	\begin{enumerate}[(i)]
	\item \label{item:one more}
	it holds 
	for every 
	$ f \in W^{q, 2}( (0,1), \R )  $,
	$ g \in W^{r, 2}( (0,1), \R )  $ that
	$ f g \in W^{s, 2 }( ( 0, 1 ), \R ) $
	and  
	\item \label{item:two more} it holds that
	\begin{equation}
	\label{eq:ResultOfLemma}
	\sup_{
		f \in W^{q, 2}( (0,1), \R ) \backslash \{0\} 
	}
\,
\sup_{
		g \in W^{r, 2}( (0,1), \R ) \backslash \{0\}
	}
	\Bigg[
	\frac{ 
		\| f g \|_{ W^{s,2}( (0,1), \R ) }
	}
	{
		\| f \|_{ W^{q, 2}( (0,1), \R ) }
		\| g \|_{ W^{r, 2}( (0,1), \R ) }
	}
	\Bigg]
	< \infty
	.	
	\end{equation} 
	\end{enumerate}
\end{lemma}
\begin{proof}[Proof of Lemma~\ref{lemma:Sobolev_multiplication}]
	Observe that, e.g., 
	Behzadan \& Holst~\cite[Theorem~7.5]{BehzadanHolst2015} 
	(with 
	$ n = 1 $,
	$ \Omega = (0,1) $,
	$ s = s $,
	$ p = 2 $,
	$ s_1 = q $,
	$ s_2 = r $,
	$ p_1 = 2 $,
	$ p_2 = 2 $
	in the notation of  
	Behzadan \& Holst~\cite[Theorem~7.5]{BehzadanHolst2015})
	establishes
	items~\eqref{item:one more}
	and~\eqref{item:two more}.
	The proof of Lemma~\ref{lemma:Sobolev_multiplication}
	is thus completed.
\end{proof} 
\begin{lemma}
	\label{lemma:Identity2}
	%
	Assume Setting~\ref{setting:Examples}.
	Then  
	\begin{enumerate}[(i)]  
	\item \label{item:Existence of extension}
	there exists a unique bounded linear function
	$ \bar \partial \colon H \to H_{-\nicefrac{1}{2}} $
	which satisfies for every $ v \in W^{1,2}((0,1), \R ) $
	that
	$ \bar \partial v = \partial v $
	and 
	\item \label{item:Norm estimate of extension}
	it holds that 
	$ \| \bar \partial \|_{ L(H, H_{-\nicefrac{1}{2} } ) }
	\leq | c_0 |^{ - \nicefrac{1}{2} } $.
	\end{enumerate}
\end{lemma}
\begin{proof}[Proof of Lemma~\ref{lemma:Identity2}]
	Observe that 
Lemma~\ref{lemma:int_parts}
and
	items~\eqref{item:01 A},
	\eqref{item:01 B},
	and~\eqref{item:norm_estimate} of Lemma~\ref{lemma:local_lipA}  
	show that for every
    $ v \in W^{1,2}((0,1), \R) $	
	it holds that 
	\begin{equation}
	\begin{split} 
	\label{eq:first H12}
	\| \partial v \|_{ H_{ - \nicefrac{1}{2} } }
	&=
	\sup_{ u \in (  H_{ \nicefrac{1}{2} } \backslash \{0\} ) }
	\tfrac{ | \langle \partial v, u \rangle_H | }{
	\| u \|_{ H_{ \nicefrac{1}{2} } }
	}
    =
    \sup_{ u \in (   W_0^{1,2}((0,1), \R) \backslash \{0\} ) }
    \tfrac{ | \langle \partial v, u \rangle_H | }{
    	\| u \|_{ H_{ \nicefrac{1}{2} } }
    }
	=
	\sup_{ u \in (   W_0^{1,2}((0,1), \R) \backslash \{0\} ) }
	\tfrac{ | \langle v, \partial u \rangle_H | }{
		\| u \|_{ H_{ \nicefrac{1}{2} } }
	}
\\
&
\leq 
\sup_{ u \in (   W_0^{1,2}((0,1), \R) \backslash \{0\} ) }
\tfrac{ \| v \|_H  \| \partial u \|_H }{
	\| u \|_{ H_{ \nicefrac{1}{2} } }
}
=
| c_0 |^{ - \nicefrac{1}{2} }
\sup_{ u \in (  H_{ \nicefrac{1}{2} } \backslash \{0\} ) }
\tfrac{ \| v \|_H  \| u \|_{ H_{ \nicefrac{1}{2} } } }{
	\| u \|_{ H_{ \nicefrac{1}{2} } }
}
=
| c_0 |^{ - \nicefrac{1}{2} }
\| v \|_H
	.
	\end{split}
	\end{equation}
The fact that
$ W^{1,2}((0,1), \R ) \subseteq H $
densely 
therefore establishes 
    items~\eqref{item:Existence of extension}
    and~\eqref{item:Norm estimate of extension}.
	%
	The proof of Lemma~\ref{lemma:Identity2}
	is thus completed.
\end{proof}
%
%
%
\begin{lemma}
	\label{lemma:DerivativeEstimate}
	%
	Assume Setting~\ref{setting:Examples}
	and let
	$ \alpha \in [0, \nicefrac{1}{2} ] $.
	Then
	\begin{equation} 
	\label{eq:Prove equivalency}
	\sup\nolimits_{ v \in W^{1,2}((0,1), \R ) \backslash \{ 0 \} }
	\tfrac{ \| \partial v \|_{ H_{ - \alpha } }
	} {
		\| v \|_{ W^{ 1- 2 \alpha, 2}( (0,1), \R ) }
	} < \infty 
	.
	\end{equation} 
	%
\end{lemma}
\begin{proof}[Proof of Lemma~\ref{lemma:DerivativeEstimate}]
	Throughout this proof consider the notation in
	Triebel~\cite[Section~1.3.2 on page~24]{Triebel1978}
	(cf., e.g., Lunardi~\cite[Definition~1.2]{Lunardi2018}) 
	and  
	let
	$ \bar \partial \colon H \to H_{ - \nicefrac{1}{2} } $
	be the continuous linear function
	which satisfies for every
	$ v \in W^{1,2}( (0,1), \R ) $
	that
	$ \bar \partial v = \partial v $
	(cf.\ item~\eqref{item:Existence of extension} of Lemma~\ref{lemma:Identity2}).
	%
	%
	%
	%
	Observe that,
	e.g., 
	Triebel~\cite[the theorem in Section~1.18.10 on page~142]{Triebel1978}
	(cf., e.g.,
	Lunardi~\cite[Theorem~4.36]{Lunardi2018}) 
	and the fact that 
	$ \forall \, s \in [0, \infty) \colon 
	( D( (-A)^s ), \| (-A)^s \left( \cdot \right) \|_H ) 
	= 
	( H_s, \left\| \cdot \right\|_{H_s} ) $
	prove that for every
	$ s \in (0, 1) $ 
	it holds that
	\begin{equation}
	\label{eq:Interpolation Hs2}
	(H, H_{\nicefrac{1}{2}} )_{ s, 2 }
	=
	( H, D( ( -A )^{\nicefrac{1}{2}} ) )_{ s, 2 }
	=
	D ( (-A)^{\nicefrac{s}{2}} )
	=
	H_{\nicefrac{s}{2}}
	\end{equation} 
	and
	\begin{equation} 
	\label{eq:Interpolation Hs2 equivalent}
	\sup_{ x \in H_1 \backslash \{ 0 \}  }
	\bigg(   
	\frac{ \| x \|_{ ( H, H_{\nicefrac{1}{2}})_{s,2} }  }
	{ \| x \|_{H_{\nicefrac{s}{2}}} }
	+
	\frac{ \| x \|_{H_{\nicefrac{s}{2}}} }
	{ \| x \|_{ ( H, H_{\nicefrac{1}{2}})_{s,2} }  }
	\bigg) 
	<
	\infty 
	.
	\end{equation}
	The fact that
	for every $ r \in [0, \infty) $
	it holds that 
	$ ( H_r )' $
	and 
	$ H_{-r} $ are isometrically isomorphic
	and, e.g., 
	Triebel~\cite[item~(b) of
	the theorem 
	in Section~1.3.3
	on page~25
	and the theorem in Section~1.11.2
	on page~69]{Triebel1978} 
	(cf., e.g., 
	Lunardi~\cite[Theorem~1.18]{Lunardi2018})
hence 
	imply that for every
	$ s \in (0,1) $
	it holds that
	\begin{equation} 
	\label{eq:interpolation 1}
	  ( H_{ - \nicefrac{1}{2} }, H )_{s, 2} 
	  =
	  H_{ \nicefrac{ ( s - 1 ) }{ 2 } }
	\end{equation}
	and
	\begin{equation}
	\label{eq:interpolation 1 norms}
	\sup_{ x \in H_{ \nicefrac{(s-1)}{2} } \backslash \{0\}  }
	\bigg( 
	\frac{ \| x \|_{ H_{ \nicefrac{(s-1)}{2} }    } }
	{ \| x \|_{ ( H_{ - \nicefrac{1}{2} }, H )_{s,2} } }
	+
	\frac{ \| x \|_{ ( H_{ - \nicefrac{1}{2} }, H )_{s,2} } }
	{ \| x \|_{ H_{ \nicefrac{(s-1)}{2} }    } } 
	\bigg)
	<
	\infty.
	\end{equation}
	In addition, note that, e.g., 
	Triebel~\cite[Definition~1 
	in Section~4.2.1 
	on page~310,
	Theorem~1 in Section~4.3.1 
	on page~317, 
	item~(a) in Theorem~1 
	in Section~4.4.2
	on page~323,
	and 
	Remark~2 in Section~4.4.2 
	on page~324]{Triebel1978}  
	ensures that 
	for every $ s \in (0,1) $
	it holds that
	\begin{equation}
	\label{eq:Interpolation norm estimate spaces}
	( H,
	W^{1,2}( (0,1), \R )
	)_{ s, 2 }  
	=
	W^{ s, 2}( (0,1), \R )
	\end{equation}
	and
	\begin{equation}
	\label{eq:Interpolation norm estimate} 
	\sup_{ x \in 
		W^{ s, 2}( (0,1), \R ) \backslash \{0\}
	}
	\Bigg(
	\frac{
	\| x \|_{ W^{ s, 2}( (0,1), \R ) }
	}{
	\| x \|_{ ( H,
		W^{1,2}( (0,1), \R )
		)_{ s, 2 } }
	}
	+
	\frac{
		\| x \|_{ ( H,
			W^{1,2}( (0,1), \R )
			)_{ s, 2 } }
	}
{
		\| x \|_{ W^{ s, 2}( (0,1), \R ) }
	}
	\Bigg)
	< \infty.
	\end{equation} 
	Furthermore,
	observe that 
	item~\eqref{item:Norm estimate of extension}
	of
	Lemma~\ref{lemma:Identity2} 
	ensures that for every
	$ v \in H $
	it holds
	that
	\begin{equation} 
	\| \bar \partial v \|_{ H_{ - \nicefrac{1}{2} } } 
	\leq
	| c_0 |^{ - \nicefrac{1}{2} }
	\| v \|_H 
	.
	\end{equation}  
	Combining this,  
	the fact that
	for every $ v \in  
	W^{1,2}( (0,1), \R ) $
	it holds that
	\begin{equation} 
	\| \partial v \|_H
	\leq
	\| v \|_{ W^{1,2}( (0,1), \R ) }
	,
	\end{equation}  
	\eqref{eq:interpolation 1}--\eqref{eq:Interpolation norm estimate}, 
	and, e.g, 
	Lunardi~\cite[Theorem~1.6]{Lunardi2018}
	establishes~\eqref{eq:Prove equivalency}.
 The proof of
 Lemma~\ref{lemma:DerivativeEstimate}
 is thus completed.  
\end{proof}
\begin{lemma}
	\label{lemma:InfEstimate}
	Assume Setting~\ref{setting:Examples} 
	and let
	$ \alpha \in ( \nicefrac{1}{4}, \infty ) $.
	Then it holds for every
	$ v \in H_{ \alpha + ( \nicefrac{1}{2} ) } $ 
	that
	\begin{equation}
	\begin{split} 
	&
	\| \partial v \|_{L^\infty(\lambda; \R) }
	\leq
	\sqrt{2} 
	| c_0 |^{ - \alpha - (\nicefrac{1}{2} ) }
	\| v \|_{ H_{ \alpha + ( \nicefrac{1}{2} )  } }
	\sqrt{ 
		\sum_{n=1}^\infty 
		| \pi n|^{-4\alpha}	
	}
	.
	\end{split} 
	\end{equation}
\end{lemma}
\begin{proof}[Proof of Lemma~\ref{lemma:InfEstimate}]
%
%
%
Note that the fact that 
$ \forall \, v \in H_{ \nicefrac{1}{2} } \colon 
\sum_{n=1}^\infty
| \values_{e_n} | 
\,
| \langle e_n, v \rangle_H |^2
=
\|(-A)^{\nicefrac{1}{2}} v \|_H^2
=
\|
v
\|_{H_{\nicefrac{1}{2}}}^2
 < \infty $ 
shows that
for every 
$ v \in H_{ \nicefrac{1}{2} } $ 
it holds that
\begin{equation}
\begin{split}
\label{eq:H12Convergence}
\limsup_{ N \to \infty }
\Big\|
v
-
\sum_{ n = 1 }^N
\langle e_n, v \rangle_H
e_n
\Big\|_{ H_{ \nicefrac{1}{2} } }^2
=
\limsup_{N\to\infty}
\bigg[ 
\sum_{n=N+1}^\infty
|\values_{e_n}|
|\langle e_n, v \rangle_H|^2
\bigg] 
=
0.
\end{split}
\end{equation}
In addition, observe that
items~\eqref{item:01 B} 
and~\eqref{item:norm_estimate} 
of Lemma~\ref{lemma:local_lipA}
ensure that
$ ( H_{ \nicefrac{1}{2} } \ni u \mapsto \partial u \in H ) 
\in L ( H_{ \nicefrac{1}{2} }, H ) $.
Combining~\eqref{eq:H12Convergence}
and the Cauchy-Schwarz inequality 
hence
implies that
for every
$ v \in H_{ \alpha + ( \nicefrac{1}{2} ) } $
it holds that
\begin{equation}
\begin{split} 
&
\| \partial v \|_{L^\infty(\lambda; \R) }
=
\bigg\|
\partial 
\bigg(
\sum_{ n = 1 }^\infty
\langle e_n, v \rangle_H
e_n
\bigg)
\bigg\|_{L^\infty(\lambda; \R) }
=
\bigg\|
\sum_{ n = 1 }^\infty
\langle e_n, v \rangle_H
\partial e_n
\bigg\|_{L^\infty(\lambda; \R) }
\\
&
\leq
\sum_{ n = 1 }^\infty
| \langle e_n, v \rangle_H |
\|
\partial e_n
\|_{L^\infty(\lambda; \R) }
\\
&
\leq
\sup\nolimits_{ n \in \N }
\big(
| c_0 \pi^2 n^2 |^{ - \nicefrac{1}{2} }
\|
\partial e_n
\|_{L^\infty(\lambda; \R) }
\big)
\Bigg[ 
\sum_{ n = 1 }^\infty
| \langle e_n, v \rangle_H |
| c_0 \pi^2 n^2 |^{\nicefrac{1}{2}} 
\Bigg]
\\
&
\leq
\sup\nolimits_{ n \in \N }
\big(
| c_0 \pi^2 n^2 |^{ - \nicefrac{1}{2} }
\sqrt{2} n \pi 
\big)
\sqrt{ 
\sum_{ n = 1 }^\infty
| \langle e_n, v \rangle_H |^2
| c_0 \pi^2 n^2 |^{ 1 + 2 \alpha } 
}
\sqrt{ 
\sum_{n=1}^\infty 
| c_0 \pi^2 n^2 |^{-2 \alpha}	
}
\\
&
=
\sqrt{ 2 } 
| c_0 |^{ - \alpha - ( \nicefrac{1}{2} ) }
\| v \|_{ H_{ \alpha + ( \nicefrac{1}{2} ) } }
\sqrt{ 
	\sum_{n=1}^\infty 
	| \pi n|^{-4 \alpha}	
}
.
\end{split} 
\end{equation}
The proof of Lemma~\ref{lemma:InfEstimate}
is thus completed.
\end{proof}
\begin{lemma}
	\label{lemma:Gagliardo}
	%
	Let $ \lambda \colon \B((0, 1)) \to [0, 1] $ be the 
	Lebesgue-Borel measure on $ (0, 1) $
	and
	let 
	$ q, r \in [1, \infty) $, 
	$ \alpha \in (0,1) $
	satisfy
	$ \alpha ( \frac{1}{q} + 1 - \frac{1}{r} ) = \frac{1}{q} $.
	Then there exists 
	$ C \in (0, \infty) $
	such that for every
	$ u \in W^{1,r}_0( (0,1), \R ) $ 
	it holds that
	\begin{equation}
	\| u \|_{L^\infty( \lambda; \R  ) }
	\leq 
	C
	\| u \|_{ W^{1,r}( (0,1), \R ) }^\alpha 
	\| u \|_{L^q ( \lambda; \R) }^{1-\alpha}
	.
	\end{equation}
\end{lemma}
\begin{proof}[Proof of Lemma~\ref{lemma:Gagliardo}]
	Throughout this proof  
	let
	$ p \in \R $
	satisfy  
	$ q = p( \frac{1}{\alpha}-1) $, 
	for every function
	$ f \colon  (0,1) \to \R $
	let 
	$ [f]_{\lambda, \B(\R)} $
	be the set given by
	\begin{equation} 
	[f]_{\lambda, \B(\R)} =
	\left\{
	g \colon (0,1) \to \R  
	\colon
	\substack{  
	\bigg[
	\begin{aligned} 
	&
	(
	\exists\, D \in \B( (0,1) ) \colon
	[ 
	\lambda( D ) = 0 \text{ and } 
	\{ t \in (0,1) \colon
	f( t ) \neq g( t )
	\}
	\subseteq D
	]
	)
	\\
	&
	\text{ and }
	(\forall \, D \in \mathcal{B}(\R)\colon 
	g^{-1}(D) \in \mathcal{B}((0,1))
	)
	\end{aligned}
	\bigg]	
	}
	\right\},
	\end{equation} 
	let  
	%
	%
	%
	$ (\underline{\cdot}) \colon
	\{ [ v ]_{\lambda, \B(\R)} \colon 
	( v \colon (0,1) \to \R
	\text{ is uniformly continous} ) \} \to \mathcal{C}( [0,1], \R ) $
	%
	%
	be the function which satisfies
	for every 
	$ v \in \mathcal{C} ( [0,1], \R ) $
	that
	\begin{equation} 
	\underline{ [ v|_{(0,1)} ]_{ \lambda, \B(\R)} } = v 
	,
	\end{equation} 
	for every  
	$ u \in W^{1,r}( (0,1), \R ) $  
	let 
	$ \partial u    
	\in L^r(\lambda; \R ) $  
	satisfy for every
	$ \varphi  
	\in 
	\mathcal{C}_{cpt}^\infty( (0,1), \R ) $,
	$ v \in \mathcal{L}^r(\lambda; \R) $
	with $ v \in \partial u $
	that
	$ \int_{(0,1)} \underline{u}(x) \varphi'(x) \, dx
	=
	- \int_{(0,1)} v(x) \varphi(x) \, dx $,
	and let
	$ G \colon \R \to \R $
	be the function
	which satisfies for every
	$ t \in \R $
	that
	$ G(t) = |t|^{\frac{1}{\alpha} -1} t $.
	Note that 
	\begin{enumerate}[(a)] 
	\item it holds that
	$ G ( 0 ) = 0 $,
	\item it holds that 
	$ G \in \mathcal{C}^1( \R, \R ) $,
	and 
	\item it holds for every
	$ t \in [0,1] $ that 
	$
	G'(t) = \tfrac{1}{\alpha}   | t |^{ \frac{1}{\alpha} -1 }  
	$.
\end{enumerate}
	This and, e.g.,
	Brezis~\cite[Corollary~8.1]{Brezis2011}
	show that for every
	$ u \in W^{1,r}( (0,1), \R ) $
	it holds that
	\begin{equation} 
	[ ( G ( \underline{u} (x) ) )_{ x \in (0,1) } 
	]_{ \lambda, \B(\R) }
	\in W^{1,r}( (0,1), \R ) 
	\end{equation}
	and
	\begin{equation} 
	\partial [ ( G ( \underline{u} (x) ) )_{ x \in (0,1) } ]_{ \lambda, \B(\R) }
	= 
	[ 
	(
	G' ( \underline{u} (x) )  
	)_{ x \in (0,1) } 
	]_{ \lambda, \B(\R) } \partial u 
	.
	\end{equation}
	%
Combining this and, e.g., 
	Brezis~\cite[Theorem~8.2]{Brezis2011} 
	ensures that
	for every
	$ u \in W^{1,r}( (0,1), \R ) $,
	$ v \in \L^r( \lambda; \R) $,
	$ x \in [0,1] $
	with
	$ v \in \partial u $
	it holds that
	\begin{equation}
	%
G ( \underline{u} (x) )
	= 
	G( \underline{u}( 0 ) )
	+
	\int_0^x
	G'( \underline{u} ( t ) ) 
	v( t ) \, dt
	.
	\end{equation}
	This implies that
	for every
	$ u \in W^{1,r}( (0,1), \R ) $,
	$ v \in \L^r( \lambda; \R) $
	with
	$ v \in \partial u $,
	$ \underline{u}(0) = 0 $
	it holds that
	\begin{equation}
	\begin{split}
	\label{eq:sup_Estimate}
	&
	\| u \|_{L^\infty( \lambda; \R ) }^{\frac{1}{\alpha} }
	=
	\sup\nolimits_{ x \in [0,1] }
|G ( \underline{u} (x) )|
	\leq
	\int_0^1
	| G'( \underline{u} (t)) v ( t ) | \, dt
	\\
	&
	=
	\tfrac{1}{\alpha}
	\int_0^1
	| \underline{u} (t) |^{ \frac{1}{\alpha} - 1 }
	| v(t) |
	\, dt
	=
	\tfrac{1}{\alpha}
	\big\| 
	|u|^{ \frac{1}{\alpha } - 1 }
	| \partial u |
	\big\|_{ L^1( \lambda; \R)}
	.
	\end{split}
	\end{equation}
	Next observe that the fact that
	$ \frac{1}{r} = \frac{1}{q} + 1 - \frac{1}{ q \alpha } $
	and
	the fact that
	$ \frac{1}{p} = \frac{1}{q \alpha} -   \frac{1}{q} $
	ensure that
	$ \frac{1}{p} + \frac{1}{r} = 1 $.
	Combining this with~\eqref{eq:sup_Estimate}
	and
	H\"older's inequality
	demonstrates that
	for every
	$ u \in W^{1,r}_0( (0,1), \R ) $ 
	it holds that
	\begin{equation}
	\begin{split}
	&
	\| u \|_{L^\infty( \lambda; \R ) }^{\frac{1}{\alpha} }
	\leq 
	\tfrac{1}{\alpha}
	\big\|
	| u |^{ \frac{1}{\alpha } - 1 }
	\big\|_{L^p( \lambda; \R)} 
	\| 
	\partial u 
	\|_{L^r( \lambda; \R)} 
	%
	\leq  
	\tfrac{1}{\alpha}
	\|
	u
	\|_{L^{p( \frac{1}{\alpha}-1)}( \lambda; \R)}^{\frac{1-\alpha}{\alpha}}
	\| 
	\partial u 
	\|_{L^r( \lambda; \R)} 
	.
	\end{split}
	\end{equation}
	This completes the proof of Lemma~\ref{lemma:Gagliardo}.
\end{proof}
\begin{lemma} 
	\label{lemma:Gagliardo2}
	%
	Let $ \lambda \colon \B((0, 1)) \to [0, 1] $ be the 
	Lebesgue-Borel measure on $ (0, 1) $
	and
	let 
	$ q \in [1, \infty) $,
	$ p \in (q, \infty) $,
	$ r \in (1, \infty) $,
	$ \alpha \in (0,1) $
	satisfy
	$ \alpha ( \frac{1}{q} + 1 - \frac{1}{r} )
	= \frac{1}{q} - \frac{1}{p} $.
	Then there exists 
	$ C \in (0, \infty) $
	such that for every
	$ u \in W^{1,r}_0( (0,1), \R ) $ it holds that
	\begin{equation}
	\| u \|_{L^p( \lambda; \R ) }
	\leq 
	C
	\| u \|_{ W^{1,r}( (0,1), \R ) }^\alpha 
	\| u \|_{L^q( \lambda; \R) }^{1-\alpha}
	.
	\end{equation}
\end{lemma}
\begin{proof}[Proof of Lemma~\ref{lemma:Gagliardo2}]
	Throughout this proof let
	$ \beta = \frac{ \alpha p }{ p - q } $.
	Note that
	H\"older's inequality
	proves that
	for every
	$ u \in W^{1,r}( (0,1), \R ) $
	it holds that
	\begin{equation}
	\begin{split}
	\| u \|_{L^p( \lambda; \R ) }^p
	=
	\| |u|^q | u |^{p-q}  \|_{L^1( \lambda; \R ) } 
	\leq
	\| u \|_{L^q( \lambda; \R ) }^q
	\| u \|_{L^\infty( \lambda; \R)}^{p-q}
	.
	\end{split}
	\end{equation}
	Lemma~\ref{lemma:Gagliardo}
	(with
	$ q = q $,
	$ r = r $,
	$ \alpha =  \beta $
	in the notation of
	Lemma~\ref{lemma:Gagliardo})
	hence shows that 
	there exists
	$ C \in (0, \infty) $
	such that for every
	$ u \in W^{1,r}_0( (0,1), \R ) $
	it holds that
	\begin{equation}
	\begin{split}
	\| u \|_{L^p( \lambda; \R  ) }^p 
	&
	\leq
	C^p
	\| u \|_{L^q( \lambda; \R ) }^q 
	\| u \|_{ W^{1,r}((0,1), \R)}^{\beta (p-q)}
	\| u \|_{L^q( \lambda; \R)}^{(1-\beta)(p-q)}
	\\
	&
	= 
	C^p
	\| u \|_{L^q( \lambda; \R)}^{ p - \beta p + \beta q}
	\| u \|_{ W^{1,r}((0,1), \R)}^{\beta (p-q)}
	.
	\end{split}
	\end{equation}
	This implies that
	there exists
	$ C \in (0, \infty) $
	such that for every
	$ u \in W^{1,r}_0( (0,1), \R ) $
	it holds that
	\begin{equation}
	\begin{split}
	\| u \|_{L^p( \lambda; \R  ) } 
	\leq 
	C 
	\| u \|_{L^q( \lambda; \R)}^{ 1 - \beta( 1- \frac{q}{p} )}
	\| u \|_{ W^{1,r}((0,1), \R)}^{\beta ( 1 - \frac{q}{p} )}
	=
	C 
	\| u \|_{L^q( \lambda; \R)}^{ 1 - \alpha }
	\| u \|_{ W^{1,r}((0,1), \R)}^{\alpha }
	.
	\end{split}
	\end{equation}
	The proof of Lemma~\ref{lemma:Gagliardo2}
	is thus completed.
\end{proof} 
\subsection{Analysis of the nonlinearity}
\label{subsection:Nonlinearity}
In this subsection we recall
in
Lemmas~\ref{lemma:local_lip}--\ref{lemma:LocalLipSimple2},
Corollary~\ref{corollary:Extension of Function F part2},
Lemmas~\ref{lemma:monotonicity}--\ref{lemma:F_growth_estimates},  
Corollary~\ref{corollary:SatisfiedF},
Lemma~\ref{lemma:CrucialPropertiesBurgers},
and
Corollary~\ref{corollary:AppropriateCoercivityEstimate}
below
a few elementary
and well-known
properties of the 
nonlinearity appearing in 
the stochastic Burgers equation.
Corollaries~\ref{corollary:SatisfiedF}
and~\ref{corollary:AppropriateCoercivityEstimate}
are then used in Section~\ref{section:Existence}
below
to establish
in Theorem~\ref{theorem:existence_Burgers} the main result of this article.
%
%
%
%
%
%
\begin{lemma}
	\label{lemma:local_lip}
Assume Setting~\ref{setting:Examples}.
	Then
	\begin{enumerate}[(i)]
		\item \label{item:001}
		it holds for every
		$ u \in H_{ \nicefrac{1}{2} } $ that
		$ u^2 \in W^{1,2}((0,1), \R ) $
		and
		$ u \partial u = \tfrac{1}{2} \partial(u^2) $, 
		\item \label{item:02} 
		it holds for every $ v, w \in H_{ \nicefrac{1}{2} } $   that 
		\begin{equation}
		\begin{split}
		&
		\| F(v) - F(w) \|_H
		\leq 
		\tfrac{ | c_1 | }{ \sqrt{3 } \, c_0 }
		(   
		\| v \|_{H_{ \nicefrac{1}{2} } }   
		+
		\| w \|_{H_{ \nicefrac{1}{2} } }
		)
		\| v - w \|_{H_{ \nicefrac{1}{2} } }    
		,
		\end{split}
		\end{equation}
		\item \label{item:03} 
		it holds that
		$ F \in \mathcal{C}^1(H_{ \nicefrac{1}{2} }, H) $,
		and
		\item \label{item:04}
		it holds for every $ v, w \in H_{ \nicefrac{1}{2} } $  that
		$ F'(v)w=  c_1 ( w  \partial v + v  \partial w ) $.
	\end{enumerate}
\end{lemma}
\begin{proof}[Proof of Lemma~\ref{lemma:local_lip}]
	Observe that 
	items~\eqref{item:01 B} 
	and~\eqref{item:01b} 
	of
	Lemma~\ref{lemma:local_lipA}
	and,
	e.g., 
	\cite[Lemma~4.5]{JentzenPusnik2018Published}
	imply
	item~\eqref{item:001}.
	Furthermore, note that for every 
	$ v, w \in H_{ \nicefrac{1}{2} } $ 
	it holds that
	\begin{equation}
	\begin{split}
	\label{eq:diff_estimate}
	&
	\| F(v) - F(w) \|_H
	\leq 
	|c_1|
	\| \partial v \|_H
	\| v - w \|_{L^\infty(\lambda; \R )}
	+ 
	|c_1|
	\| w \|_{L^\infty(\lambda; \R )}
	\| \partial( v -w ) \|_H
	.
	\end{split}
	\end{equation}
	Items~\eqref{item:norm_estimate}
	and~\eqref{item:norm_estimate_infty} 
	of Lemma~\ref{lemma:local_lipA} 
	therefore
	show that for every 
	$ v, w \in H_{ \nicefrac{1}{2} } $ 
	it holds that
	\begin{equation}
	\begin{split}
	&
	\| F(v) - F(w) \|_H
	\leq
	\tfrac{ |c_1| }{  \sqrt{3} \, c_0 }
	\big(
	\| v \|_{H_{ \nicefrac{1}{2} } } 
	\| v - w \|_{H_{ \nicefrac{1}{2} } }
	+
	\| w \|_{H_{\nicefrac{1}{2}} } 
	\| v - w \|_{H_{\nicefrac{1}{2} } }
	\big)  
	.
	\end{split}
	\end{equation}
	This establishes item~\eqref{item:02}.
	In addition, note that for every 
	$ v, w \in H_{ \nicefrac{1}{2} } $ it holds that
	\begin{equation}
	\begin{split}
	\label{eq:difference}
	&
	F( v + w ) - F( v )
	=  
	c_1 \big( 
	( v  + w  ) ( \partial v + \partial  w ) 
	- 
	v  \partial  v    
	\big)
	=  
	c_1 ( v \partial  w + w \partial  v + w \partial w )
	.
	\end{split}
	\end{equation}
	Items~\eqref{item:norm_estimate}
	and~\eqref{item:norm_estimate_infty}
	of Lemma~\ref{lemma:local_lipA} 
	hence imply
	that for every 
	$ v, w \in H_{ \nicefrac{1}{2} } $ 
	it holds that
	\begin{equation}
	\begin{split} 
	\|
	F(v+w)
	-
	F(v) 
	-
	c_1
	(  v   \partial w + w   \partial v  ) 
	\|_H
	& =
	\| 
	c_1
	w  \partial  w   
	\|_H
	\\
	&
	\leq 
	\| 
	c_1
	w  
	\|_{ L^\infty( \lambda; \R ) } 
	\|
	\partial w 
	\|_H
	\leq 
	\tfrac{ | c_1 | }{\sqrt{3} \, c_0 }
	\| w \|_{ H_{ \nicefrac{1}{2} } }^2
	.
	\end{split} 
	\end{equation}
	Therefore, we obtain that   
	\begin{enumerate}[(a)]
	\item \label{item:additional 1}
	it holds that $ F \colon H_{ \nicefrac{1}{2} } \to H $ is differentiable
	and
	\item \label{item:additional 2}
	it holds for every $ v, w \in H_{ \nicefrac{1}{2} } $  that
	$ F'(v)w=  c_1 ( w  \partial v + v  \partial w ) $.
	\end{enumerate}
	Items~\eqref{item:norm_estimate}
	and~\eqref{item:norm_estimate_infty} 
	of Lemma~\ref{lemma:local_lipA} 
	hence
	assure that for every 
	$ u, v \in H_{ \nicefrac{1}{2} } $ it holds that
	\begin{equation}
	\begin{split}
	&
	\| F'(u) - F'(v) \|_{L(H_{ \nicefrac{1}{2} }, H)}
	=
	| c_1 |
	\sup\nolimits_{ w \in H_{ \nicefrac{1}{2} } \backslash \{ 0 \} }
	\tfrac{
		\| u  \partial w + w \partial u
		-
		(
		v  \partial w + w  \partial v
		)
		\|_H
	}{
		\| w \|_{ H_{ \nicefrac{1}{2} } }
	}
	\\
	&
	\leq
	| c_1 |
	\sup\nolimits_{ w \in H_{ \nicefrac{1}{2} } \backslash \{ 0 \} }
	\tfrac{
		\| 
		u \partial w 
		-
		v \partial w
		\|_H
		+
		\|
		w \partial u
		-  
		w \partial v 
		\|_H
	}{
		\| w \|_{H_{ \nicefrac{1}{2} }}
	}
	\\
	&
	\leq
	| c_1 |
	\sup\nolimits_{ w \in H_{ \nicefrac{1}{2} } \backslash \{ 0 \} }
	\tfrac{
		\| 
		u - v
		\|_{ L^\infty( \lambda; \R ) }
		\| 
		\partial w
		\|_H
		+
		\|
		w 
		\|_{ L^\infty( \lambda; \R ) }
		\|
		\partial 
		(u
		-  
		v) 
		\|_H
	}{
		\| w \|_{H_{ \nicefrac{1}{2} } }
	}
	\leq 
	\tfrac{ 2 | c_1 | }{\sqrt{3} \, c_0 }  \| u - v \|_{ H_{ \nicefrac{1}{2} } } 
	.
	\end{split}
	\end{equation}
	Combining items~\eqref{item:additional 1}
	and~\eqref{item:additional 2}
	therefore establishes items~\eqref{item:03}
	and~\eqref{item:04}.
	The proof of Lemma~\ref{lemma:local_lip}
	is thus completed.
\end{proof}
\begin{lemma}
\label{lemma:Extend local uniformly continuous function}
Let $ (X, d_X ) $
be a metric space,
let $ (Y, d_Y ) $ be a complete metric space,
let
$ S \subseteq X $ be a dense subset,
and let
$ F \colon S \to Y $ be a locally uniformly continuous function.
Then there exists a unique
continuous function
$ \bar F \colon X \to Y $
which satisfies
for every $ x \in S $ that $ \bar F(x) = F(x) $.
\end{lemma}
\begin{proof}[Proof of Lemma~\ref{lemma:Extend local uniformly continuous function}]
Throughout this proof let
$ U_x \subseteq X $, $ x \in S $,
be non-empty open sets which satisfy 
that
\begin{enumerate}[(a)]
\item it holds
for every
$ x \in S $ 
that  
$ F|_{ U_x \cap S } \colon U_x \cap S \to Y $
is uniformly continuous
and
\item it holds for every $ x \in S $ that
$ x \in U_x $.
\end{enumerate}
Observe that the fact that
for every $ x \in S $ it holds that
$ U_x \cap S $ is a dense
subset of
$ U_x $
and, e.g., 
Searc\'{o}id~\cite[Theorem~10.9.1]{Searcoid2007}  
show that 
there exist 
unique uniformly continuous functions 
$ F_x \colon U_x \to Y $,
$ x \in S $,
which satisfy for every
$ x \in S $,
$ u \in ( U_x \cap S ) $ that
\begin{equation} 
\label{eq:ExtensionBasic} 
F_x( u ) = F( u ) 
.
\end{equation}
Note that~\eqref{eq:ExtensionBasic} 
and the fact that
for every
$ x, \mathbf{x} \in S $
with
$ ( U_x \cap U_{ \mathbf{x} } ) \neq \emptyset $
it holds that
$ ( U_x \cap U_{ \mathbf{x} } ) \cap S $
is a dense subset of $ ( U_x \cap U_{ \mathbf{x} } ) $
ensure that
for every
$ x, \mathbf{x} \in S $,
$ u \in ( U_x \cap U_{ \mathbf{x} } ) $
there exist
$ ( u_n )_{ n \in \N } \subseteq ( U_x \cap U_{ \mathbf{x} } \cap S ) $
such that
$ \limsup_{ n \to \infty } \| u - u_n \|_X
=
0 $ 
and
\begin{equation}
\begin{split}
\| F_x(u) - F_{ \mathbf{x} }(u) \|_Y 
=
\limsup_{ n \to \infty }
\| F_x(u_n) - F_{ \mathbf{x} }(u_n) \|_Y
=
\limsup_{ n \to \infty }
\| F(u_n) - F(u_n) \|_Y
=
0.
\end{split}
\end{equation}
This proves that
for every
$ x, \mathbf{x} \in S $,
$ u \in ( U_x \cap U_{ \mathbf{x} } ) $
it holds that
\begin{equation} 
\label{eq:Proves uniqueness}
F_x(u) = F_{ \mathbf{x} }(u) 
.
\end{equation} 
Moreover, observe that the 
assumption that
$ S \subseteq X $
is a dense subset
ensures that
$ X = \cup_{x \in S} U_x $.
Combining~\eqref{eq:ExtensionBasic} 
and~\eqref{eq:Proves uniqueness}
hence
shows that
there exists 
a unique continuous function
$ \bar F \colon X \to Y $
which satisfies for every
$ u \in S $ that
\begin{equation}
\bar F(u) = F(u).
\end{equation}
The proof of Lemma~\ref{lemma:Extend local uniformly continuous function}
is thus completed.
\end{proof} 
\begin{lemma}
	\label{lemma:LocalLipSimple}
	Assume Setting~\ref{setting:Examples} 
	and let
	$ \gamma \in (\frac{1}{8}, \frac{1}{2}] $,
%
	$ \nu \in ( [ \frac{1}{2} - \gamma, \frac{1}{2} ] 
\cap ( \frac{3}{4}  -2 \gamma, \infty ) ) $.  
	Then there exists $ C \in \R $
	such that  
	for every 
	$ v, w \in H_{ \nicefrac{1}{2} } $ 
	it holds that 
	\begin{equation}
	\begin{split}
	\label{eq:LocLipSimple}
	&
	\| F(v) - F(w)  \|_{ H_{ - \nu } }
	\leq  
	C 
	\| v - w \|_{ H_\gamma } 
	(   
	1
	+
	\| v \|_{ H_\gamma }   
	+
	\| w \|_{ H_\gamma }
	)   
	.
	\end{split}
	\end{equation}
\end{lemma} 
\begin{proof}[Proof of Lemma~\ref{lemma:LocalLipSimple}]
	Note that the fact that
	$ \nu > \frac{3}{4} - 2 \gamma $
	ensures that
	$ (2\gamma ) + (2\gamma) - (1-2\nu)
	> \frac{1}{2} $.
	Combining this,
	Lemma~\ref{lemma:Sobolev_multiplication}
	(with
	$ s = 1 - 2 \nu $,
	$ q = 2 \gamma $,
	$ r = 2 \gamma $
	in the notation of 
	Lemma~\ref{lemma:Sobolev_multiplication}), 
	Lemma~\ref{lemma:DerivativeEstimate}
	(with
	$ \alpha = \nu $
	in the notation of Lemma~\ref{lemma:DerivativeEstimate}),
	and
	item~\eqref{item:001}
	of Lemma~\ref{lemma:local_lip}
	shows that  
	there exists 
	$ C \in [1, \infty) $
	such that for every
	$ v, w \in  H_{ \nicefrac{1}{2} } $
	it holds that
	$ ( v^2 - w^2 ) \in W^{1,2}((0,1), \R ) $
	and
	\begin{equation}
	\begin{split}
	\| 
	\partial( v^2 - w^2 ) 
	\|_{ H_{-  \nu} } 
	&\leq 
	C
	\| v^2 - w^2 \|_{ W^{ 1 - 2  \nu,2}( (0,1), \R ) }
	\\
	&
	\leq 
	C ^2
	\| v - w \|_{ W^{2  \gamma, 2}( (0,1), \R ) }
	\| v + w \|_{ W^{2  \gamma, 2}( (0,1), \R ) } 
	.
	\end{split} 
	\end{equation}
	Item~\eqref{item:Inclusion} of Lemma~\ref{lemma:NormEquivalency} 
	hence proves that 
	there exists $ C \in \R $ 
	such that for every
	$ v, w \in H_{ \nicefrac{1}{2} } $ it holds that
	\begin{equation}
	\begin{split}
	& 
	\| \partial( v^2 - w^2 )  \|_{ H_{ - \nu } }
	\leq 
	C
	\| v - w \|_{  H_{  \gamma} }
	( \| v \|_{  H_{ \gamma} }  + \| w \|_{  H_{  \gamma} } )
	.
	\end{split} 
	\end{equation}
	Item~\eqref{item:001} of Lemma~\ref{lemma:local_lip}
	therefore establishes~\eqref{eq:LocLipSimple}.
	The proof of Lemma~\ref{lemma:LocalLipSimple}
	is thus completed.
\end{proof}
\begin{lemma}
	\label{lemma:Extension of Function F part1}
	Assume Setting~\ref{setting:Examples}
	and let
	$ \gamma \in (\frac{1}{8}, \frac{1}{2}] $,
	$ \nu \in ( [ \frac{1}{2} - \gamma, \frac{1}{2} ] 
	\cap ( \frac{3}{4} - 2 \gamma, \infty ) ) $.
	Then 
	\begin{enumerate} [(i)]
		\item \label{item: Extension of F} 
		there exists 
		a unique continuous function
		$ \bar F \colon H_\gamma \to H_{-\nu} $
		which satisfies 
		for every $ v \in H_{\nicefrac{1}{2} } $
		that
		$ \bar F(v) = F(v) $
		and
		\item \label{item:Uniform continuity estimate} there exists
		$ C \in \R $
		which satisfies   
		for every
		$ v, w \in H_{ \gamma } $ 
		that
		\begin{equation}
		\begin{split}
		&
		\| \bar F(v) - \bar F(w) \|_{ H_{ - \nu } }
		\leq  
		C 
		\| v - w \|_{ H_\gamma } 
		(   
		1
		+
		\| v \|_{ H_\gamma }   
		+
		\| w \|_{ H_\gamma }
		)   
		.
		\end{split}
		\end{equation}
	\end{enumerate}
\end{lemma}
\begin{proof}[Proof of Lemma~\ref{lemma:Extension of Function F part1}]
	\sloppy 
	Observe that
	Lemma~\ref{lemma:LocalLipSimple}
	(with
	$ \gamma = \gamma $,
	$ \nu = \nu $
	in the notation of
	Lemma~\ref{lemma:LocalLipSimple})
	ensures  
	that
	there exists $ C \in \R $
	such that
	for every 
	$ v, w \in H_{ \nicefrac{1}{2} } $ 
	it
	holds that 
	\begin{equation}
	\begin{split}
	\label{eq:LocLipSimple Application}
	&
	\| F(v) - F(w)  \|_{ H_{ - \nu } }
	\leq  
	C 
	\| v - w \|_{ H_\gamma } 
	(   
	1
	+
	\| v \|_{ H_\gamma }   
	+
	\| w \|_{ H_\gamma }
	)   
	.
	\end{split}
	\end{equation}
	%
	Lemma~\ref{lemma:Extend local uniformly continuous function}
	(with
	$ X = H_{ \gamma } $,
	$ d_X = ( ( H_{ \gamma }  \times H_{ \gamma } )
	\ni (h_1, h_2 ) \mapsto \| h_1 - h_2 \|_{ H_{ \gamma }  } 
	\in [0, \infty) ) $,
	$ Y = H_{-\nu} $,
	$ d_Y = ( ( H_{ - \nu }  \times H_{ - \nu } )
	\ni (h_1, h_2 ) \mapsto \| h_1 - h_2 \|_{ H_{ - \nu }  } 
	\in [0, \infty) ) $,
	$ S = H_{ \nicefrac{1}{2} } $,
	$ F = F $
	in the notation of
	Lemma~\ref{lemma:Extend local uniformly continuous function})
	therefore
	establishes item~\eqref{item: Extension of F}.
	Moreover, note that the
	fact that
	$ H_{ \nicefrac{1}{2} } \subseteq H_\gamma $
	continuously and densely
	ensures that
	for every 
	$ v \in H_{ \gamma } $
	there exist 
	$ ( v_n )_{ n \in \N } 
	\subseteq H_{ \nicefrac{1}{2} } $
	such
	that
	$ \limsup_{n \to \infty } 
	\| v - v_n \|_{H_\gamma} = 0 $.
	Item~\eqref{item: Extension of F} 
	therefore
	implies that 
	for every
	$ v, w \in H_\gamma $ there
	exist 
	$ ( v_n )_{ n \in \N } 
	\subseteq H_{ \nicefrac{1}{2} } $
	and 
	$ ( w_n )_{ n \in \N } 
	\subseteq H_{ \nicefrac{1}{2} } $
	such that 
	$ \limsup_{n \to \infty } 
	(
	\| v - v_n \|_{H_\gamma} 
	+
	\| w - w_n \|_{H_\gamma} 
	)
	= 0 $
	and
	\begin{equation}
	\begin{split}
	\| \bar F(v) - \bar F(w) \|_{ H_{ - \nu } } 
	&\leq
	\limsup\nolimits_{n \to \infty}
	\| \bar F(v) - \bar F(v_n) \|_{ H_{ - \nu } }
	+
	\limsup\nolimits_{n \to \infty}
	\| \bar F(v_n) - \bar F(w_n) \|_{ H_{ - \nu } }
	\\
	&\quad+
	\limsup\nolimits_{n \to \infty}
	\| \bar F(w_n) - \bar F(w) \|_{ H_{ - \nu } }
	\\&=
	\limsup\nolimits_{n \to \infty}
	\| \bar F(v_n) - \bar F(w_n) \|_{ H_{ - \nu } }
	\\&
	=
	\limsup\nolimits_{n \to \infty}
	\| F(v_n) - F(w_n) \|_{ H_{ - \nu } }
	.
	\end{split}
	\end{equation}
	%
Combining this and~\eqref{eq:LocLipSimple Application} shows that
there exists 
$ C \in \R $
such that 
	for every $ v, w \in H_\gamma $
	there exist 
	$ ( v_n )_{ n \in \N } 
	\subseteq H_{ \nicefrac{1}{2} } $
	and
	$ ( w_n )_{ n \in \N } 
	\subseteq H_{ \nicefrac{1}{2} } $
	such that
	\begin{equation}
	\begin{split}
	\| \bar F(v) - \bar F(w) \|_{ H_{- \nu} }
	&\leq 
	C
	\limsup\nolimits_{ n \to \infty }
	\big( 
	\| 
	v_n - w_n
	\|_{ H_\gamma } 
	( 1 
	+ 
	\| v_n \|_{H_\gamma } 
	+
	\| w_n \|_{H_\gamma} ) 
	\big)
	\\
	&
	=
	C
	\| v - w \|_{H_\gamma} 
	(
	1
	+
	\| v \|_{H_\gamma}
	+
	\| w \|_{H_\gamma}
	)
	.
	\end{split}
	\end{equation}
	This establish item~\eqref{item:Uniform continuity estimate}.
	The proof of Lemma~\ref{lemma:Extension of Function F part1}
	is thus completed.
\end{proof} 
\begin{lemma}
	\label{lemma:LocalLipSimple2}
	Assume Setting~\ref{setting:Examples}
	and let
	$ \bar \partial \colon H \to H_{ - \nicefrac{1}{2} } $
	be the continuous function
	which satisfies for every
	$ v \in W^{1,2}( (0,1), \R ) $
	that
	$ \bar \partial v = \partial v $
	(cf.\ item~\eqref{item:Existence of extension} of Lemma~\ref{lemma:Identity2}).
	Then there exists $ C \in \R $
	such that  
	for every
	$ v, w \in H_{ \nicefrac{1}{8} } $ 
	it holds that
	$ (  v^2 - w^2  ) \in H $
	and
	\begin{equation}
	\begin{split}
	\label{eq:LocLipSimple2}
	&
	\| \bar \partial( v^2 - w^2 )  \|_{ H_{ - \nicefrac{1}{2} } }
	\leq  
	C 
	\| v - w \|_{ H_{ \nicefrac{1}{8} } } 
	(   
	1
	+
	\| v \|_{ H_{ \nicefrac{1}{8} } }   
	+
	\| w \|_{ H_{ \nicefrac{1}{8} } }
	)   
	.
	\end{split}
	\end{equation}
\end{lemma} 
\begin{proof}[Proof of Lemma~\ref{lemma:LocalLipSimple2}]
	Note that
	item~\eqref{item:Inclusion}
	of Lemma~\ref{lemma:NormEquivalency}
	(with
	$ s = \nicefrac{1}{8} $
	in the notation
	of item~\eqref{item:Inclusion}
	of Lemma~\ref{lemma:NormEquivalency}) 
	ensures that
	$ H_{ \nicefrac{1}{8} } \subseteq W^{\nicefrac{1}{4}, 2}( (0,1), \R ) $
	continuously.
	The Sobolev embedding theorem
	hence shows that
	\begin{equation}
	\label{eq:Sobolev embedding}
	H_{ \nicefrac{1}{8} } \subseteq L^4( \lambda; \R ) 
	\end{equation}
    continuously.
	This implies
	that 
	for every $ v \in H_{ \nicefrac{1}{8} } $
	it holds that
	$ v^2 \in H $
	and
	\begin{equation}
	\sup_{ w \in H_{ \nicefrac{1}{8} } \backslash \{ 0 \} }
	\frac{ \| w \|_{ L^4(\lambda; \R ) } }
	{ \| w \|_{ H_{ \nicefrac{1}{8} } } }
	< \infty 
	.
	\end{equation}
Item~\eqref{item:Norm estimate of extension}
of Lemma~\ref{lemma:Identity2}
	and the Cauchy-Schwarz inequality
	hence prove that
	\begin{enumerate}[(a)]
		\item 
	it holds for every
	$ v, w \in H_{ \nicefrac{1}{8} } $ 
	that 
	$ ( v^2 - w^2 ) \in H $
	and 
	\item 
	there exists 
	$ C \in [1, \infty) $ 
	such that
	for every
	$ v, w \in H_{ \nicefrac{1}{8} } $
	it holds that 
	\begin{equation}
	\begin{split}
	&
	\| 
	\bar \partial( v^2 - w^2 ) 
	\|_{ H_{-  \nicefrac{1}{2} } } 
	\leq 
	C
	\| v^2 - w^2 \|_{ H } 
	\leq 
	C
	\| v - w \|_{ L^4(\lambda; \R ) }
	\| v + w \|_{ L^4(\lambda; \R ) } 
	\\
	&\leq
	C^3
	\| v - w \|_{ H_{ \nicefrac{1}{8} } }
	\| v + w \|_{ H_{ \nicefrac{1}{8} } }
	\leq
	C^3
	\| v - w \|_{ H_{ \nicefrac{1}{8} } }
	(
	1+
	\| v \|_{ H_{ \nicefrac{1}{8} } }
	+
	\| w \|_{ H_{ \nicefrac{1}{8} } }
	)
	.
	\end{split} 
	\end{equation}
\end{enumerate}
	The proof of Lemma~\ref{lemma:LocalLipSimple2}
	is thus completed.
\end{proof}
\begin{corollary}
	\label{corollary:Extension of Function F part2}
	Assume Setting~\ref{setting:Examples}.
	Then 
	\begin{enumerate} [(i)]
		\item \label{item: Extension of F 2} 
		there exists 
		a unique continuous function
		$ \bar F \colon H_{ \nicefrac{1}{8} } \to H_{-\nicefrac{1}{2}} $
		which satisfies
		for every $ v \in H_{\nicefrac{1}{2} } $
		that
		$ \bar F(v) = F(v) $
		and
		\item \label{item:Uniform continuity estimate 2} there exists
		$ C \in \R $
		which satisfies
		for every
		$ v, w \in H_{ \nicefrac{1}{8} } $ 
		that 
		\begin{equation}
		\begin{split}
		&
		\| \bar F(v) - \bar F(w) \|_{ H_{ - \nicefrac{1}{2} } }
		\leq  
		C 
		\| v - w \|_{ H_{ \nicefrac{1}{8} } } 
		(   
		1
		+
		\| v \|_{ H_{ \nicefrac{1}{8} } }   
		+
		\| w \|_{ H_{ \nicefrac{1}{8} } }
		)   
		.
		\end{split}
		\end{equation}
	\end{enumerate}
\end{corollary}
\begin{proof}[Proof of Corollary~\ref{corollary:Extension of Function F part2}]
	Throughout this proof let
	$ \bar \partial \colon H \to H_{ - \nicefrac{1}{2} } $
	be the continuous function 
	which satisfies for every
	$ v \in W^{1,2}((0,1), \R ) $ that
	$ \bar \partial v = \partial v $
		(cf.\ item~\eqref{item:Existence of extension} of Lemma~\ref{lemma:Identity2}).
	Note that 
	item~\eqref{item:001}
	of Lemma~\ref{lemma:local_lip}
	ensures that for every 
	$ v \in H_{ \nicefrac{1}{2} } $
	it holds that
	$ v^2 \in W^{1,2}((0,1), \R) $
	and 
	\begin{equation}
	\label{eq:F identity}
	F(v ) 
	=  
	\tfrac{ c_1 }{2} \partial ( v^2 )
	= 
	\tfrac{ c_1 }{2} \bar \partial ( v^2 )
	.
	\end{equation} 
	Lemma~\ref{lemma:LocalLipSimple2}
	(with
	$ \bar \partial = \bar \partial $
	in the notation of Lemma~\ref{lemma:LocalLipSimple2}) 
	hence shows that
	there exists
	$ C \in \R $
	such that
	for every 
	$ v, w \in H_{ \nicefrac{1}{2} } $
	it holds that
	\begin{equation}
	\begin{split}
	\label{eq:important estimate 2}
	\| F(v) - F(w) \|_{ H_{ - \nicefrac{1}{2} } } 
	&=
	\tfrac{ |c_1| }{2}
	\| \partial (v^2) - \partial (w^2) \|_{ H_{ - \nicefrac{1}{2} } }
	=
	\tfrac{ |c_1| }{2}
	\| \bar \partial( v^2 - w^2 )  \|_{ H_{ - \nicefrac{1}{2} } }
	\\
	&
	\leq  
	C 
	\| v - w \|_{ H_{ \nicefrac{1}{8} } } 
	(   
	1
	+
	\| v \|_{ H_{ \nicefrac{1}{8} } }   
	+
	\| w \|_{ H_{ \nicefrac{1}{8} } }
	)   
	.
	\end{split}
	\end{equation}
	%
	%
	%
	%
	Lemma~\ref{lemma:Extend local uniformly continuous function}
	(with
	$ X = H_{ \nicefrac{1}{8} } $,
	$ d_X = ( ( H_{ \nicefrac{1}{8} }  \times H_{ \nicefrac{1}{8} } )
	\ni (h_1, h_2 ) \mapsto \| h_1 - h_2 \|_{ H_{ \nicefrac{1}{8} }  } 
	\in [0, \infty) ) $,
	$ Y = H_{ -\nicefrac{1}{2} } $,
	$ d_Y = ( ( H_{ - \nicefrac{1}{2} }  \times H_{ - \nicefrac{1}{2} } )
	\ni (h_1, h_2 ) \mapsto \| h_1 - h_2 \|_{ H_{ - \nicefrac{1}{2} }  } 
	\in [0, \infty) ) $,
	$ S = H_{ \nicefrac{1}{2} } $,
	$ F = F $
	in the notation of
	Lemma~\ref{lemma:Extend local uniformly continuous function})
	therefore
	establishes 
	item~\eqref{item: Extension of F 2}.
	Moreover, note that the
	fact that
	$ H_{ \nicefrac{1}{2} } \subseteq H_{ \nicefrac{1}{8} } $
	continuously and densely
	ensures that
	for every 
	$ v \in H_{ \nicefrac{1}{8} } $
	there exist 
	$ ( v_n )_{ n \in \N } 
	\subseteq H_{ \nicefrac{1}{2} } $
	such
	that
	$ \limsup_{n \to \infty } \| v - v_n \|_{H_{ \nicefrac{1}{8} }} = 0 $.
	This and item~\eqref{item: Extension of F 2} imply that 
	for every
	$ v, w \in H_{ \nicefrac{1}{8} } $ there
	exist 
	$ ( v_n )_{ n \in \N } 
	\subseteq H_{ \nicefrac{1}{2} } $
	and 
	$ ( w_n )_{ n \in \N } 
	\subseteq H_{ \nicefrac{1}{2} } $
	such that 
	$ \limsup_{n \to \infty } 
	(
	\| v - v_n \|_{ H_{ \nicefrac{1}{8} } } 
	+
	\| w - w_n \|_{ H_{ \nicefrac{1}{8} } }
	)
	= 0 $
	and
	\begin{equation}
	\begin{split}
	\| \bar F(v) - \bar F(w) \|_{ H_{ - \nicefrac{1}{2} } } 
	&\leq
	\limsup\nolimits_{n \to \infty}
	\| \bar F(v) - \bar F(v_n) \|_{ H_{ - \nicefrac{1}{2} } }
	\\
	&
	\quad
	+
	\limsup\nolimits_{n \to \infty}
	\| \bar F(v_n) - \bar F(w_n) \|_{ H_{ - \nicefrac{1}{2} } }
	\\
	&\quad+
	\limsup\nolimits_{n \to \infty}
	\| \bar F(w_n) - \bar F(w) \|_{ H_{ - \nicefrac{1}{2} } }
	\\&=
	\limsup\nolimits_{n \to \infty}
	\| \bar F(v_n) - \bar F(w_n) \|_{ H_{ - \nicefrac{1}{2} } }
	\\&=
	\limsup\nolimits_{n \to \infty}
	\| F(v_n) - F(w_n) \|_{ H_{ - \nicefrac{1}{2} } }
	.
	\end{split}
	\end{equation}
	%
	%
%
Combining 
this 
and~\eqref{eq:important estimate 2}
shows that
there exists $ C \in \R $
such that 
for every $ v, w \in H_{ \nicefrac{1}{8} } $
	there exist 
	$ ( v_n )_{ n \in \N } 
	\subseteq H_{ \nicefrac{1}{2} } $
	and
	$ ( w_n )_{ n \in \N } 
	\subseteq H_{ \nicefrac{1}{2} } $
	such that
	\begin{equation}
	\begin{split}
	\| \bar F(v) - \bar F(w) \|_{ H_{- \nicefrac{1}{2} } }
	&\leq 
	C
	\limsup\nolimits_{ n \to \infty }
	\big( 
	\| 
	v_n - w_n
	\|_{ H_{ \nicefrac{1}{8} } }
	( 1 
	+ 
	\| v_n \|_{ H_{ \nicefrac{1}{8} } } 
	+
	\| w_n \|_{ H_{ \nicefrac{1}{8} } } ) 
	\big)
	\\
	&
	=
	C
	\| v - w \|_{ H_{ \nicefrac{1}{8} } } 
	(
	1
	+
	\| v \|_{ H_{ \nicefrac{1}{8} } }
	+
	\| w \|_{ H_{ \nicefrac{1}{8} } }
	)
	.
	\end{split}
	\end{equation}
	This establishes 
	item~\eqref{item:Uniform continuity estimate 2}.
	The proof of Corollary~\ref{corollary:Extension of Function F part2}
	is thus completed.
\end{proof} 
\begin{lemma}
	\label{lemma:monotonicity}
	Assume Setting~\ref{setting:Examples}. 
	Then 
	\begin{enumerate}[(i)]   
		\item \label{item:b}
		it holds that
		$ F \in \mathcal{C}^1( H_{ \nicefrac{1}{2} }, H ) $ 
		and
		\item\label{item:c} 
		there exists $ C \in (0,\infty) $
		such that  
		for every 
		$ \varepsilon \in (0,\infty) $,
		$ v, w \in H_{ \nicefrac{1}{2} } $ 
		it holds that
		\begin{equation} 
		\begin{split} 
		&
		\langle F'(v) w, w \rangle_H 
		\leq  
		\varepsilon
		\| v \|_{H_{\nicefrac{1}{2}}}^2    \| w \|_H^2 
		+
		\tfrac{C}{ \varepsilon^2 } \| w \|_H^2
		+
		\| w \|_{H_{\nicefrac{1}{2}}}^2
		.
		\end{split}
		\end{equation}	
	\end{enumerate}
\end{lemma}
\begin{proof}[Proof of Lemma~\ref{lemma:monotonicity}]
	Note that item~\eqref{item:03} of Lemma~\ref{lemma:local_lip}
	establishes item~\eqref{item:b}.	
	Next observe that	
	item~\eqref{item:01 B} 
	of Lemma~\ref{lemma:local_lipA},
	Lemma~\ref{lemma:int_parts},
	items~\eqref{item:001} and~\eqref{item:04} of Lemma~\ref{lemma:local_lip},  
	and the Cauchy-Schwarz inequality
	imply that for every
	$ v, w \in H_{ \nicefrac{1}{2} } $ it holds that
	\begin{equation}
	\begin{split}
	\label{eq:starting_estimate}
	&
	2
	\langle F'( v ) w, w \rangle_H 
	= 
	2 c_1
	\langle w \partial v 
	+
	v  \partial w, w \rangle_H
	= 
	2 c_1 \langle w \partial v, w \rangle_H
	+
	2 c_1 \langle v \partial w, w \rangle_H
	\\
	&
	= 
	2 c_1
	\langle \partial v, w^2 \rangle_H
	+
	c_1
	\langle v, 2 w \partial w \rangle_H
	= 
	2 c_1
	\langle \partial v, w^2 \rangle_H
	+
	c_1
	\langle v, \partial ( w^2 ) \rangle_H
	\\
	&
	= 
	2 c_1
	\langle \partial v, w^2 \rangle_H
	-
	c_1
	\langle \partial v, w^2 \rangle_H
	=
	c_1
	\langle \partial v, w^2 \rangle_H
	\\
	&
	\leq
	| c_1 |
	\| \partial v \|_H
	\| w^2 \|_H
	=
	| c_1 | 
	\| \partial v \|_H
	\| w \|_{ L^4( \lambda; \R )}^2 
	.
	\end{split}
	\end{equation}	
	Moreover, note that
	Lemma~\ref{lemma:Gagliardo2}
	(with
	$ q = 2 $,
	$ p = 4 $,
	$ r = 2 $,
	$ \alpha = \nicefrac{1}{4} $
	in the notation of 
	Lemma~\ref{lemma:Gagliardo2}) 
	and
	item~\eqref{item:01 B} of Lemma~\ref{lemma:local_lipA}
	%
	prove that there exists 
	$ C \in \R $
	such that for every $ w \in H_{ \nicefrac{1}{2} } \subseteq W^{1,2}_0((0,1), \R) $ it holds that
	\begin{equation}
	\label{eq:Gagliardo}
	\| w \|_{L^4( \lambda; \R  )} 
	\leq
	C  
	\| w \|_{ W^{1,2}( (0,1), \R )  }^{ \nicefrac{1}{4} }
	\| w \|_H^{ \nicefrac{3}{4} }
	.
	\end{equation}
	Items~\eqref{item:01 B} 
	and~\eqref{item:norm_estimate} of
	Lemma~\ref{lemma:local_lipA}, 
	\eqref{eq:starting_estimate},
	and the fact that
	for every
	$ x_1, x_2, x_3, x_4 \in \R $
	it holds that
	$ 4 x_1 x_2 x_3 x_4 
	\leq 
	|x_1|^4 + |x_2|^4 + |x_3|^4 + |x_4|^4 $ 
	hence show that
	there exists 
	$ C \in (0,\infty) $
	such that
	for every
	$ \varepsilon \in (0,\infty) $,
	$ v, w \in H_{ \nicefrac{1}{2} } $ it holds that
	\begin{equation}
	\begin{split} 
	&
	\langle F'(v) w, w \rangle_H 
	\leq
	C
	\| \partial v \|_H
	\| w \|_{W^{1,2}((0,1), \R )}^{
		\nicefrac{1}{2}}
	\| w \|_H^{ \nicefrac{3}{2} }
	\\
	&
	\leq
	C
	| c_0 |^{- \nicefrac{1}{2} }
	\Big[
	\sup\nolimits_{ u \in H_{ \nicefrac{1}{2} } \backslash \{ 0 \} }
	\tfrac{ \| u \|_{W^{1,2}((0,1), \R) } }{
		\| u \|_{ H_{ \nicefrac{1}{2} } }	
	}
	\Big]^{ \nicefrac{1}{2} }
	\|  v \|_{ H_{\nicefrac{1}{2} } }		
	\| w \|_{ H_{ \nicefrac{1}{2} } }^{ \nicefrac{1}{2} }
	\| w \|_H^{ \nicefrac{3}{2} }
	\\
	&
	=
	4
	\bigg[ 
	\big( 
	\tfrac{ \varepsilon }{2}  
	\| v \|_{H_{\nicefrac{1}{2}}}^2    \| w \|_H^2
	\big) 
	\big( 
	\tfrac{ \varepsilon }{2}  
	\| v \|_{H_{\nicefrac{1}{2}}}^2    \| w \|_H^2
	\big) 
	\Big( 
	\tfrac{C^4}{64 \varepsilon^2 | c_0 |^2}
	\Big[
	\sup\nolimits_{ u \in H_{ \nicefrac{1}{2} } \backslash \{ 0 \} }
	\tfrac{ \| u \|_{W^{1,2}((0,1), \R) } }{
		\| u \|_{ H_{ \nicefrac{1}{2} } }	
	}
	\Big]^2
	\| w \|_H^2
	\Big)  
	\| w \|_{H_{\nicefrac{1}{2}}}^2
	\bigg]^{\frac{1}{4}}
	\\
	& 
	\leq
	\tfrac{ \varepsilon }{2}  
	\| v \|_{H_{\nicefrac{1}{2}}}^2    \| w \|_H^2
	+ 
	\tfrac{ \varepsilon }{2}  
	\| v \|_{H_{\nicefrac{1}{2}}}^2    \| w \|_H^2
	+
	\tfrac{C^4}{64 \varepsilon^2 | c_0 |^2}
	\Big[
	\sup\nolimits_{ u \in H_{ \nicefrac{1}{2} } \backslash \{ 0 \} }
	\tfrac{ \| u \|_{W^{1,2}((0,1), \R) } }{
		\| u \|_{ H_{ \nicefrac{1}{2} } }	
	}
	\Big]^2
	\| w \|_H^2
	+ 
	\| w \|_{H_{\nicefrac{1}{2}}}^2
	\\
	&
	=
	\varepsilon
	\| v \|_{H_{\nicefrac{1}{2}}}^2    \| w \|_H^2 
	+
	\tfrac{C^4}{64 \varepsilon^2 | c_0 |^2}
	\Big[
	\sup\nolimits_{ u \in H_{ \nicefrac{1}{2} } \backslash \{ 0 \} }
	\tfrac{ \| u \|_{W^{1,2}((0,1), \R) } }{
		\| u \|_{ H_{ \nicefrac{1}{2} } }	
	}
	\Big]^2
	\| w \|_H^2
	+
	\| w \|_{H_{\nicefrac{1}{2}}}^2
	< \infty 
	.
	\end{split}
	\end{equation}	
	This establishes item~\eqref{item:c}.
	The proof of Lemma~\ref{lemma:monotonicity}
	is thus completed.
\end{proof}
\begin{lemma}
	\label{lemma:PositiveNormF}
	%
Assume Setting~\ref{setting:Examples},
let 
$ \alpha \in [0, \nicefrac{1}{2} ] \backslash \{ \nicefrac{1}{4} \} $,
let $ \mathcal{P} (\H) $ be the
power set of $ \H $,
let $ \mathcal{P}_0(\H) = \{ \theta \in \mathcal{P}(\H) \colon \theta \text{ is a finite set} \} $, 
and
let 
$ ( P_I )_{ I\in \mathcal{P} (\H) } \subseteq L(H) $ 
satisfy
for every
$ I \in \mathcal{P}(\H) $,  
$ v \in H $ 
that
$ P_I(v) =\sum_{h \in I} \left< h ,v \right>_H h $.
	Then it holds 
	that
	\begin{equation} 
	\sup\nolimits_{ I \in \mathcal{P}_0(\H) }
	\sup\nolimits_{ v \in H_{ \alpha + ( \nicefrac{1}{2} ) } \backslash \{ 0 \} }
	\tfrac{ \| P_I F(v) \|_{ H_{ \alpha } } }{ \| v \|_{H_{ \alpha + ( \nicefrac{1}{2} ) } }^2 }
	<
	\infty
	.
	\end{equation}
\end{lemma}
\begin{proof}[Proof of Lemma~\ref{lemma:PositiveNormF}]
	Throughout this proof consider the notation in
	Triebel~\cite[Section~1.3.2 on page~24]{Triebel1978}
	(cf., e.g., Lunardi~\cite[Definition~1.2]{Lunardi2018}).
	Note that
	item~\eqref{item:Equivalence 2}  
	of
	Lemma~\ref{lemma:NormEquivalency}
	shows that
	for every
	$ I \in \mathcal{P}_0(\H) $,
	$ v \in H_{ \alpha + ( \nicefrac{ 1 }{ 2 } ) } $
	it holds that
	%
	%
	%
	%
	\begin{equation}
	\label{eq:SOmeEst}
	\| P_I F( v ) \|_{H_\alpha }
	\leq 
	\Big( 
	\sup\nolimits_{ u \in H_{\alpha} \backslash \{0\} }
	\tfrac{ \| u \|_{H_\alpha } }
	{ \| u \|_{ W^{2\alpha,2}((0,1), \R ) } }
	\Big)
	\| P_I F( v ) \|_{ W^{2 \alpha, 2 } ( (0,1), 
		\R ) }
	< \infty
	.
	\end{equation}
	Moreover,
	observe that
	the fact that
	\begin{equation} 
	\big( 
	W^{2 \alpha + 1, 2}((0,1), \R )  
	\ni 
	v \mapsto  \partial v 
	\in 
	W^{2\alpha, 2}((0,1), \R ) 
	\big) 
	\in 
	L\big( W^{2 \alpha + 1, 2}((0,1), \R ),
	W^{2 \alpha, 2}((0,1), \R ) \big)  
	\end{equation} 
	(cf., e.g.,
	Triebel~\cite[item~(a) of Theorem~1 in Section~4.4.2
	on page~323, 
	and Remark~2 in Section~4.4.2
	on page~323]{Triebel1978}), 
	item~\eqref{item:Inclusion}
	of Lemma~\ref{lemma:NormEquivalency},	
	Lemma~\ref{lemma:Sobolev_multiplication}
	(with
	$ s = 2 \alpha + 1 $,
	$ q = 2 \alpha + 1 $, 
	$ r = 2 \alpha + 1 $
	in the notation of  
	Lemma~\ref{lemma:Sobolev_multiplication}),
	and 
	item~\eqref{item:001}
	of
	Lemma~\ref{lemma:local_lip}
	imply that 
	there exists $ C \in (0, \infty) $
	such that
	for every
	$ v \in H_{ \alpha + ( \nicefrac{1}{2} ) } $
	it holds that
	$ v \in W^{ 2 \alpha + 1, 2 } ( (0,1), \R) $,
	$ v^2 \in W^{ 2 \alpha + 1, 2 } ( (0,1), \R) $,
	$ F(v) \in W^{ 2 \alpha, 2 } ( (0,1), \R) $,
	and 
	\begin{equation}
	\begin{split}
	\label{eq:Est}
	& 
	\| F( v ) \|_{ W^{2 \alpha, 2 } ( (0,1), 
		\R ) } 
	\leq
C
	\| v^2 \|_{ W^{2 \alpha + 1, 2 } ( (0,1), 
		\R ) }
	\\
	&
	\leq
	C
	\Big(
	\sup\nolimits_{ u \in  W^{2 \alpha + 1, 2 } ( (0,1), 
		\R ) \backslash \{0 \}  }
	\tfrac{ \| u^2 \|_{ W^{2 \alpha + 1, 2 } ( (0,1), 
			\R ) } }{ \| u \|_{ W^{2 \alpha + 1, 2 } ( (0,1), 
			\R ) }^2 }
	\Big)
	\| v \|_{ W^{2 \alpha + 1, 2 } ( (0,1), 
		\R ) }^2
	<
	\infty 
	. 
	\end{split}
	\end{equation}
	Moreover, note that, e.g.,  
	Triebel~\cite[Definition~1 
	in Section~4.2.1 
	on page~310,
	Theorem~1 in Section~4.3.1 
	on page~317,
	item~(a) in Theorem~1 
	in Section~4.4.2
	on page~323, 
	and 
	Remark~2 in Section~4.4.2 
	on page~324]{Triebel1978}  
	shows that for every 
	$ \iota \in (0, \nicefrac{1}{2}) $
	it holds that
	\begin{equation} 
	\label{eq:Equal spaces}
	( H,
	W^{1,2}( (0,1), \R )
	)_{ 2 \iota, 2 }  
	=
	W^{ 2 \iota, 2 }( (0,1), \R )
	\end{equation}
	and
	\begin{equation} 
	\label{eq:Equal norms}
	\sup_{ x \in 
		W^{ 2 \iota, 2 }( (0,1), \R ) \backslash \{0\}
	}
	\Bigg(
	\frac{
		\| x \|_{ W^{ 2 \iota, 2 }( (0,1), \R ) }
	}{
		\| x \|_{ ( H,
			W^{1,2}( (0,1), \R )
			)_{ 2 \iota, 2 } }
	}
	+
	\frac{
		\| x \|_{ ( H,
			W^{1,2}( (0,1), \R )
			)_{ 2 \iota, 2 } }
	}
	{
		\| x \|_{ W^{ 2 \iota, 2}( (0,1), \R ) }
	}
	\Bigg)
	< \infty.
	\end{equation} 
	In addition, observe that the fact that 
	$ \H \subseteq W^{1,2}( (0,1), \R ) $
	is an orthogonal system
	ensures that for every 
	$ I \in \mathcal{P}_0(\H) $,
	$ v \in W^{1,2}( (0,1), \R ) $
	it holds that
	\begin{equation} 
	\label{eq:Projection0}
	\| P_I v \|_{ W^{1, 2 } ( (0,1), 
		\R ) }
	\leq 
	\| v \|_{ W^{1, 2 } ( (0,1), 
		\R ) } .
	\end{equation} 
	The fact that
		for every
	$ I \in \mathcal{P}_0(\H) $, 
	$ v \in H $
	it holds that
	$ \| P_I v \|_H \leq \| v \|_H $,
	%
	%
	\eqref{eq:Equal spaces},
	\eqref{eq:Equal norms},
	and, e.g., 
	Lunardi~\cite[Theorem~1.6]{Lunardi2018}
	therefore prove that for every
	$ I \in \mathcal{P}_0(\H) $,
	$ v \in W^{2 \alpha,2}( (0,1), \R) $
	it holds that
	\begin{equation} 
	\| P_I v \|_{ W^{2 \alpha,2}( (0,1), \R) }
	\leq
	\| v \|_{ W^{2 \alpha,2}( (0,1), \R) } 
	.
	\end{equation} 
	%
	%
	%
	%
	Combining~\eqref{eq:SOmeEst},
	\eqref{eq:Est},  
	and
	item~\eqref{item:Inclusion} of
	Lemma~\ref{lemma:NormEquivalency}
	%
	hence
	implies that 
	there exists $ C \in (0, \infty) $
	such that
	for every
	$ I \in \mathcal{P}_0(\H) $,
	$ v \in H_{ \alpha + ( \nicefrac{1}{2} ) } $
	it holds that 
	\begin{equation}
	\begin{split}
	%
	\| P_I F( v ) \|_{ H_\alpha }   
	&\leq
C
	\Big( 
	\sup\nolimits_{ u \in H_{\alpha} \backslash \{0\} }
	\tfrac{ \| u \|_{H_\alpha } }
	{ \| u \|_{ W^{2\alpha,2}((0,1), \R ) } }
	\Big) 
	\Big(
	\sup\nolimits_{ u \in  W^{2 \alpha + 1, 2 } ( (0,1), 
		\R ) \backslash \{0 \}  }
	\tfrac{ \| u^2 \|_{ W^{2 \alpha + 1, 2 } ( (0,1), 
			\R ) }  }{ \| u \|_{ W^{2 \alpha + 1, 2 } ( (0,1), 
			\R ) }^2 }
	\Big)
	\\
	&
	\quad
	\cdot
	\Big( 
	\sup\nolimits_{ u \in H_{\alpha + ( \nicefrac{1}{2} ) } \backslash \{0\} }
	\tfrac{ \| u \|_{ W^{2\alpha + 1,2}((0,1), \R ) } }
	{ \| u \|_{H_{\alpha + ( \nicefrac{1}{2} ) } } }
	\Big)^2
	\| v \|_{ H_{ \alpha + ( \nicefrac{1}{2} ) } }^2
	<
	\infty 
	. 
	\end{split}
	\end{equation}
	The proof of 
	Lemma~\ref{lemma:PositiveNormF}
	is thus completed.
\end{proof}
\begin{lemma}
	\label{lemma:F_growth_estimates}
	Assume Setting~\ref{setting:Examples}.
	Then 
	\begin{enumerate}[(i)]
		\item \label{item:Growth1}
		it holds for every 
		$ v \in H_{ \nicefrac{1}{2} } $,
		$ \alpha \in ( \nicefrac{3}{4}, \infty) $ that
		\begin{equation}
		\begin{split} 
		& 
		\| F( v ) \|_{H_{-\alpha}} 
		\leq   
		|c_1|  
		| c_0 |^{ - \alpha }
		\bigg[ 
		\tfrac{1}{2}
			\sum_{n=1}^\infty 
			| \pi n |^{ 2 - 4 \alpha }
		\bigg]^{ \nicefrac{1}{2} } 
		\| v \|_H^2
		< \infty,
		\end{split} 
		\end{equation}
		\item \label{item:Growth2} 
		it holds for every 
		$ v \in H_{  \nicefrac{1}{2} } $, 
		$ \alpha \in (\nicefrac{1}{4}, \nicefrac{1}{2}] $
		that
		\begin{equation}
		\begin{split} 
		\| F(v) \|_{H_{-\alpha}}
		& 
		\leq 
		\tfrac{ |c_1| }{2}
		\Big( 
		\sup\nolimits_{ u  \in H_{ \nicefrac{1}{2} } \backslash \{0\} }
		\tfrac{ \| \partial (u^2) \|_{ H_{-\alpha} } }
		{ \| u^2 \|_{ W^{ 1 - 2 \alpha, 2 }( (0,1), \R ) } }
		\Big)
		\Big( 
		\sup\nolimits_{ u \in H_{ \nicefrac{1}{2} } \backslash \{0\} }
		\tfrac{\| u^2 \|_{ { W^{ 1 - 2 \alpha, 2 }( (0,1), \R ) } } }{
			\| u \|_{ W^{ \nicefrac{2(1 -  \alpha)}{3}, 2  }( (0,1), \R )}^2	
		}
		\Big)
		\\
		&
		\quad
		\cdot
		\Big(  
		\sup\nolimits_{ u \in H_{ \nicefrac{1}{2} } \backslash \{ 0 \} }
		\tfrac{ 
			\| u \|_{ W^{ \nicefrac{2(1 -  \alpha)}{3}, 2  }( (0,1), \R )}^2
		}
		{	\| u \|_{ H_{ \nicefrac{ (1 - \alpha) }{3} } }^2
		}
		\Big)
		\| v \|_{ H_{ \nicefrac{ (1 - \alpha) }{3} } }^2
		< \infty,
		%
		\end{split}
		\end{equation}
		and
		\item \label{item:Growth4}
		it holds for every
		$ v \in H_{ \nicefrac{1}{2} } $
		that
		\begin{equation}
		\| F(v) \|_H 
		\leq
		\tfrac{ | c_1 | }{ \sqrt{3} \, c_0 }
		\| v \|_{ H_{ \nicefrac{1}{2} } }^2
		.
		\end{equation} 
	\end{enumerate} 
\end{lemma}
\begin{proof}[Proof of Lemma~\ref{lemma:F_growth_estimates}]
	Note that items~\eqref{item:01 A} 
	and~\eqref{item:01 B}
	of
	Lemma~\ref{lemma:local_lipA}
	ensures that
	\begin{equation}
	\label{eq:EqualSpaces}
	 H_{ \nicefrac{1}{2} } 
	 =
	 W_0^{1,2}((0,1), \R) 
	.
	\end{equation}
	Next observe that
	item~\eqref{item:001} of Lemma~\ref{lemma:local_lip} 
	shows that for every 
	$ v \in H_{ \nicefrac{1}{2} } $
	it holds that
	$ v^2 \in W^{1,2}((0,1), \R ) $.
	This,
	\eqref{eq:EqualSpaces},
	and
	Lemma~\ref{lemma:int_parts}
	(with
	$ u = u $,
	$ v = v^2 $
	for
	$ u, v \in H_{\nicefrac{1}{2}} $
	in the notation of 
	Lemma~\ref{lemma:int_parts})  
	ensure that for every
	$ u, v \in H_{\nicefrac{1}{2}} $
	it holds that
	$ \langle \partial ( v^2 ), u \rangle_H
	=
	-
	\langle v^2, \partial u \rangle_H $.
	Item~\eqref{item:001} of Lemma~\ref{lemma:local_lip} 
	and
	Lemma~\ref{lemma:InfEstimate}
	(with
	$ \alpha = \alpha - \frac{1}{2} $
	for 
	$ \alpha \in ( \frac{3}{4}, \infty) $
	in the notation of Lemma~\ref{lemma:InfEstimate})
	therefore
	prove that 
	for every 
	$ v \in H_{  \nicefrac{1}{2} } $, 
	$ \alpha \in ( \frac{3}{4}, \infty ) $
	it holds
	that
	\begin{equation}
	\begin{split}
	2
	\| F(v) \|_{H_{-\alpha}}
	& = 
	\| c_1 \partial (v^2) \|_{H_{-\alpha}}
	= 
	|c_1|
	\sup_{ u \in 
		  H_\alpha \backslash \{0\}  }
	\frac{ | \langle 
		\partial ( v^2 ), u \rangle_H |}{ \| u \|_{H_{ \alpha } } }
	\\
	&
	=
	|c_1|
	\sup_{ u \in 
			  H_\alpha \backslash \{0\}  }
	\frac{ | \langle v^2, \partial u \rangle_H |}{ \| u \|_{
			H_\alpha } }
	\\
	&
	\leq
	|c_1|
	\sup_{ u \in 
		  H_\alpha \backslash \{0\}   }
	\frac{ \| v^2 \|_{L^1( \lambda; \R)} 
		\| \partial u \|_{ L^\infty( \lambda; \R)}
	}{ \| u \|_{
			H_\alpha } }
	\\
	&
	\leq
|c_1|
	| c_0 |^{ - \alpha }
	\| v \|_H^2 
	\sqrt{2
		\sum_{n=1}^\infty 
		| \pi n |^{ 2 - 4 \alpha }
	}
	<
	\infty
	.
	\end{split}
	\end{equation}
	This establishes
	item~\eqref{item:Growth1}.
	Next note that 
	item~\eqref{item:001} of Lemma~\ref{lemma:local_lip}, 
	Lemma~\ref{lemma:DerivativeEstimate}
	(with
	$ \alpha = \alpha $
	for
	$ \alpha \in [0, \nicefrac{1}{2}] $
	in the notation of Lemma~\ref{lemma:DerivativeEstimate}),
	and
	Lemma~\ref{lemma:Sobolev_multiplication}
	(with
	$ s = 1 - 2 \alpha $,
	$ q = \nicefrac{ 2 ( 1 - \alpha ) }{ 3 } $,
	$ r = \nicefrac{ 2 ( 1 - \alpha ) }{ 3 } $
	for $ \alpha \in ( \nicefrac{1}{4}, \nicefrac{1}{2} ] $
	in the notation of Lemma~\ref{lemma:Sobolev_multiplication})  
%
%
%
%
%
%
%
%
	show that for every
	$ v \in H_{ \nicefrac{1}{2} } $,    
	$ \alpha \in ( \nicefrac{1}{4}, \nicefrac{1}{2} ] $
	it holds that 
	\begin{equation}
	\begin{split} 
	&
	2
	\| F(v) \|_{H_{-\alpha}} 
	= 
	|c_1|
	\| \partial (v^2) \|_{H_{-\alpha}}
	\leq
	|c_1|
	\bigg( 
	\sup_{ u  \in H_{ \nicefrac{1}{2} } \backslash \{0\} }
	\tfrac{ \| \partial (u^2) \|_{ H_{-\alpha} } }
	{ \| u^2 \|_{ W^{ 1 - 2 \alpha, 2 }( (0,1), \R ) } }
	\bigg)
	\| v^2 \|_{ W^{ 1 - 2 \alpha, 2 }( (0,1), \R ) }
	\\
	& 
	\leq
	|c_1|
	\bigg( 
	\sup_{ u  \in H_{ \nicefrac{1}{2} } \backslash \{0\} }
	\tfrac{ \| \partial (u^2) \|_{ H_{-\alpha} } }
	{ \| u^2 \|_{ W^{ 1 - 2 \alpha, 2 }( (0,1), \R ) } }
	\bigg)
	\bigg( 
	\sup_{ u \in H_{ \nicefrac{1}{2} } \backslash \{0\} }
	\tfrac{\| u^2 \|_{ { W^{ 1 - 2 \alpha, 2 }( (0,1), \R ) } } }{
		\| u \|_{ W^{ \nicefrac{2(1 -  \alpha)}{3}, 2  }( (0,1), \R )}^2	
	}
	\bigg)
	\| v \|_{ W^{ \nicefrac{2(1 -  \alpha)}{3}, 2  }( (0,1), \R )}^2
	< \infty.
	\end{split}
	\end{equation}
	Item~\eqref{item:Inclusion} of
	Lemma~\ref{lemma:NormEquivalency}
	hence implies item~\eqref{item:Growth2}.
	Furthermore, observe that 
	items~\eqref{item:norm_estimate}
	and~\eqref{item:norm_estimate_infty} 
	of Lemma~\ref{lemma:local_lipA}
	imply that for every
	$ v \in H_{ \nicefrac{1}{2} } $
	it holds that
	\begin{equation}
	\| F(v) \|_H 
	=
	|c_1|
	\| v \partial v \|_H
	\leq
	|c_1|
	\| v \|_{ L^{ \infty} ( \lambda; \R) }
	\| \partial v \|_H
	\leq
	\tfrac{ 
	|c_1|
}
{
\sqrt{3} \, c_0 
} 
	\| v \|_{ H_{ \nicefrac{1}{2} } }^2
	.
	\end{equation} 
	This establishes item~\eqref{item:Growth4}. 
	The proof of Lemma~\ref{lemma:F_growth_estimates} is thus completed.
\end{proof}
\begin{corollary}
	\label{corollary:SatisfiedF}
	%
Assume Setting~\ref{setting:Examples}
	and let  
	$ \alpha_1 \in ( \nicefrac{3}{4}, \infty) $,
	$ \alpha_2 \in ( \nicefrac{1}{4}, \nicefrac{1}{2} ] $.
	Then
	\begin{equation}
	\label{eq:finite fun}
	\Big[ 
	\sup\nolimits_{ v \in H_{ \nicefrac{1}{2} } \backslash \{0\} }
	\tfrac{ \| F(v) \|_{H  } }
	{ \| v \|_{H_{ \nicefrac{1}{2} } }^2 }
	\Big]
	+
	\Big[
	\sup\nolimits_{ v \in H_{ 
			\nicefrac{1}{2}   } \backslash \{ 0 \} }
	\tfrac{ \| F(v) \|_{H_{ - \alpha_2 } } }
	{ \| v \|_{ H_{  \nicefrac{ ( 1 - \alpha_2 ) }{3} } }^2 } 
	\Big]
	+
	\Big[ 
	\sup\nolimits_{ v \in H_{ \nicefrac{1}{2} } \backslash
		\{0\} } 
	\tfrac{ \| F(v) \|_{H_{-\alpha_1} } }{ \| v \|_H^2 }
	\Big]
	< \infty 
	.
	\end{equation}  
\end{corollary}
\begin{proof}[Proof of Corollary~\ref{corollary:SatisfiedF}]
	Observe that
	item~\eqref{item:Growth1} 
	of
	Lemma~\ref{lemma:F_growth_estimates}   
	(with 
	$ \alpha = \alpha_1 $ 
	in the notation of
	item~\eqref{item:Growth1} 
	of
	Lemma~\ref{lemma:F_growth_estimates})
	implies that
	\begin{equation}
	\begin{split} 
	\label{eq:First term}
	& 
	\Big[ 
	\sup\nolimits_{ v \in H_{ \nicefrac{1}{2} } \backslash
		\{0\} } 
	\tfrac{ \| F(v) \|_{H_{-\alpha_1} } }{ \| v \|_H^2 }
	\Big]
	< \infty
	.
	\end{split} 
	\end{equation}
	Next note that
	item~\eqref{item:Growth2} 
	of
	Lemma~\ref{lemma:F_growth_estimates}
	(with
	$ \alpha = \alpha_2 $
	in the notation of
	item~\eqref{item:Growth2} 
	of
	Lemma~\ref{lemma:F_growth_estimates})
	shows that 
	\begin{equation}
	\begin{split} 
	\label{eq:Second term} 
	\Big[
	\sup\nolimits_{ v \in H_{ 
			\nicefrac{1}{2}   } \backslash \{ 0 \} }
	\tfrac{ \| F(v) \|_{H_{ - \alpha_2 } } }
	{ \| v \|_{ H_{  \nicefrac{ ( 1 - \alpha_2 ) }{3} } }^2 } 
	\Big]
	< \infty
	.
	\end{split}
	\end{equation}
	Moreover, observe that
	item~\eqref{item:Growth4}
	Lemma~\ref{lemma:F_growth_estimates} 
	ensures that
	\begin{equation}
	\label{eq:Last term}
	\Big[ 
	\sup\nolimits_{ v \in H_{ \nicefrac{1}{2} } \backslash \{0\} }
	\tfrac{ \| F(v) \|_{H  } }
	{ \| v \|_{H_{ \nicefrac{1}{2} } }^2 }
	\Big]
	<
	\infty .
	\end{equation}
	Combining~\eqref{eq:First term}
	and~\eqref{eq:Second term} 
	therefore 
	establishes~\eqref{eq:finite fun}.
	The proof of Corollary~\ref{corollary:SatisfiedF}
	is thus completed. 
\end{proof}
\begin{lemma}
	\label{lemma:CrucialPropertiesBurgers} 
	Assume Setting~\ref{setting:Examples}.
	Then it holds for every
	$ x \in H_{ \nicefrac{1}{2} } $ that
	$ \langle x, F(x) \rangle_H = 0 $.
\end{lemma}
\begin{proof}[Proof of Lemma~\ref{lemma:CrucialPropertiesBurgers}]
	Note that
	items~\eqref{item:01 A} 
	and~\eqref{item:01 B} 
	of Lemma~\ref{lemma:local_lipA},
	item~\eqref{item:001} of
	Lemma~\ref{lemma:local_lip}, 
	%
	and 
	Lemma~\ref{lemma:int_parts}
	(with
	$ u = x $,
	$ v = x^2 $
	for 
	$ x \in H_{ \nicefrac{1}{2} } $
	in the notation of
	Lemma~\ref{lemma:int_parts}) 
	ensure that for every
	$ x \in H_{\nicefrac{1}{2}} = W^{1,2}_0( (0,1), \R) $ 
	it holds that
	$ x^2 \in W^{1,2}( (0,1), \R) $ 
	and
	\begin{equation}
	\begin{split}
	2 \langle x, F(x) \rangle_H
	&
	= 
	2
	c_1
	\langle x, x \partial x \rangle_H
	=
	c_1
	\langle x, \partial ( x^2 ) \rangle_H 
	\\&= 
	- 
	c_1
	\langle \partial x, x^2 \rangle_H  
	=
	-
	c_1
	\langle x \partial x, x \rangle_H 
	= 
	-
	\langle F(x), x \rangle_H
	.
	\end{split}
	\end{equation} 
	The proof of Lemma~\ref{lemma:CrucialPropertiesBurgers} 
	is thus completed.
\end{proof}
\begin{corollary}
	\label{corollary:AppropriateCoercivityEstimate}
	Assume Setting~\ref{setting:Examples}
	and let $ \iota \in (\nicefrac{1}{4}, \infty) $,
	$ v \in H_{ \nicefrac{1}{2} } $,
	$ w \in H_{ \max \{ \nicefrac{1}{2}, \iota \} } $.
	Then it holds that
	\begin{equation}
	\begin{split}
	&
	\langle v, F( v + w ) \rangle_H 
	\\
	&\leq
	\tfrac{ 3 | c_1 |^2 }{ 8  | c_0 |  } 
	\bigg[   
	\sup_{ u \in H_\iota \backslash \{ 0 \} }
	\tfrac{ \| u \|_{ L^\infty(\lambda; \R) } }{ \| u \|_{ H_\iota } }
	+
	\sup_{ u \in H_\iota \backslash \{ 0 \} }
	\tfrac{ \| u \|_{ L^4(\lambda; \R) }^2 }{ \| u \|_{ H_\iota }^2 } 
	\bigg]^2
	%
	\big(
	\| v \|_H^2
	+ 
	\| w \|_{ H_\iota }^2
	\big) 
	\| w \|_{ H_\iota }^2
	+  
	\| v \|_{H_{\nicefrac{1}{2}}}^2
	<
	\infty 
	.
	\end{split}
	\end{equation}
\end{corollary}
\begin{proof}[Proof of Corollary~\ref{corollary:AppropriateCoercivityEstimate}] 
	Throughout this proof assume 
	w.l.o.g.\ that
	$ c_1 \neq 0 $
	and
	let
	$ C \in [0, \infty] $
	satisfy that
	\begin{equation} 
	C
	=
	\sup_{ u \in H_\iota \backslash \{ 0 \} }
	\tfrac{ \| u \|_{ L^\infty(\lambda; \R) } }{ \| u \|_{ H_\iota } }
	+
	\sup_{ u \in H_\iota \backslash \{ 0 \} }
	\tfrac{ \| u \|_{ L^4(\lambda; \R) }^2 }{ \| u \|_{ H_\iota }^2 } 
	%
.
	\end{equation} 
	Note that the Sobolev embedding theorem
	and 
	item~\eqref{item:Inclusion} 
	of
	Lemma~\ref{lemma:NormEquivalency} 
	ensure that $ C \in (0, \infty) $.
	Next observe that
	Lemma~\ref{lemma:CrucialPropertiesBurgers}, 
	item~\eqref{item:001} of Lemma~\ref{lemma:local_lip},
	and
	Lemma~\ref{lemma:int_parts}
	(with
	$ u = v $,
	$ v = w^2 $
	in the notation of Lemma~\ref{lemma:int_parts})
	ensure that  
	\begin{equation}
	\begin{split}
	\langle v, F( v + w ) \rangle_H 
	&
	= 
	c_1
	\langle v, ( v + w ) ( \partial v + \partial w ) \rangle_H 
	\\
	&
	=
	c_1
	\langle v, v \partial v \rangle_H
	+
	c_1
	\langle v, w \partial v \rangle_H 
	+
	c_1
	\langle v, v \partial w \rangle_H 
	+
	c_1
	\langle v, w \partial w \rangle_H 
	\\
	&
	=
	c_1
	\langle v, w \partial v \rangle_H 
	+
	c_1
	\langle v, v \partial w \rangle_H 
	+
	\tfrac{ 
		c_1
	}{2}
	\langle v,  \partial ( w ^2 ) \rangle_H 
	\\
	&
	=
	c_1
	\langle v, w \partial v \rangle_H 
	+
	c_1
	\langle v^2, \partial w \rangle_H 
	-
	\tfrac{c_1}{2}
	\langle \partial v, w^2 \rangle_H 
	.
	\end{split} 
	\end{equation} 
	Lemma~\ref{lemma:int_parts}
	(with
	$ u = w $,
	$ v = v^2 $
	in the notation of Lemma~\ref{lemma:int_parts})
	and
	item~\eqref{item:001} of Lemma~\ref{lemma:local_lip}
	therefore imply that
	\begin{equation} 
	\begin{split} 
	\langle v, F( v + w ) \rangle_H
	&
	=
	c_1
	\langle v \partial v, w \rangle_H 
	-
	2
	c_1
	\langle v \partial v, w \rangle_H 
	-
	\tfrac{c_1}{2}
	\langle \partial v, w^2 \rangle_H
	\\
	&
	=
	-
	c_1
	\langle v \partial v, w \rangle_H 
	-
	\tfrac{c_1}{2}
	\langle \partial v, w^2 \rangle_H 
	\\
	&
	\leq 
	| c_1 |
	| \langle v \partial v, w \rangle_H |
	+
	\tfrac{|c_1|}{2}
	| \langle \partial v, w^2 \rangle_H |
	.
	\end{split} 
	\end{equation}
	H\"older's inequality
	and
	item~\eqref{item:norm_estimate} 
	of Lemma~\ref{lemma:local_lipA}
	hence prove that
	\begin{equation} 
	\begin{split}
	\langle v, F( v + w ) \rangle_H   
	&\leq
	\tfrac{| c_1 |}{2} 
	\big( 
	2
	\| v \|_H
	\| \partial v \|_H
	\| w \|_{ L^\infty(\lambda; \R) }
	+
	\| \partial v \|_H
	\| w \|_{ L^4(\lambda; \R)}^2
	\big) 
	\\
	&
	\leq 
	\tfrac{| c_1 |}{2 |c_0|^{ \nicefrac{1}{2} } } 
	\big( 
	2
	\| v \|_H
	\| v \|_{H_{\nicefrac{1}{2}}}
	\| w \|_{ L^\infty(\lambda; \R) }
	+
	\| v \|_{H_{\nicefrac{1}{2}}}
	\| w \|_{ L^4(\lambda; \R)}^2
	\big)
	\\
	&
	\leq
	\tfrac{| c_1 |  }{2 |c_0|^{  \nicefrac{1}{2} } } 
	%
	\bigg( 
	2
	\| v \|_H
	\| v \|_{H_{\nicefrac{1}{2}}}
	\Big[ 
	\sup\nolimits_{ u \in H_\iota \backslash \{ 0 \} }
	\tfrac{ \| u \|_{ L^\infty(\lambda; \R) } }{ \| u \|_{ H_\iota } }
	\Big]
	\| w \|_{ H_\iota }
	\\
	&
	\qquad\qquad\quad
	+
	\| v \|_{H_{\nicefrac{1}{2}}}
	\Big[ 
	\sup\nolimits_{ u \in H_\iota \backslash \{ 0 \} }
	\tfrac{ \| u \|_{ L^4(\lambda; \R) } }{ \| u \|_{ H_\iota } }
	\Big]^2
	\| w \|_{ H_\iota }^2
	\bigg) 
	\\
	&
	\leq
	\tfrac{| c_1 | C  }{2 |c_0|^{  \nicefrac{1}{2} } } 
	%
	\Big( 
	2
	\| v \|_H
	\| v \|_{H_{\nicefrac{1}{2}}}
	\| w \|_{ H_\iota }
	+
	\| v \|_{H_{\nicefrac{1}{2}}}
	\| w \|_{ H_\iota }^2
	\Big) 
	.
	\end{split}
	\end{equation} 
	The fact that
		for every
	$ x, y \in \R $, 
	$ \varepsilon \in (0, \infty) $
	it holds that
	$ 2 x y \leq \tfrac{x^2}{\varepsilon}  + \varepsilon y^2 $
	therefore 
	shows that
	\begin{equation} 
	\begin{split}
	&
	\langle v, F( v + w ) \rangle_H   
	\\
	&
	\leq 
	\tfrac{| c_1 | C }{2 |c_0|^{  \nicefrac{1}{2} } } 
	%
	\Big( 
	\big[ 
	\tfrac{ 3 | c_1 | C}{ 4  | c_0 |^{  \nicefrac{1}{2} }  }
	\| v \|_H^2
	\| w \|_{ H_\iota }^2
	+ 
	\tfrac{4 |c_0|^{ \nicefrac{1}{2} }}{ 3 | c_1 | C}
	\| v \|_{H_{\nicefrac{1}{2}}}^2
	\big]
	%
	%
	+
	\tfrac{1}{2}
	\big[ 
	\tfrac{4 |c_0|^{ \nicefrac{1}{2} }}{ 3 | c_1 | C}
	\| v \|_{H_{\nicefrac{1}{2}}}^2
	+
	\tfrac{ 3 | c_1 | C }{ 4  | c_0 |^{  \nicefrac{1}{2} }  } 
	\| w \|_{ H_\iota }^4
	\big] 
	\Big)
	\\
	&
	=    
	\tfrac{ 3 | c_1 |^2 C^2 }{ 8  | c_0 |  }
	\| v \|_{ H }^2
	\| w \|_{ H_\iota }^2
	+ 
	\tfrac{2 }{ 3 }
	\| v \|_{H_{\nicefrac{1}{2}}}^2
	+
	\tfrac{1}{ 3}
	\| v \|_{H_{\nicefrac{1}{2}}}^2
	+
	\tfrac{ 3 | c_1 |^2 C^2 }{ 16  | c_0 |  } 
	\| w \|_{ H_\iota }^4
	\\
	&
	=
	\tfrac{ 3 | c_1 |^2 C^2 }{ 8  | c_0 |  }
	\| v \|_H^2
	\| w \|_{ H_\iota }^2
	+
	\tfrac{ 3 | c_1 |^2 C^2 }{ 16  | c_0 |  } 
	\| w \|_{ H_\iota }^4
	+  
	\| v \|_{H_{\nicefrac{1}{2}}}^2
	\\
	&
	\leq
	\tfrac{ 3 | c_1 |^2 C^2 }{ 8  | c_0 |  } 
	\big(  
	\| v \|_H^2
	+ 
	\| w \|_{ H_\iota }^2
	\big)
	\| w \|_{ H_\iota }^2
	+  
	\| v \|_{H_{\nicefrac{1}{2}}}^2
	.
	\end{split} 
	\end{equation} 
	The proof of  Corollary~\ref{corollary:AppropriateCoercivityEstimate}
	is thus completed.
\end{proof} 
\section{Existence and uniqueness of mild solutions to stochastic Burgers equations}
\label{section:Existence} 

In this section  
we prove  
in
Theorem~\ref{theorem:existence_Burgers}
below
the unique
existence of suitably regular mild solutions
to stochastic Burgers equations
with additive trace class noise.
%
%
%
To do so, 
we first  
establish 
in
Lemmas~\ref{lemma:AprioriBound}--\ref{lemma:GalerkinRegularity}
(cf., e.g., Bl\"omker \& Jentzen~\cite[Lemma~5.5]{BloemkerJentzen2013}),
Lemma~\ref{lemma:ConvergenceSpeed}
(cf., e.g., Kloeden \& Neuenkirch~\cite[Lemma~2.1]{KloedenNeuenkirch2007}),
and
Lemma~\ref{lemma:PathwiseRates} 
(cf., e.g., Bl\"omker \& Jentzen~\cite[Lemma~4.3]{BloemkerJentzen2013})
a
few elementary and partially well-known auxiliary results.
Only for the sake of completeness we include in this section also a proof of Lemma~\ref{lemma:ConvergenceSpeed}.
Thereafter, we combine these
auxiliary results
with
the results 
from 
Subsection~\ref{subsection:Nonlinearity}
and
the abstract existence and uniqueness result
in Bl\"omker \& Jentzen~\cite[Theorem~3.1]{BloemkerJentzen2013}  
to establish
in
Theorem~\ref{theorem:existence_Burgers}
below
the main result of this article.
\begin{lemma}
	\label{lemma:AprioriBound}
	Assume Setting~\ref{setting:Examples},
	let
	$ T \in (0,\infty) $,    
	$ \iota \in (\nicefrac{1}{4}, \infty) $,
	$ \xi \in H $,
	let
	$ I \subseteq \H $
	be a finite set, 
	let 
	$ P \in L(H) $ 
	satisfy
	for every 
	$ v \in H $ 
	that
	$ Pv =\sum_{h \in I} \left< h, v \right>_H h $,
	and
	let	
	$ O, X \in \mathcal{C} ( [0,T], P(H) ) $
	satisfy for every
	$ t \in [0,T] $
	that 
	\begin{equation}
	X_t 
	= 
	e^{tA} P \xi 
	+
	\int_0^t e^{(t-s)A} P F( X_s ) \, ds 
	+  
	O_t 
	.
	\end{equation}
	Then it holds
	for every
	$ t \in [0,T] $ 
	that 
	\begin{equation} 
	\begin{split} 
	&
	\| X_t \|_H 
	\leq 
	\| O_t \|_H
	\\
	&
	+
	\bigg(
	\| \xi \|_H^2
	+ 
	\tfrac{ 3 | c_1 |^2 }{ 8  | c_0 |  } 
	\bigg[ 
	\sup_{ u \in H_\iota \backslash \{ 0 \} }
	\tfrac{ \| u \|_{ L^\infty(\lambda; \R) } }{ \| u \|_{ H_\iota } } 
	+ 
	\sup_{ u \in H_\iota \backslash \{ 0 \} }
	\tfrac{ \| u \|_{ L^4(\lambda; \R) }^2 }{ \| u \|_{ H_\iota }^2 } 
	\bigg]^2 
	\Big[
	1
	+
	\sup_{u \in [0,T] }
	\| O_u \|_{ H_\iota }^2 
	\Big]^2
	T
	\bigg)^{\frac{1}{2}}
	\\
	&
	\quad
	\cdot
	\exp\! 
	\bigg(
	\tfrac{ 3 | c_1 |^2 }{ 16  | c_0 |  } 
	\bigg[ 
	\sup_{ u \in H_\iota \backslash \{ 0 \} }
	\tfrac{ \| u \|_{ L^\infty(\lambda; \R) } }{ \| u \|_{ H_\iota } } 
	+ 
	\sup_{ u \in H_\iota \backslash \{ 0 \} }
	\tfrac{ \| u \|_{ L^4(\lambda; \R) }^2 }{ \| u \|_{ H_\iota }^2 } 
	\bigg]^2  
	\Big[
	1
	+
	\sup_{u \in [0,T] }
	\| O_u \|_{ H_\iota }^2 
	\Big]^2
	T
	\bigg) 
	<
	\infty 
	.
	\end{split} 
	\end{equation}
	%
	%
	%
	%
	%
\end{lemma}
\begin{proof}[Proof of Lemma~\ref{lemma:AprioriBound}]
Throughout this proof let 
$ C \in [0, \infty] $
satisfy that
\begin{equation} 
C
=
	\tfrac{ 3 | c_1 |^2 }{ 8  | c_0 |  } 
\bigg[ 
\sup_{ u \in H_\iota \backslash \{ 0 \} }
\tfrac{ \| u \|_{ L^\infty(\lambda; \R) } }{ \| u \|_{ H_\iota } }
+
\sup_{ u \in H_\iota \backslash \{ 0 \} }
\tfrac{ \| u \|_{ L^4(\lambda; \R) }^2 }{ \| u \|_{ H_\iota }^2 } 
\bigg]^2
\end{equation} 
and let
$ Z  \colon [0,T] \to P(H) $
be the function
which satisfies for every
$ t \in [0,T] $ that
$ Z_t = X_t - O_t $.
Observe that the Sobolev embedding theorem
and 
item~\eqref{item:Inclusion} 
of
Lemma~\ref{lemma:NormEquivalency} 
ensure that $ C \in [0, \infty) $.
	Next note that
	for every $ t \in [0,T] $ it holds that
	\begin{equation}
	\begin{split}
	Z_t = e^{tA} P \xi + \int_0^t e^{(t-s)A} P F( Z_s + O_s ) \, ds.
	\end{split}
	\end{equation}
	This implies for every
	$ t \in [0,T] $ that
	\begin{equation}
	Z_t = P \xi + \int_0^t [  A Z_s + P F(Z_s + O_s ) ] \, ds  
	.
	\end{equation}
	Therefore, we obtain that for every $ t \in [0,T] $ 
	it holds that
	\begin{equation}
	\begin{split}
	\| Z_t \|_H^2
	&
	=
	\| P \xi \|_H^2
	+
	2 
	\int_0^t 
	\langle 
	Z_s , 
	A Z_s   
	+ 
	P F ( Z_s + O_s )
	\rangle_H
	\,
	ds
	\\
	&
	\leq	
	\| \xi \|_H^2
	+
	2 
	\int_0^t 
	\langle 
	Z_s, 
	A Z_s 
	+ 
	F ( Z_s + O_s )
	\rangle_H
	\,
	ds.
	\end{split} 
	\end{equation}
	Corollary~\ref{corollary:AppropriateCoercivityEstimate}
	(with
	$ \iota = \iota $,
	$ v = Z_s $, $ w = O_s $
	for 
	$ s \in [0,T] $
	in the notation of
	Corollary~\ref{corollary:AppropriateCoercivityEstimate})
	hence proves that for every $ t \in [0,T] $
	it holds that
	\begin{equation}
	\begin{split}
	\| Z_t \|_H^2
	&
	\leq 
	\| \xi \|_H^2
	+
	2 
	\int_0^t 
	\tfrac{C}{2}
	%
	\big[ 
	\| Z_s \|_H^2
	+ 
	\| O_s \|_{ H_\iota }^2
	\big] 
	\| O_s \|_{ H_\iota }^2 
	\,
	ds
	\\
		&
	\leq 
	\| \xi \|_H^2
	+ 
	C
	\Big[
	1
	+
	\sup_{u \in [0,T] }
	\| O_u \|_{ H_\iota }^2 
	\Big]^2
	\int_0^t
	\big[ 
	1
	+
	\| Z_s \|_H^2 
	\big]  
	\,
	ds
	.
%
	%
	%
	\end{split} 
	\end{equation}
	The fact that
	$ O, Z \in \mathcal{C}([0,T], P(H) ) $ and Gronwall's lemma
	therefore establish that for every $ t \in [0,T] $ it holds 
	that
	\begin{equation} 
	\begin{split} 
    \| Z_t \|_H^2
	&\leq
	\Big( 
	\| \xi \|_H^2
	+ 
	C 
	\big[
	1
	+
	\sup\nolimits_{u \in [0,T] }
	\| O_u \|_{ H_\iota }^2 
	\big]^2
	T
	\Big)
	%
	\exp\! 
	\bigg( 
	C 
	\big[
	1
	+
	\sup\nolimits_{u \in [0,T] }
	\| O_u \|_{ H_\iota }^2 
	\big]^2
	T
	\bigg) 
	. 
	\end{split} 
	\end{equation}
	This completes the proof of Lemma~\ref{lemma:AprioriBound}.
\end{proof}
\begin{lemma}
	\label{lemma:ContractiveProjection} 
	Assume Setting~\ref{setting:main},   
	let
	$ \alpha \in \R $, 
	$ I \subseteq \H $,
	and let
	$ R \colon H_{ \max \{ \alpha, 0 \} } \to H_\alpha $
	be the function
	which satisfies 
	for every 
	$ v \in  H_{ \max \{ \alpha, 0 \} } $ 
	that 
	$ R v = \sum_{h \in I } \langle h, v \rangle_H h $.
	Then  
	\begin{enumerate}[(i)] 
		\item \label{item:Exstension1}
		it holds that there
		exists $ P \in L( H_\alpha ) $
		which satisfies for every 
		$ v \in  H_{ \max \{ \alpha, 0 \} } $ that $ P v = R v $
		and
		\item \label{item:Extension2}
		it holds that
		$ \| P \|_{ L (H_{ \alpha } ) } \leq 1 $.
	\end{enumerate}
\end{lemma}
\begin{proof}[Proof of Lemma~\ref{lemma:ContractiveProjection}]
	Note that for every 
	$ v \in H_{ \max \{ \alpha, 0 \} } $ it holds that
	\begin{equation}
	\label{eq:NormEstimate1}
	\| Rv \|_{ H_{ \max \{ \alpha, 0 \} }  }^2
	= 
	\sum_{ h \in I }
	| \langle h, v \rangle_H |^2 | \values_h |^{ 2 \max \{ \alpha, 0 \} }
	\leq
	\sum_{ h \in \H }
	| \langle h, v \rangle_H |^2 | \values_h |^{ 2 \max \{ \alpha, 0 \} }
	=
	\| v \|_{ H_{ \max \{ \alpha, 0 \} } }^2
	.
	\end{equation}
	Furthermore, observe that
	the fact that 
	$ \forall \, v \in H_{ \max \{\alpha, 0 \} } \colon R v \in H $
	ensures that 
	for every
	$ v \in H_{ \max \{\alpha, 0 \} } $
	it holds that
	\begin{equation}
	\begin{split} 
	\label{eq:NormEstimate2}
	& 
	\| R v \|_{ H_{ \min \{ \alpha, 0 \} }  }^2  
	=
	\| ( -A)^{ \min \{ \alpha, 0 \}  } R v \|_H^2
	=
	\Big\|
	\sum_{ h \in \H }
	| \values_h |^{ \min \{ \alpha, 0 \}  } 
	\langle h, R v \rangle_H h 
	\Big\|_H^2
	\\
	&
	= 
	\sum_{ h \in I }
	| \langle h, v \rangle_H |^2 | \values_h |^{ 2 \min \{ \alpha, 0 \} }
	\leq
	\sum_{ h \in \H }
	| \langle h, v \rangle_H |^2 | \values_h |^{ 2 \min \{ \alpha, 0 \} }
	=
	\| v \|_{ H_{ \min \{ \alpha, 0 \} } }^2
	.
	\end{split}
	\end{equation}
	Combining
	this and~\eqref{eq:NormEstimate1} 
	proves that for every 
	$ v \in H_{ \max \{ \alpha, 0 \} } $
	it holds that
	\begin{equation}
	\| R v \|_{ H_{ \alpha }  }
	\leq
	\| v \|_{H_\alpha}
	.
	\end{equation}
	The fact that $ H_{ \max \{ \alpha, 0 \} } \subseteq H_\alpha $ 
	densely
	therefore
	establishes
	items~\eqref{item:Exstension1} and~\eqref{item:Extension2}.
	The proof of Lemma~\ref{lemma:ContractiveProjection}
	is thus completed.
\end{proof}
\begin{lemma}
	\label{lemma:AprioriBound3}
	Assume Setting~\ref{setting:Examples}, 
	let $ \mathcal{P}(\H) $ be the power set of $ \H $, 
	let
	$ T \in (0,\infty) $,    
	$ \iota \in [0, 1 ) $, 
	$ \gamma \in (\nicefrac{1}{4}, \infty) $,
	$ \xi \in H_\iota $,  
	$ \mathcal{P}_0(\H) = \{ \theta \in \mathcal{P}(\H) \colon  \theta \text{ is a finite set} \} $, 
	let 
	$ ( P_I )_{ I \in \mathcal{P} (\H) } \subseteq L(H) $ 
	satisfy
	for every
	$ I \in \mathcal{P}(\H) $,  
	$ v \in H $ 
	that
	$ P_I(v) =\sum_{h \in I} \left< h ,v \right>_H h $,
	let
	$ O^I \in \mathcal{C} ( [0,T], P_I(H) ) $,
	$ I \in  \mathcal{P}_0 ( \H ) $,
	satisfy 
	$
	\sup\nolimits_{ I \in  \mathcal{P}_0 ( \H ) }
	\sup\nolimits_{ u \in [0,T] } 
	\| O_u^I \|_{ H_{ \max \{ \gamma, \iota \} } } 
	<
	\infty $,
	let 
	$ X^I \in \mathcal{C} ( [0,T], P_I(H) ) $,
	$ I \in  \mathcal{P}_0 ( \H ) $,
	and assume
	for every 
	$ I \in  \mathcal{P}_0 ( \H ) $,
	$ t \in [0,T] $
	that 
	\begin{equation}
	X_t^I 
	= 
	e^{tA} P_I \xi 
	+
	\int_0^t e^{(t-s)A} P_I F( X_s^I ) \, ds 
	+  
	O_t^I 
	.
	\end{equation}
	Then it holds that
	\begin{equation}
	\begin{split}  
	\label{eq:Finite iota}
	&  
	\sup\nolimits_{ I \in  \mathcal{P}_0 ( \H ) }
	\sup\nolimits_{ t \in [0,T] } 
	\| X_t^I \|_{ H_\iota }
	< \infty
	.
	\end{split}
	\end{equation}
\end{lemma}
\begin{proof}[Proof of Lemma~\ref{lemma:AprioriBound3}]
Note that
Corollary~\ref{corollary:SatisfiedF}
(with 
$ \alpha_1 = \alpha_1 $,
$ \alpha_2 = \alpha_2 $ 
for  
$ \alpha_1 \in (\nicefrac{3}{4}, \infty) $,
$ \alpha_2 \in (\nicefrac{1}{4}, \nicefrac{1}{2} ] $ 
in the notation of Corollary~\ref{corollary:SatisfiedF})
shows that
for every 
$ \alpha_1 \in (\nicefrac{3}{4}, \infty) $,
$ \alpha_2 \in (\nicefrac{1}{4}, \nicefrac{1}{2}] $ 
it holds that
\begin{equation}
\label{eq:Finite!! basic}
\Big( 
\sup\nolimits_{ v \in H_{ \nicefrac{1}{2} } \backslash \{0\} }
\tfrac{ \| F(v) \|_{H  } }
{ \| v \|_{H_{ \nicefrac{1}{2} } }^2 }
\Big)
+
\Big(
\sup\nolimits_{ v \in H_{ \nicefrac{1}{2} } \backslash \{ 0 \} }
\tfrac{ \| F(v) \|_{H_{ - \alpha_2 } } }
{ \| v \|_{ H_{  \nicefrac{ ( 1 - \alpha_2 ) }{ 3 } } }^2 }
\Big)
+
\Big( 
\sup\nolimits_{ v \in H_{ \nicefrac{1}{2} } \backslash
	\{0\} } 
\tfrac{ \| F(v) \|_{H_{-\alpha_1} } }{ \| v \|_H^2 }
\Big)
< \infty 
.
\end{equation}
In addition, observe that
Lemma~\ref{lemma:ContractiveProjection} 
(with
$ \alpha = - \alpha $,
$ I = I $,
$ R = ( H \ni x \mapsto P_I x \in H_{ - \alpha } ) $ 
for  
$ I \in \mathcal{P}_0(\H) $, 
$ \alpha \in \R $
in the notation of
Lemma~\ref{lemma:ContractiveProjection})
proves that
for every
$ x \in H $,
$ I \in \mathcal{P}_0(\H) $, 
$ \alpha \in \R $
it holds that
\begin{equation} 
\begin{split}
\| P_I x \|_{H_{-\alpha}}
\leq 
\| P_I \|_{ L( H_{-\alpha} ) }
\| x \|_{H_{-\alpha}}
\leq 
\| x \|_{H_{-\alpha}}
.
\end{split}
\end{equation}
Combining this and~\eqref{eq:Finite!! basic}
ensures that
for every
$ \alpha_1 \in (\nicefrac{3}{4}, \infty) $,
$ \alpha_2 \in (\nicefrac{1}{4}, \nicefrac{1}{2}] $ 
it holds that
\begin{equation}
\label{eq:Finite 1}
\sup\nolimits_{ I \in \mathcal{P}_0(\H) }
\Big( 
\sup\nolimits_{ v \in H_{ \nicefrac{1}{2} } \backslash \{0\} }
\tfrac{ \| P_I F(v) \|_{H  } }
{ \| v \|_{H_{ \nicefrac{1}{2} } }^2 }
\Big) <\infty,
\end{equation} 
\begin{equation}
\label{eq:Finite 2}
\qquad
\sup\nolimits_{ I \in \mathcal{P}_0(\H) }
\Big(
\sup\nolimits_{ v \in H_{ \nicefrac{1}{2} } \backslash \{ 0 \} }
\tfrac{ \| P_I F(v) \|_{H_{ - \alpha_2 } } }
{ \| v \|_{ H_{  \nicefrac{ ( 1 - \alpha_2 ) }{ 3 } } }^2 }
\Big) <\infty,
\end{equation} 
and
\begin{equation}
\label{eq:Finite 3}
\sup\nolimits_{ I \in \mathcal{P}_0(\H) }
\Big( 
\sup\nolimits_{ v \in H_{ \nicefrac{1}{2} } \backslash
	\{0\} } 
\tfrac{ \| P_I F(v) \|_{H_{-\alpha_1} } }{ \| v \|_H^2 }
\Big)
< \infty 
.
\end{equation}
Moreover, observe that 
Lemma~\ref{lemma:AprioriBound}
(with
$ T = T $,
$ \iota = \max \{ \gamma, \iota \} $, 
$ \xi = \xi $,
$ I = I $,
$ P = P_I $,
$ O = O^I $,
$ X = X^I $
for
$ I \in \mathcal{P}_0(\H) $ 
in the notation of 
Lemma~\ref{lemma:AprioriBound})
implies that
\begin{equation}
\label{eq:FiniteBasicMoment}
\sup\nolimits_{ I \in \mathcal{P}_0(\H) }
\sup\nolimits_{ t \in [0,T] } 
\| X_t^I \|_H 
<
\infty.
\end{equation}
Combining~\eqref{eq:Finite 3}
and
Lemma~\ref{lemma:F_Burgers_bootstrap20}
(with
$ ( \Omega, \F, \P ) 
=
( \{ 1 \}, \{ \emptyset, \{ 1 \} \}, ( \{ \emptyset, \{ 1 \} \} 
\ni A \mapsto \1_A(1) \in [0,1] ) ) $, 
$ T = T $,
$ \beta = \nicefrac{1}{2} $,
$ \gamma = \nicefrac{1}{2} $,
$ \xi = ( \{ 1 \} \ni \omega \mapsto P_I \xi \in H_{ \nicefrac{1}{2} } ) $,
$ F = ( H_{\nicefrac{1}{2}  } \ni v \mapsto P_I F(v) \in H) $,
$ \kappa = ( [0,T] \ni t \mapsto t \in [0,T] ) $,
$ Z = ( [0,T] \times \{ 1 \} \ni (t, \omega) \mapsto X_t^I \in H_{ \nicefrac{1}{2} } ) $,
$ O = ( [0,T] \times \{ 1 \} \ni (t, \omega) \mapsto O_t^I \in H_{ \nicefrac{1}{2} } ) $,
$ Y = ( [0,T] \times \{ 1 \} \ni (t, \omega) \mapsto X_t^I \in H ) $,
$ p = 1 $,
$ \rho = \rho $,
$ \alpha = \alpha_1 $ 
for
$ \alpha_1 \in (\nicefrac{3}{4}, 1 - \rho ) $,
$ \rho \in [0, \nicefrac{1}{4}) $,
$ I \in \mathcal{P}_0( \H ) $
in the notation of
Lemma~\ref{lemma:F_Burgers_bootstrap20})
hence shows that
for every 
$ \rho \in [0, \nicefrac{1}{4}) $,
$ \alpha_1 \in ( \nicefrac{3}{4}, 1 - \rho ) $,
$ I \in \mathcal{P}_0(\H) $,
$ t \in [0,T] $
it holds that
\begin{equation}
\begin{split} 
\| X_t^I \|_{ H_{ \rho } } 
&
\leq
\|
P_I \xi
\|_{ H_{ \rho } }
+ 
\| O_t^I \|_{ H_{ \rho } }
+
\tfrac{ T^{1-\alpha_1 - \rho} }{ 1 - \alpha_1 - \rho }
\Big(
\sup\nolimits_{ v \in H_{ \nicefrac{1}{2} } } 
\tfrac{ \| P_I F(v) \|_{H_{-\alpha_1} } }{ 1 + \| v \|_H^2 } 
\Big) 
\big( 1
+
\sup\nolimits_{ u \in [0,T] }
\| X_u^I \|_{H}^2
\big)   
.
\end{split}
\end{equation}
%
This,
\eqref{eq:Finite 3},
\eqref{eq:FiniteBasicMoment},
and the assumption that
$
\sup\nolimits_{ I \in  \mathcal{P}_0 ( \H ) }
\sup\nolimits_{ u \in [0,T] } 
\| O_u^I \|_{ H_\iota } 
<
\infty $
show
that for every
$ \rho \in [0,  \nicefrac{1}{4} ) $ 
with
$ \rho \leq \iota $
it holds that
\begin{equation}
\begin{split} 
\label{eq:FirstBoot}
\sup\nolimits_{ I \in \mathcal{P}_0(\H) }
\sup\nolimits_{ t \in [0,T] }
\| X_t^I \|_{ H_{ \rho } } 
< \infty. 
\end{split}
\end{equation}
Furthermore, observe that Lemma~\ref{lemma:F_Burgers_bootstrap0}
(with
$ H = H $,
$ ( \Omega, \F, \P ) 
=
( \{ 1 \}, \{ \emptyset, \{ 1 \} \}, ( \{ \emptyset, \{ 1 \} \} 
\ni A \mapsto \1_A(1) \in [0,1] ) ) $,
$ T = T $,
$ \beta = \nicefrac{1}{2} $,
$ \gamma = \nicefrac{1}{2} $,
$ \xi = ( \{ 1 \} \ni \omega \mapsto P_I \xi \in H_{ \nicefrac{1}{2} } ) $,
$ F = ( H_{\nicefrac{1}{2}  } \ni v \mapsto P_I F(v) \in H) $,
$ \kappa = ( [0,T] \ni t \mapsto t \in [0,T] ) $,
$ Z = ( [0,T] \times \{ 1 \} \ni (t, \omega) \mapsto X_t^I \in H_{ \nicefrac{1}{2} } ) $,
$ O = ( [0,T] \times \{ 1 \} \ni (t, \omega) \mapsto O_t^I \in H_{ \nicefrac{1}{2} } ) $,
$ Y = ( [0,T] \times \{ 1 \} \ni (t, \omega) \mapsto X_t^I \in H ) $,
$ p = 1 $,
$ \rho = \nicefrac{ ( 1 - \alpha_2 ) }{ 3 } $,
$ \eta = \eta $, 
$ \alpha_1 = \alpha_1 $,
$ \alpha_2 = \alpha_2 $ 
for  
$ \alpha_1 \in (\nicefrac{3}{4}, \nicefrac{ ( 2 + \alpha_2 )}{ 3 }  ) $,
$ \alpha_2 \in (\nicefrac{1}{4}, \nicefrac{1}{2} ) $,
$ \eta \in [\nicefrac{1}{4}, \nicefrac{1}{2} ] $,
$ I \in \mathcal{P}_0( \H ) $
in the notation of
Lemma~\ref{lemma:F_Burgers_bootstrap0}),
\eqref{eq:Finite 2}, 
and~\eqref{eq:Finite 3}
ensure that for every 
$ \alpha_2 \in (\nicefrac{1}{4}, \nicefrac{1}{2} ) $,
$ \alpha_1 \in (\nicefrac{3}{4}, \nicefrac{ ( 2 + \alpha_2 )}{ 3 }  ) $,
$ \eta \in [ \nicefrac{1}{4}, \nicefrac{1}{2} ] $,
$ I \in \mathcal{P}_0(\H) $,
$ t \in [0,T] $
it holds that
\begin{equation} 
\begin{split}  
\|
X_t^I
\|_{ H_{ \eta } }  
&
\leq
\|  
P_I \xi
\|_{ H_{ \eta } }
+ 
\| O_t^I \|_{ H_{ \eta } }
+
\tfrac{ T^{1 - \alpha_2 - \eta } }{ 1 - \alpha_2 - \eta }
\Big(
\sup\nolimits_{ v \in H_{ \nicefrac{1}{2} }  }
\tfrac{ \| P_I F(v) \|_{H_{ - \alpha_2 } } }
{ 1 + \| v \|_{ H_{ \nicefrac{ ( 1 - \alpha_2 ) }{ 3 } } }^2 }
\Big) 
\\
&
\quad 
\cdot 
\Big[
1
+
\|
\xi
\|_{ H_{ \nicefrac{ ( 1 - \alpha_2 ) }{ 3 } } }
+
\sup\nolimits_{ u \in [0,T] }
\| O_u^I \|_{ H_{ \nicefrac{ ( 1 - \alpha_2 ) }{ 3 } } }
\\
&
\qquad 
+
\tfrac{   T^{ 1-\alpha_1 - 
(  ( 1 - \alpha_2 ) / 3  )
 } }{ 1 - \alpha_1 - (   ( 1 - \alpha_2 ) / 3  ) }
\Big(
\sup\nolimits_{ v \in H_{ \nicefrac{1}{2} } } 
\tfrac{ \| P_I F(v) \|_{H_{-\alpha_1} } }{ 1 + \| v \|_H^2 } 
\Big) 
\big( 
1
+
\sup\nolimits_{ u \in [0,T] }
\| X_u^I
\|_H^2 
\big) 
\Big]^2
.
\end{split}
\end{equation}
%
%
Combining~\eqref{eq:Finite 2}--\eqref{eq:FiniteBasicMoment} 
and the assumption that
$
\sup\nolimits_{ I \in  \mathcal{P}_0 ( \H ) }
\sup\nolimits_{ u \in [0,T] } 
\| O_u^I \|_{ H_\iota } 
<
\infty $
hence implies that for every 
$ \eta \in [ \nicefrac{1}{4}, \nicefrac{1}{2} ] $  
with
$ \eta \leq \iota $
it holds
that
\begin{equation} 
\begin{split}  
\label{eq:SecondBoot}
\sup\nolimits_{ I \in \mathcal{P}_0(\H) }
\sup\nolimits_{ t \in [0,T] }
\| X_t^I \|_{ H_{ \eta } } 
< \infty. 
\end{split}
\end{equation}
Moreover, note that
Lemma~\ref{lemma:F_Burgers_bootstrap220} 
(with
$ H = H $,
$ ( \Omega, \F, \P ) 
=
( \{ 1 \}, \{ \emptyset, \{ 1 \} \}, ( \{ \emptyset, \{ 1 \} \} 
\ni A \mapsto \1_A(1) \in [0,1] ) ) $,
$ T = T $,
$ \beta = \kappa $,
$ \gamma = \nicefrac{1}{2} $,
$ \xi = ( \{ 1 \} \ni \omega \mapsto P_I \xi \in H_{ \kappa } ) $,
$ F = ( H_{\nicefrac{1}{2}  } \ni v \mapsto P_I F(v) \in H) $,
$ \kappa = ( [0,T] \ni t \mapsto t \in [0,T] ) $,
$ Z = ( [0,T] \times \{ 1 \} \ni (t, \omega) \mapsto X_t^I \in H_{ \nicefrac{1}{2} } ) $,
$ O = ( [0,T] \times \{ 1 \} \ni (t, \omega) \mapsto O_t^I \in H_{ \kappa } ) $,
$ Y = ( [0,T] \times \{ 1 \} \ni (t, \omega) \mapsto X_t^I \in H ) $,
$ p = 1 $,
$ \rho = \nicefrac{ ( 1 - \alpha_2 ) }{ 3 } $,
$ \eta = \nicefrac{1}{2} $,
$ \iota = \kappa $,
$ \alpha_1 = \alpha_1 $,
$ \alpha_2 = \alpha_2 $ 
for   
$ \alpha_1 \in [0, \nicefrac{ ( 2 + \alpha_2 )}{ 3 }  ) $,
$ \alpha_2 \in [0, \nicefrac{1}{2} ) $,
$ \kappa
\in 
[ \nicefrac{1}{2}, 1)  $,
$ I \in \mathcal{P}_0( \H ) $
in the notation of
Lemma~\ref{lemma:F_Burgers_bootstrap220})
and~\eqref{eq:Finite 1}--\eqref{eq:Finite 3} 
prove that
for every  
$ \alpha_2 \in ( \nicefrac{1}{4}, \nicefrac{1}{2} ) $,
$ \alpha_1 \in (\nicefrac{3}{4}, \nicefrac{ ( 2 + \alpha_2 )}{3}  ) $,
$ \kappa
 \in 
 [ \nicefrac{1}{2}, 1)  $,
$ I \in \mathcal{P}_0(\H) $,
$ t \in [0,T] $ 
it holds
that
\begin{equation}
\begin{split} 
\label{eq:result11}
\| X_t^I \|_{ H_{ \kappa } }
&
\leq
\| 
P_I \xi
\|_{ H_{ \kappa } }
+
\sup\nolimits_{u \in [0,T]}
\| 
O_u^I
\|_{ H_{ \kappa } }
+ 
\tfrac{ T^{1- \kappa } }{ 1 - \kappa }
\bigg(
\sup\nolimits_{ v \in H_{ \nicefrac{1}{2} }  }
\tfrac{ \| P_I F(v) \|_{H  } }
{ 1 + \| v \|_{H_{ \nicefrac{1}{2} } }^2 }
\bigg) 
\\
&
\quad 
\cdot
\bigg[
1
+
\|  
\xi
\|_{ H_{ \nicefrac{1}{2} } }
+
\sup\nolimits_{ u \in [0,T] }
\| O_u^I \|_{ H_{ \nicefrac{1}{2} } }
+
\tfrac{ T^{ ( 1 / 2 ) - \alpha_2 } }{ ( 1 / 2 ) - \alpha_2 }
\bigg(
\sup\nolimits_{ v \in H_{ \nicefrac{1}{2} }  }
\tfrac{  \| P_I F(v) \|_{H_{ - \alpha_2 } } }
{ 1 + \| v \|_{ H_{ \nicefrac{(1- \alpha_2 )}{3} } }^2 }
\bigg) 
\\
&
\quad 
\cdot 
\Big[
1
+
\|
\xi
\|_{ H_{ \nicefrac{(1- \alpha_2 )}{3} } }
+
\sup\nolimits_{ u \in [0,T] }
\| O_u^I \|_{ H_{ \nicefrac{(1- \alpha_2 )}{3} } }
\\
&
\qquad 
+
%
\tfrac{   T^{ 1-\alpha_1 - 
		(  ( 1 - \alpha_2 ) / 3  )
} }{ 1 - \alpha_1 - (   ( 1 - \alpha_2 ) / 3  ) }
\bigg(
\sup\nolimits_{ v \in H_{ \nicefrac{1}{2} }  } 
\tfrac{ \| P_I F(v) \|_{H_{-\alpha_1} } }{ 1 + \| v \|_H^2 } 
\bigg)
\sup\nolimits_{ u \in [0,T] }
\| X_u^I
\|_{H}^2 
\Big]^2
\bigg]^2
.
\end{split}
\end{equation}
Combining~\eqref{eq:Finite 1}--\eqref{eq:FiniteBasicMoment}
and the assumption that
$
\sup\nolimits_{ I \in  \mathcal{P}_0 ( \H ) }
\sup\nolimits_{ u \in [0,T] } 
\| O_u^I \|_{ H_\iota } 
<
\infty $
therefore 
assures that for every 
$ \kappa \in [\nicefrac{1}{2}, 1) $
with $ \kappa \leq \iota $
it holds that
\begin{equation} 
\begin{split}  
\label{eq:ThirdBoot}
\sup\nolimits_{ I \in \mathcal{P}_0(\H) }
\sup\nolimits_{ t \in [0,T] }
\| X_t^I \|_{ H_{ \kappa } } 
< \infty. 
\end{split}
\end{equation}
This, \eqref{eq:FirstBoot},
and~\eqref{eq:SecondBoot}
establish~\eqref{eq:Finite iota}.
The proof of Lemma~\ref{lemma:AprioriBound3}
is thus completed.
%
\end{proof} 
\begin{lemma}
	\label{lemma:ZContinuous}
	Assume Setting~\ref{setting:main}
%
	and 
	let  
	$ T \in (0, \infty) $, 
	$ \alpha \in (0,1) $,
	$ \gamma \in \R $,
	$ \mathcal{Z} \in \mathcal{C} ([0,T], H_\gamma) $.
	Then 
	\begin{enumerate}[(i)]
	\item \label{item:well defined}
	it holds for every $ t \in [0,T] $ that
	$ \int_0^t
	\|
	(t-u)^{\alpha - 1 }
	 e^{(t-u)A} 
	\mathcal{Z}_u \|_{H_\gamma}
	\, du < \infty $
	and
	\item \label{item:actual statement}
	it holds that  
	$ ( [0,T] \ni t \mapsto \int_0^t
	(t-u)^{\alpha - 1 }
	e^{(t-u)A} 
	\mathcal{Z}_u
	\, du \in H_\gamma ) \in 
	\mathcal{C}( [0,T], H_\gamma) $. 
	\end{enumerate}
\end{lemma}
\begin{proof}[Proof of Lemma~\ref{lemma:ZContinuous}]
	Note that for every
	$ t \in [0,T] $ it holds that
	\begin{equation}
	\begin{split}  
	\int_0^t
	\| 
	(t-u)^{\alpha - 1 }
	e^{(t-u)A} 
	\mathcal{Z}_u 
	\|_{H_\gamma}
	\, du
	& \leq 
	\int_0^t
	(t-u)^{\alpha - 1 }
	\| 
	\mathcal{Z}_u \|_{H_\gamma}
	\, du
	\\
	&
	\leq
	\tfrac{ t^{\alpha} }{ \alpha }
	\sup\nolimits_{ u \in [0,T] } 
	\| 
	\mathcal{Z}_u \|_{H_\gamma}
	< \infty 
	.
	\end{split} 
	\end{equation}
	This establishes item~\eqref{item:well defined}.
	Next observe that
	item~\eqref{item:well defined}
	ensures that there exists a function 
	$ Z \colon [0,T] \to H_\gamma $ 
	which satisfies for every
	$ t \in [0,T] $ that
	\begin{equation} 
	\label{eq:another Z}
	Z_t = \int_0^t
	(t-u)^{\alpha - 1 }
	e^{(t-u)A} 
	\mathcal{Z}_u
	\, du 
	.
	\end{equation} 
	%
	Note that~\eqref{eq:another Z}
	and the triangle inequality show
	that for 
	every 
	$ s \in [0,T] $,
	$ t \in [s,T] $
	it holds 
	that
	\begin{equation}
	\begin{split}
	\label{eq:FirstPart}
	&
	\| Z_t - Z_s \|_{H_\gamma}
	\\
	&
	\leq
	\Big\| 
	\int_s^t
	(t-u)^{\alpha - 1 }
	e^{(t-u)A} 
	\mathcal{Z}_u
	\, du
	\Big\|_{H_\gamma}
	+
	\Big\|
	\int_0^s
	\big(
	(t-u)^{\alpha - 1 }
	e^{(t-u)A} 
	-
	(s-u)^{\alpha - 1 }
	e^{(s-u)A} 
	\big)
	\mathcal{Z}_u
	\, du
	\Big\|_{H_\gamma}
	\\
	&
	\leq  
	\int_s^t
	(t-u)^{\alpha - 1 }
	\| 
	e^{(t-u)A} 
	\mathcal{Z}_u
	\|_{H_\gamma}
	\, du
	%
	+ 
	\int_0^s
	\big\|
	\big(
	(t-u)^{\alpha - 1 }
	e^{(t-u)A} 
	-
	(s-u)^{\alpha - 1 }
	e^{(s-u)A} 
	\big)
	\mathcal{Z}_u
	\big\|_{H_\gamma}
	\, du 
	.
	\end{split}
	\end{equation}
	Furthermore, observe that for every 
	$ s \in [0,T] $,
	$ t \in [s,T] $ it holds that
	\begin{equation}
	\begin{split}
	\label{eq:Est1}
	&
	\int_s^t
	(t-u)^{\alpha - 1 }
	\| 
	e^{(t-u)A} 
	\mathcal{Z}_u
	\|_{H_\gamma}
	\, du
	\leq 
	\int_s^t
	(t-u)^{\alpha - 1 }
	\|  
	\mathcal{Z}_u
	\|_{H_\gamma}
	\, du
		\\
		&
	\leq
	[ \sup\nolimits_{ u \in [0,T] }
	\| \mathcal{Z}_u \|_{H_\gamma} ]
	\int_s^t
	(t-u)^{\alpha - 1 } 
	\, du
	\leq
	[ \sup\nolimits_{ u \in [0,T] }
	\| \mathcal{Z}_u \|_{H_\gamma} ]
	\tfrac{ ( t - s )^{ \alpha }  }{ \alpha } 
	.
	\end{split}
	\end{equation}
	In addition, 
	note that the triangle inequality assures
	that for every 
	$ s \in [0,T] $,
	$ t \in [s,T] $
	it holds that
	\begin{equation}
	\begin{split}
	\label{eq:Est2}
	& 
	\int_0^s
	\big\|
	\big(
	(t-u)^{\alpha - 1 }
	e^{(t-u)A} 
	-
	(s-u)^{\alpha - 1 }
	e^{(s-u)A} 
	\big)
	\mathcal{Z}_u
	\big\|_{H_\gamma}
	\, du 
	\\
	&
	\leq 
	\int_0^s
	\big[ 
	(t-u)^{\alpha - 1 }
	\| 
	(
	e^{(t-u)A} 
	- 
	e^{(s-u)A} 
	)
	\mathcal{Z}_u
	\|_{H_\gamma}
	%
	+ 
	(
	(s-u)^{\alpha - 1 } 
	-
	(t-u)^{\alpha - 1 }
	)
	\|
	e^{(s-u)A} 
	\mathcal{Z}_u
	\|_{H_\gamma}
	\big] 
	\, du 
	\\
	&
	\leq 
	\int_0^s
	(t-u)^{\alpha - 1 }
	\|  
	e^{(s-u)A} 
	(
	e^{(t-s)A} 
	- 
	\operatorname{Id}_{H_\gamma}	
	)
	\mathcal{Z}_u
	\|_{H_\gamma}
	\,du
	\\
	&
	\quad 
	+
	\int_0^s
	(
	(s-u)^{\alpha - 1 } 
	-
	(t-u)^{\alpha - 1 }
	)
	\|
	\mathcal{Z}_u
	\|_{H_\gamma}
	\, du 
	.
	%
	%
	\end{split}
	\end{equation}
	Next observe that
	the fact that
		for every
	$ t \in (0,T] $, 
	$ s \in (0, t) $, 
	$ u \in [0, s) $
	it holds that
	$ ( t - u )^{ \alpha - 1 } \leq (t-s)^{\alpha-1} $
	proves that
	for every 
	$ s \in [0,T] $,
	$ t \in [s,T] $,
	$ \rho \in ( 1 - \alpha, 1 ) $
	it holds that
	\begin{equation} 
	\begin{split}
	\label{eq:Est4}
	&
	\int_0^s
	(t-u)^{\alpha - 1 }
	\|  
	e^{(s-u)A} 
	(
	e^{(t-s)A} 
	- 
	\operatorname{Id}_{H_\gamma}
	)
	\mathcal{Z}_u
	\|_{H_\gamma}
	\,du
	\\
	&
	\leq 
	\int_0^s
	(t-u)^{\alpha - 1 }
	\| 
	(-A)^{ \rho }
	e^{(s-u)A} 
	\|_{ L(H_\gamma) }
	\|
	( -A)^{ - \rho }
	(
	e^{(t-s)A} 
	- 
	\operatorname{Id}_{H_\gamma}
	)
	\|_{ L(H_\gamma) }
	\|
	\mathcal{Z}_u
	\|_{H_\gamma}
	\, du
	\\
	&
	\leq
	[ \sup\nolimits_{ u \in [0,T] } 
    \| \mathcal{Z}_u \|_{H_\gamma} ] 
	\int_0^s
	( t - u )^{ \alpha - 1 }
	( s - u )^{ - \rho }
	( t - s )^\rho
	\, du
	\\
	&
	\leq 
	(t-s)^{ \rho + \alpha - 1 }
	[ \sup\nolimits_{ u \in [0,T] } \| \mathcal{Z}_u \|_{H_\gamma} ]
	\int_0^s
	( s - u )^{  - \rho  }
	\, du
	=
	(t-s)^{ \rho + \alpha - 1 }
	[ \sup\nolimits_{ u \in [0,T] } \| \mathcal{Z}_u \|_{H_\gamma} ]
	\tfrac{ s^{ 1 - \rho } }{ 1 - \rho }
	.
	\end{split}
	\end{equation}
	Moreover, observe that the fact that 
		for every
	$ x, y \in [0,T] $,
	$ z \in [0,1] $
	it holds that
	$ | x^z - y^z | \leq |x - y|^z $
	ensures that for every 
	$ t \in (0,T] $, $ s \in (0,t) $,
	$ u \in [0,s) $
	it holds that
	\begin{equation}
	(s-u)^{ \alpha - 1 } - (t-u)^{ \alpha - 1 }
	\leq
	\tfrac{
		(t-s)^{ 1 - \alpha }
	}
	{
		(s-u)^{ 1 - \alpha } (t-u)^{ 1 - \alpha } 
	}
	.
	\end{equation}
	This implies that for every 
	$ s \in [0,T] $,
	$ t \in [s,T] $,
	$ \varepsilon \in ( 0,  \min \{ \nicefrac{\alpha}{ ( 2 (1-\alpha) ) }, \nicefrac{1}{2}, 1 - \alpha \} ) $ it holds that
	\begin{equation}
	\begin{split}
	\label{eq:Est3}
	&
	\int_0^s
	(
	(s-u)^{\alpha - 1 } 
	-
	(t-u)^{\alpha - 1 }
	) 
	\, du 
	\leq 
	( t - s )^{ 1 - \alpha }
	\int_0^s
	\tfrac{
		du
	}
	{
		(s-u)^{ 1 - \alpha } (t-u)^{ 1 - \alpha } 
	} 
	\\
	&
	=
	( t - s )^{ 1 - \alpha }
	\int_0^s 
	(s-u)^{ \alpha - 1 } 
	(t-u)^{ \alpha - 1 + \varepsilon  }
	(t - u )^{-\varepsilon} 
	\, du
	\\
	&
	\leq 
	( t - s )^{ 1 - \alpha }
	\int_0^s 
	(s-u)^{ \alpha - 1 } 
	(t-s)^{ \alpha - 1 + \varepsilon  }
	(t - u )^{- \varepsilon } 
	\, du
	\\
	&
	\leq
	( t - s )^{ \varepsilon }
	\int_0^s 
	(s-u)^{ \alpha - 1 }  
	(t - u )^{- \varepsilon } 
	\, du
	\\
	&
	\leq
	( t - s )^{ \varepsilon }
	\Big[ 
	\int_0^s
	(s-u)^{ ( \alpha - 1 ) (1 + 2 \varepsilon ) }
	\, du
	\Big]^{ \nicefrac{1}{ ( 1 + 2 \varepsilon ) } }
	\Big[ 
	\int_0^s
	(t-u)^{ - \nicefrac{ \varepsilon ( 1 + 2 \varepsilon)}{ 2 \varepsilon } }
	\, du
	\Big]^{ \nicefrac{ 2 \varepsilon }{ ( 1 + 2 \varepsilon ) } }
	\\
	&
	\leq 
	( t - s )^{ \varepsilon }
	\Big[ 
	\tfrac{ 
		s^{ 1 + ( \alpha - 1 ) ( 1 + 2 \varepsilon ) }
	}
	{
		1 + ( \alpha - 1 )  ( 1 + 2 \varepsilon )
	}
	\Big]^{ \nicefrac{1}{ ( 1 + 2 \varepsilon ) } }
	%
	\Big[ 
	\int_0^t
	( t - u )^{ - \varepsilon  - ( 1/2 ) }
	\, du
	\Big]^{  \nicefrac{ 2 \varepsilon  }{  ( 1 + 2 \varepsilon)} }
	\\
	&
	= 
	( t - s )^{ \varepsilon }
	%
	%
	\Big[ 
	\tfrac{ 
		s^{ 1 + ( \alpha - 1 ) ( 1 + 2 \varepsilon) }
	}
	{
		1 + ( \alpha - 1 ) ( 1 + 2 \varepsilon)
	}
	\Big]^{ \nicefrac{1}{( 1 + 2 \varepsilon)} }
	\Big[ 
	\tfrac{ 
		t^{ ( 1/2 ) - \varepsilon }
	}
	{
	( 1/2 ) - \varepsilon
	}
	\Big]^{ \nicefrac{ 2 \varepsilon  }{   ( 1 + 2 \varepsilon)} }
	.
	\end{split}
	\end{equation}
	Combining~\eqref{eq:FirstPart}--\eqref{eq:Est4}
	therefore
	demonstrates that
	for every 
	$ s \in [0,T] $, $ t \in [s,T] $, 
	$ \rho \in ( 1 - \alpha, 1 ) $,
	$ \varepsilon \in ( 0,  \min \{ \nicefrac{\alpha}{ ( 2 (1-\alpha) ) }, \nicefrac{1}{2}, 1 - \alpha \} ) $
	it holds that
	\begin{equation}
	\begin{split} 
	&
	\| Z_t - Z_s \|_{H_\gamma}
	\leq 
	\sup\nolimits_{ u \in [0,T] }
	\| \mathcal{Z}_u \|_{H_\gamma}
	\\
	&
	\cdot 
	\bigg[ 
	\tfrac{ ( t - s )^{ \alpha }  }{ \alpha } 
	+
	(t-s)^{ \rho + \alpha - 1 } 
	\tfrac{ s^{ 1 - \rho } }{ 1 - \rho }
	%
	+
	( t - s )^{ \varepsilon }
	%
	%
	\Big[ 
	\tfrac{ 
		\max \{ T, 1 \}
	}
	{
		1 + ( \alpha - 1 ) ( 1 + 2 \varepsilon)
	}
	\Big]^{ \nicefrac{1}{( 1 + 2 \varepsilon)} }
	\Big[ 
	\tfrac{ 
		\max \{ T, 1 \}
	}
	{
		( 1 / 2 )  - \varepsilon  
	}
	\Big]^{ \nicefrac{ 2 \varepsilon }{ ( 1 + 2 \varepsilon)} }
	\bigg]
	.
	\end{split} 
	\end{equation} 
	This
	establishes item~\eqref{item:actual statement}.
	The proof of Lemma~\ref{lemma:ZContinuous}
	is thus completed.
\end{proof}
\begin{lemma}
	\label{lemma:Justify}
	%
Assume Setting~\ref{setting:main},
	let  
	$ T \in (0, \infty) $, 
	%
	%
	$ \beta \in \R $,
	$ \gamma \in ( -\infty, \nicefrac{1}{2} + \beta ) $,
	$ B \in \HS( H, H_\beta ) $,
	let
	$ ( \Omega, \F, \P ) $
	be a probability space,
	for every set $ R $
	and every function
	$ f \colon \Omega \to R $
	let
	$ [f]_{\P, \B(H_\gamma)} =
	\{
	g \in \M(\F,
	\B(H_\gamma))
	\colon
	(
	\exists\, D \in \F \colon
	\P( D ) = 0 \,\,\text{and}\,\,
	\{\omega \in \Omega \colon
	f(\omega) \neq g(\omega)
	\}
	\subseteq D
	)
	\} $,   
	and
	let
	$ (W_t)_{t\in [0,T]} $
	be an
	$ \operatorname{Id}_H $-cylindrical 
	 Wiener process.
	Then there exists an up to
	 indistinguishability unique
	stochastic process
	 $ O \colon [0,T] \times \Omega \to H_\gamma $  
	with continuous sample paths
	which satisfies for every 
	$ t \in [0,T] $ 
	that
	$ [O_t]_{\P, \B( H_\gamma )} = \int_0^t e^{(t-s)A} B \, dW_s $.
	%
\end{lemma}
\begin{proof}[Proof of Lemma~\ref{lemma:Justify}]
	%
	%
	Note that 
	the fact that
	$ \gamma - \beta < \nicefrac{1}{2} $
	ensures that
	for every
	$ t \in [0,T] $
	it holds that
	\begin{equation}
	\begin{split} 
	&
	\int_0^t \| e^{(t-s)A} B \|_{ \HS(H, H_\gamma) }^2 \, ds
	= 
	\int_0^t \| 
	(-A)^{ \min \{ 0, \gamma - \beta \} }
	(-A)^{ \max \{ 0, \gamma - \beta \} }
	e^{(t-s)A} 
	B 
	\|_{ \HS(H, H_{ \beta } ) }^2 \, ds
	\\
	&
	\leq 
	\| (-A)^{ \min \{ 0, \gamma - \beta \} } \|_{L(H)}^2
	\int_0^t
	\| (-A)^{\max \{ 0, \gamma - \beta \} } e^{(t-s)A} \|_{L(H)}^2
	\| B \|_{ \HS(H, H_\beta)}^2 
	\, ds
	\\
	&
	\leq
	\| (-A)^{ \min \{ 0, \gamma - \beta \} } \|_{L(H)}^2
	\| B \|_{ \HS(H, H_\beta)}^2
	\int_0^t
	( t - s )^{ - 2 \max \{ 0, \gamma - \beta \} }
	\, 
	ds
	\\
	&
	=
	\| (-A)^{ \min \{ 0, \gamma - \beta \} } \|_{L(H)}^2
	\| B \|_{ \HS(H, H_\beta)}^2
	\tfrac{ 
		t^{ 1 - 2 \max \{ 0, \gamma - \beta \} }
	}{
	1 - 2 \max \{ 0, \gamma - \beta \}
	}
	.
	\end{split}
	\end{equation}
	This shows that there
	exists
	a stochastic process
	$ \mathbf{O} \colon [0,T] \times \Omega \to H_{ \gamma } $
	which satisfies for every 
	$ t \in [0,T] $ that
	\begin{equation} 
	\label{eq:Introduce O}
	[ \mathbf{O}_t ]_{ \P, \B( H_{ \gamma } )} 
	=
	\int_0^t e^{(t-s)A} B \, d W_s 
	.
	\end{equation} 
	Observe that~\eqref{eq:Introduce O}
	and the
	Burkholder-Davis-Gundy-type inequality in Da Prato \& Zabczyk~\cite[Lemma~7.7]{dz92}
	prove that
	for every 
 	$ p \in [2, \infty) $,
	$ s \in [0,T] $,
	$ t \in [s,T] $,
	$ \rho \in (0, 
	\min \{ 1, \nicefrac{1}{2} + \beta - \gamma \} ) $
	it holds that
	\begin{equation}
	\begin{split}
	&
	\| \mathbf{O}_t - \mathbf{O}_s \|_{ \L^p(\P; H_{  \gamma } ) }
	=
	\Big\|
	\int_0^t e^{(t-u)A} B \, dW_u
	-
	\int_0^s e^{(s-u)A} B \, dW_u
	\Big\|_{ L^p(\P; H_{ \gamma } ) }
	\\
	&
	\leq 
	\Big\| 
	\int_s^t e^{(t-u)A} B \, dW_u
	\Big\|_{ L^p(\P; H_{ \gamma } ) }
	+
	\Big\|
	\int_0^s
	e^{(s-u)A}
	 (
	e^{(t-s)A}
	-
	\operatorname{Id}_{H_\beta}
	 )
	B
	\, dW_u
	\Big\|_{ L^p(\P; H_{ \gamma } ) }
	\\
	&
	\leq 
	\Big[
	\tfrac{ p (p-1) }{ 2 }
	\int_s^t
	\| (-A)^{ \min \{ 0, \gamma - \beta \} } \|_{L(H)}^2
	\|
	(-A)^{ \max \{ 0, \gamma - \beta \} }
	e^{(t-u)A} \|_{L(H)}^2
	\| B \|_{ \HS(H, H_\beta)}^2
	\, du
	\Big]^{ \nicefrac{1}{2} }
	\\
	&
	\quad
	+
	\Big[
	\tfrac{ p (p-1) }{ 2 }
	\int_0^s
	\| (-A)^{ \min \{0, \rho + \gamma - \beta \} } \|_{L(H)}^2
	\|
	(-A)^{ \max \{0, \rho + \gamma - \beta \} }
	e^{(s-u)A}
	\|_{L(H)}^2
	\\
	&
	\qquad 
	\cdot 
	\|
	(-A)^{-\rho} 
	(
	e^{(t-s)A}
	-
	\operatorname{Id}_H
	)
	\|_{L(H)}^2
	\| B \|_{ \HS(H, H_\beta)}^2
	\, du
	\Big]^{ \nicefrac{1}{2} }
	\\
	&
	\leq 
	\tfrac{ p }{ \sqrt{2} }
	\| B \|_{ \HS(H, H_\beta)} 
	\bigg(
	\| (-A)^{ \min \{ 0, \gamma - \beta \} } \|_{L(H)}
	\Big[ 
	\int_s^t 
	(t-u)^{ - 2 \max \{ 0, \gamma - \beta \} }
	\, du
	\Big]^{ \nicefrac{1}{2} }
	\\
	&
	\quad 
	+
	\| (-A)^{ \min \{0, \rho + \gamma - \beta \} } \|_{L(H)}
	\Big[ 
	\int_0^s
	(s-u)^{ - 2 \max \{ 0, \rho + \gamma - \beta \} }
	( t - s )^{ 2 \rho } 
	\, du
	\Big]^{ \nicefrac{1}{2} }
	\bigg)
	\\
	&
	\leq 
	\tfrac{ p }{ \sqrt{2} }
	\| B \|_{ \HS(H, H_\beta)} 
	\big(
	\| (-A)^{ \min \{ 0, \gamma - \beta \} } \|_{L(H)}
	+
	\| (-A)^{ \min \{0, \rho + \gamma - \beta \} } \|_{L(H)}
	\big)
	\\
	&
	\quad 
	\cdot 
	\Big[
	\tfrac{ (t-s)^{ ( 1 / 2 ) - \max \{ 0, \gamma - \beta \} } }{ 
		\sqrt{ 1 - 2 \max \{ 0, \gamma - \beta \} } }
	+
	(t-s)^{ \rho } 
	\tfrac{ s^{ ( 1 / 2 ) - \max \{ 0, \rho + \gamma - \beta \} } }{
	\sqrt{1 - 2 \max \{ 0, \rho + \gamma - \beta \} } } 
	\Big]
	.
	\end{split}
	\end{equation}
	The Kolmogorov-Chentsov theorem
	(cf., e.g., Kallenberg~\cite[Theorem~2.23]{Kallenberg1997})
	therefore
	assures that
	there exists an up to indistinguishability unique stochastic
	process 
	$ O \colon [0,T] \times \Omega \to H_\gamma $
	with continuous sample paths
	which satisfies for every
	$ t \in [0,T] $
	that
	$ [ O_t ]_{ \P, \B(H_\gamma)} 
	=
	\int_0^t e^{(t-s)A} B \, d W_s $.
	The proof of Lemma~\ref{lemma:Justify}
	is thus completed.
\end{proof}
\begin{lemma}
	\label{lemma:GalerkinRegularity}
\sloppy 
Assume Setting~\ref{setting:main},  
	let  
	$ T \in (0, \infty) $, 
	$ I \subseteq \H $,  
$ \beta \in \R $, 
$ \gamma \in ( -\infty, \frac{1}{2} + \beta ) $, 
$ \alpha \in ( 0,  
 \frac{1}{2} - \max \{0, \gamma - \beta \} ) $,
$ B \in \HS( H, H_\beta ) $,
let
$ \mathbb{B} \in L(H) $,
$ P \in L(H_{ \min \{0, \gamma \} }) $ 
satisfy for every
$ u, v \in H $ that
$ \langle B u, v \rangle_H
=
\langle u, \mathbb{B} v \rangle_H $
%
and
%
$ Pv =\sum_{h \in I} \langle h , v \rangle_H h $,
let
$ ( \Omega, \F, \P ) $
be a probability space, 
for every set $ R $
and every function
$ f \colon \Omega \to R $
let
$ [f]_{\P, \B(H_\gamma)} =
\{
g \in \M(\F,
\B(H_\gamma))
\colon
(
\exists\, D \in \F \colon
\P( D ) = 0 \,\,\text{and}\,\,
\{\omega \in \Omega \colon
f(\omega) \neq g(\omega)
\}
\subseteq D
)
\} $,  
let
$ (W_t)_{t\in [0,T]} $
be an
$ \operatorname{Id}_H $-cylindrical 
 Wiener process,
	and let $ O \colon [0,T] \times \Omega \to H_\gamma $ be a stochastic process
	with continuous sample paths
	which satisfies for every
	$ t \in [0,T] $ 
	that
	$ [O_t]_{\P, \B( H_\gamma )} = \int_0^t e^{(t-s)A} B \, dW_s $.
%
	%
	%
	Then 
	it holds for every
	$ p \in ( \nicefrac{1}{\alpha}, \infty) $
	that
\begin{equation}
\begin{split}
&
\big(
\E \big[
\sup\nolimits_{ t \in [0,T] }
\| P O_t \|_{H_\gamma}^p
\big]
\big)^{\nicefrac{1}{p}}
\leq
T^{ \alpha }
2^{ \alpha - 1 }
\big[ 
\tfrac{ p ( p - 1 ) }{ p \alpha - 1 }
\big] 
\Big[
\Big( 
\int_0^\infty
s^{ - 2 \alpha } e^{-s} \, ds
\Big)
\sum\nolimits_{ h \in I } 
\| \mathbb{B} h \|_{ H }^2
|\values_h|^{ 2 ( \alpha + \gamma ) - 1 } 
\Big]^{\nicefrac{1}{2}}
.
\end{split} 
\end{equation}
\end{lemma}
\begin{proof}[Proof of Lemma~\ref{lemma:GalerkinRegularity}]
Note that for every $ t \in [0,T] $ it holds that
\begin{equation} 
\begin{split} 
&
\int_0^t (t-u)^{-2 \alpha} 
\| e^{(t-u)A} B \|_{\HS(U,H_\gamma)}^2 \,du
\\
&
\leq 
\int_0^t 
(t-u)^{-2 \alpha }  
\| (-A)^{ \max \{0, \gamma - \beta \} } 
e^{(t-u)A}
\|_{L(H)}^2 
\| (-A)^{ \min \{0, \gamma - \beta \} }
\|_{L(H)}^2 
\| B \|_{\HS(U,H_\beta)}^2 \,du 
\\
&
\leq 
\| (-A)^{ \min \{0, \gamma - \beta \} }
\|_{L(H)}^2 
\| B \|_{\HS(U,H_\beta)}^2 
\int_0^t 
(t-u)^{-2 [\alpha + \max \{0, \gamma - \beta \} ]}   
\,du  
< \infty 
.
\end{split} 
\end{equation}
This ensures that
there exists a stochastic process 
	$ Z \colon [0, T] \times \Omega \to H_{ \gamma } $ 
 which satisfies for every 
	$ t \in [0,T] $ 
	that
	\begin{equation} 
	\label{eq:Again Z}
	 [ Z_t ]_{\P, \B( H_{ \gamma } )} 
	= 
	\int_0^t (t-u)^{-\alpha} e^{(t-u)A}  B \, dW_u 
	.
	\end{equation}
	Note
	that~\eqref{eq:Again Z}
	and the triangle inequality prove
	that 
	for every
	$ p \in [2,\infty) $,
	$ t \in (0,T] $, $ s \in [0,t) $
	it holds that
	\begin{equation}
	\begin{split}
	\| Z_t - Z_s \|_{\L^p(\P; H_{ \gamma } )}
	&
    =
	\Big\|
	\int_0^s
	\big(
	(t-u)^{-\alpha} e^{(t-u)A}
	-
	(s-u)^{-\alpha} e^{(s-u)A}
	\big)
	 B \, dW_u
	\Big\|_{L^p(\P; H_{ \gamma })}
	\\ 
	&
	\quad
	+
	\Big\|
	\int_s^t (t-u)^{-\alpha} e^{(t-u)A} B \, dW_u 
	\Big\|_{L^p(\P; H_{ \gamma })}
	\\
	&
	\leq
	\Big\|
	\int_0^s 
	(t-u)^{-\alpha} 
	\big(
	e^{(t-u)A}
	-
	e^{(s-u)A}
	\big)
	B \, dW_u
	\Big\|_{L^p(\P; H_{ \gamma })}
	\\
	&
	\quad
	+
	\Big\|
	\int_0^s 
	e^{(s-u)A}
	\big(
	(t-u)^{-\alpha} 
	-
	(s-u)^{-\alpha}  
	\big)
	B \, dW_u
	\Big\|_{L^p(\P; H_{ \gamma })}
	\\
	&
	\quad
	+
	\Big\|
	\int_s^t (t-u)^{-\alpha} e^{(t-u)A} B \, dW_u 
	\Big\|_{L^p(\P; H_{ \gamma })}
	.
	\end{split}
	\end{equation} 
	The
	Burkholder-Davis-Gundy-type inequality in Da Prato \& Zabczyk~\cite[Lemma~7.7]{dz92}
	hence shows that for every
	$ p \in [2,\infty) $,
	$ t \in (0,T] $, $ s \in [0,t) $
	it holds that
	\begin{equation}
	\begin{split}
	& 
	\| Z_t - Z_s \|_{\L^p(\P; H_{ \gamma } )} 
	\leq
	\tfrac{ \sqrt{p(p-1)} }{ \sqrt{2} }
	\Big[
	\int_0^s 
	(t-u)^{-2\alpha} 
	\big\|
	(-A)^{ \gamma - \beta }
	\big(
	e^{(t-u)A}
	-
	e^{(s-u)A}
	\big)
	B 
	\big\|_{\HS(H, H_\beta)}^2
	\, du
	\Big]^{ \nicefrac{1}{2} }
	\\
	&
	+
	\tfrac{ \sqrt{p(p-1)} }{ \sqrt{2} }
	\Big[
	\int_0^s 
	\big(
	(s-u)^{-\alpha}
	-
	(t-u)^{-\alpha}  
	\big)^2
	\| ( - A )^{ \min \{ 0, \gamma - \beta \} } \|_{L(H)}^2
	\\
	&
	\quad 
	\cdot 
\| ( - A )^{ \max \{ 0, \gamma - \beta \} } e^{(s-u)A} \|_{L(H)}^2
\| B \|_{\HS(H, H_\beta)}^2
	\, du
	\Big]^{ \nicefrac{1}{2} }
	\\
	&
	+
	\tfrac{ \sqrt{p(p-1)} }{ \sqrt{2} }
	\Big[
	\int_s^t (t-u)^{-2\alpha} 
	\| ( - A )^{ \min \{ 0, \gamma - \beta \} } \|_{L(H)}^2
	\| ( - A )^{ \max \{ 0, \gamma - \beta \} } e^{(t-u)A} \|_{L(H)}^2
	\| B \|_{\HS(H, H_\beta)}^2 \, du 
	\Big]^{ \nicefrac{1}{2} }
	.
	\end{split}
	\end{equation}
	Therefore, we obtain that for every
	$ p \in [2,\infty) $,
	$ \varepsilon \in  (0, \frac{ 1 }{ 2 } + \beta - \alpha -  \gamma ) $,
	$ t \in (0,T] $, $ s \in [0,t) $
	it holds 
	that
	\begin{equation}
	\begin{split}
	\label{eq:basic1}
	&
	\| Z_t - Z_s \|_{\L^p(\P; H_{ \gamma } )}
	\leq
	\tfrac{ \sqrt{p(p-1)} }{ \sqrt{2} }
	\| B \|_{\HS(H, H_\beta )}
	\\
	&
	\quad
	\cdot
	\bigg(
	\Big[
	\int_0^s 
	(t-u)^{-2\alpha} 
	\| 
	(-A)^{ \min \{ 0, \varepsilon + \alpha + \gamma - \beta \} }
	\|_{L(H)}
	\|
	(-A)^{ \max \{ 0, \varepsilon + \alpha + \gamma - \beta \} }
	e^{(s-u)A}
	\|_{ L(H) }^2
	\\
	&
	\qquad
	\cdot 
	\| 
	(-A)^{-\varepsilon - \alpha }
	(
	e^{(t-s)A}
	-
	\operatorname{Id}_H
	)
	\|_{ L(H) }^2
	\, du
	\Big]^{ \nicefrac{1}{2} }
	\\
	&
	\quad
	+ 
	\| ( - A )^{ \min \{ 0, \gamma - \beta \} } \|_{ L(H) }
	\Big[
	\int_0^s   
	\big(
	(s-u)^{-\alpha}
	-
	(t-u)^{-\alpha} 
	\big)^2
	(s-u)^{ - 2 \max \{ 0, \gamma - \beta \} }
	\, du
	\Big]^{ \nicefrac{1}{2} }
	\\
	&
	\quad  
	+ 
	\| ( - A)^{ \min \{ 0, \gamma - \beta \} } \|_{L(H)}
	\Big[
	\int_s^t (t-u)^{ - 2 \alpha - 2 \max \{ 0, \gamma - \beta \} }  
	\, du 
	\Big]^{ \nicefrac{1}{2} }
	\bigg)
	\\
	&
	\leq
	\tfrac{ \sqrt{p(p-1)} }{ \sqrt{2} }
	\| B \|_{\HS(H, H_\beta )}
	\\
	&
	\quad 
	\cdot 
	\bigg(
	\|
	( - A )^{ \min \{ 0, \varepsilon + \alpha + \gamma - \beta \} }
	\|_{L(H)} 
	\Big[
	\int_0^s 
	(t-u)^{-2 \alpha} 
	(s-u)^{ - 2 \max \{ 0, \varepsilon + \alpha + \gamma - \beta \} }
	(t-s)^{2 \varepsilon + 2 \alpha }
	\, du
	\Big]^{ \nicefrac{1}{2} }
	\\
	&
	\quad
	+ 
	\| ( - A)^{ \min \{ 0, \gamma - \beta \} } \|_{L(H)}
	\Big[
	\int_0^s    
	\big(
	(s-u)^{-\alpha}
	-
	(t-u)^{-\alpha} 
	\big)^2
	(s-u)^{ - 2 \max \{ 0, \gamma - \beta \}  }
	\, du
	\Big]^{ \nicefrac{1}{2} }
	\\
	&
	\quad
	+ 
	\| ( - A)^{ \min \{ 0, \gamma - \beta \} } \|_{L(H)}
	\Big[
\tfrac{ (t-s)^{ 1 - 2 \alpha - 2 \max \{ 0, \gamma - \beta \} } }{
	 1 - 2 \alpha - 2 \max \{ 0, \gamma - \beta \}
 }
	\Big]^{ \nicefrac{1}{2} }
	\bigg)
	.
	\end{split} 
	\end{equation} 
	In addition, 
	note  
	that for every  
	$ \varepsilon \in  (0, \frac{ 1 }{ 2 }  + \beta - \alpha -  \gamma ) $,
	$ t \in (0,T] $, $ s \in [0,t) $
	it holds that
	\begin{equation}
	\begin{split}
	\label{eq:basic2}
	&
	\int_0^s 
	(t-u)^{-2 \alpha} 
	(s-u)^{ - 2 \max \{0, \varepsilon + \alpha + \gamma - \beta \} }
	( t - s )^{ 2 \varepsilon + 2 \alpha }
	\, du
	\\
	&
	\leq 
	\int_0^s 
	(t-s)^{-2 \alpha} 
	(s-u)^{ - 2 \max \{0, \varepsilon + \alpha + \gamma - \beta \} }
	( t - s )^{ 2 \varepsilon + 2 \alpha }
	\, du
	\\
	&
	\leq 
	(t-s)^{ 2 \varepsilon } 
	\tfrac{ s^{ 1 - 2 \max \{0, \varepsilon + \alpha + \gamma - \beta \} } }
	{
		1 - 2 \max \{0, \varepsilon + \alpha + \gamma - \beta \}
	}
	\leq 
	(t-s)^{ 2 \varepsilon }
	\tfrac{ 
		\max \{ T, 1 \} }
	{
		1 - 2 \max \{0, \varepsilon + \alpha + \gamma - \beta \}
	}
	.
	\end{split}
	\end{equation}
	Next observe that the fact that 
	for every
	$ x, y \in [0,T] $,
	$ z \in [0,1] $
	it holds that
	$ | x^z - y^z | \leq |x - y|^z $
	ensures that for every
	$ t \in (0,T] $, $ s \in (0,t) $,
	$ u \in (0,s) $
	it holds that
	\begin{equation}
	(s-u)^{-\alpha} - (t-u)^{-\alpha}
	\leq
	\tfrac{
		(t-s)^\alpha
	}
	{
		(s-u)^\alpha (t-u)^\alpha 
	}
	.
	\end{equation}
	H\"older's inequality 
	hence proves that for every
	$ \varepsilon \in 
	(0,  \min \{ \frac{ 1 }{  8 ( \alpha + \max \{ 0, \gamma - \beta \} )  }  -  \frac{1}{4}, \frac{1}{4}, \alpha \} ) $,
	$ t \in (0,T] $, 
	$ s \in [0,t) $ 
	it holds 
	that
	\begin{equation}
	\begin{split}
	\label{eq:basic3}
	&
	\int_0^s    
	\big(
	(s-u)^{-\alpha}
	-
	(t-u)^{-\alpha} 
	\big)^2
	(s-u)^{ - 2 \max \{ 0, \gamma - \beta \}  }
	\, du
	\leq 
	\int_0^s 
	\tfrac{ 
		(t-s)^{2\alpha}
	}
	{
		(s-u)^{2\alpha} (t-u)^{2\alpha} 
	}
	(s-u)^{ - 2 \max \{ 0, \gamma - \beta \} }
	\, du
	\\
	&
	= 
	(t-s)^{ 2 \varepsilon }
	\int_0^s 
	(s-u)^{ - 2 \alpha - 2 \max \{ 0, \gamma - \beta \} }
	(t-u)^{-2\alpha}
	(t-s)^{ 2 \alpha - 2 \varepsilon } 
	\,du
	\\
	&
	\leq 
	(t-s)^{ 2 \varepsilon }
	\int_0^s 
	(s-u)^{ - 2 \alpha - 2 \max \{ 0, \gamma - \beta \} }
	(t-u)^{ - 2 \varepsilon } 
	\,du
	\\
	&
	\leq 
	(t-s)^{ 2 \varepsilon }
	\Big(
	\int_0^s 
	(s-u)^{ - 2 ( \alpha + \max \{ 0, \gamma - \beta \} ) 
		(1 + 4 \varepsilon)} 
	\,du
	\Big)^{\nicefrac{1}{(1 + 4 \varepsilon) }}
	\Big(
	\int_0^s 
	(t-u)^{ - \nicefrac{ 2 \varepsilon (1 + 4 \varepsilon) }{ 4 \varepsilon } }
	\, du
	\Big)^{ \nicefrac{ 4 \varepsilon }{ ( 1 + 4 \varepsilon ) } }
	\\
	&
	= 
	(t-s)^{ 2 \varepsilon }
	\big(
	\tfrac{ s^{ 1 - 2 ( \alpha + \max \{ 0, \gamma - \beta \} ) (1 + 4 \varepsilon) } }
	{
		1 - 2 ( \alpha + \max \{ 0, \gamma - \beta \} ) (1 + 4 \varepsilon) 
	}
	\big)^{\nicefrac{1}{(1 + 4 \varepsilon) }}
	\big(
	\tfrac{ 
		t^{  ( 1/2 ) - 2\varepsilon  }
		-
		(t-s)^{  ( 1/2 ) - 2\varepsilon   }	
	}
	{
		( 1/2 ) - 2 \varepsilon 
	}
	\big)^{ \nicefrac{ 4 \varepsilon }{ ( 1 + 4 \varepsilon ) } }
	\\
	&
	\leq 
	(t-s)^{ 2 \varepsilon }
	\tfrac{
		\max \{ T, 1 \} 
	}{
		(
		1 - 2 ( \alpha + \max \{ 0, \gamma - \beta \} ) (1 + 4 \varepsilon) 
		)^{\nicefrac{1}{(1 + 4 \varepsilon) }}
		(
		( 1/2 ) - 2 \varepsilon 
		)^{ \nicefrac{4 \varepsilon}{(1 + 4 \varepsilon)} }
	}
	.
	\end{split}
	\end{equation}
	Combining
	this,
	\eqref{eq:basic1},
	and~\eqref{eq:basic2}
	%
	demonstrates that 
	for every
	$ \varepsilon \in (
	0,
	\min \{ 
	 \frac{ 1 }{ 8 ( \alpha + \max \{ 0, \gamma - \beta \} ) }  -  \frac{1}{4},
	\frac{1}{4},
	\alpha,
	\frac{ 1 }{ 2 } + \beta - \alpha - \gamma 
	\}
	) $,
	$ p \in [2,\infty) $ 
	it holds that
	\begin{equation}
	\sup\nolimits_{ t \in (0,T], s \in [0,t) }
	\tfrac{ \| Z_t - Z_s \|_{\L^p(\P; H_{ \gamma } ) } }
	{ | t- s |^\varepsilon }
	< \infty.
	\end{equation}
	The Kolmogorov-Chentsov theorem
	(cf., e.g., Kallenberg~\cite[Theorem~2.23]{Kallenberg1997})
	therefore
	assures that there exists a stochastic process
	$ \mathcal{Z} \colon [0,T] \times \Omega \to H_{ \gamma } $
	with continuous sample paths
	which satisfies for every
	$ t \in [0,T] $ that
	\begin{equation} 
	\label{eq:Modified}
	\P( \mathcal{Z}_t = Z_t) = 1 
	.
	\end{equation} 
	Next note that 
	the fact that
	$ 0 < \alpha < \frac{1}{2} - \max \{ 0, \gamma - \beta \} $
	ensures that
	for every   
	$ t \in [0,T] $ it holds that 
	\begin{equation}
	\begin{split}
	&
	\int_0^t 
	(t-s)^{\alpha - 1}
	\Big[
	\int_0^s 
	(s- u)^{-2\alpha}
	\E \big[  \| e^{(t-u)A} P B  \|_{\HS(H, H_{ \gamma } )}^2 \big] 
	\, du
	\Big]^{\nicefrac{1}{2}}
	\, ds
	\\
	&
	=
	\int_0^t 
	(t-s)^{\alpha - 1}
	\Big[
	\int_0^s 
	(s- u)^{-2\alpha}
	\| ( - A)^{ \min \{0, \gamma  - \beta \} + \max \{ 0, \gamma - \beta \} } e^{(t-u)A} P B  \|_{\HS(H, H_\beta)}^2 
	\, du
	\Big]^{\nicefrac{1}{2}}
	\, ds
	\\
	&
	\leq
	\| B \|_{\HS(H, H_\beta )}
	\| (-A)^{ \min \{ 0, \gamma - \beta \} } \|_{L(H)}
	\int_0^t 
	(t-s)^{\alpha - 1}
	\Big[
	\int_0^s 
	(s- u)^{-2 ( \alpha + \max \{ 0, \gamma - \beta \} ) }
	\, du
	\Big]^{\nicefrac{1}{2}}
	\, ds
	\\
	&
	=
	\| B \|_{\HS(H, H_\beta )}
	\| (-A)^{ \min \{ 0, \gamma - \beta \} } \|_{L(H)}
	\tfrac{ 1 }
	{ \sqrt{ 1 - 2 ( \alpha + \max \{ 0, \gamma - \beta \} )  } }
	\int_0^t 
	(t-s)^{\alpha - 1}
	 s^{ (1/2) - ( \alpha + \max \{ 0, \gamma - \beta \} ) }  
	\, ds 
	\\
	&
	\leq
	\| B \|_{\HS(H, H_\beta)} 
	\| (-A)^{ \min \{ 0, \gamma - \beta \} } \|_{L(H)}
	\tfrac{ [ \max \{ T, 1 \} ]^{ 1/2 } }{ \sqrt{ 1 - 2 ( \alpha + \max \{ 0, \gamma - \beta \} )  } }
	\tfrac{ t^{\alpha} }{ \alpha }
	< \infty
	.
	\end{split}
	\end{equation}
	Combining~\eqref{eq:Again Z},
	\eqref{eq:Modified},
the fact that
$ \mathcal{Z} \colon [0,T] \times 
\Omega \to H_{ \gamma } $
has continuous sample paths,
item~\eqref{item:well defined} of
Lemma~\ref{lemma:ZContinuous} 
(with 
$ T = T $,
$ \alpha = \alpha $,
$ \gamma = \gamma $,
$ \mathcal{Z} = 
( [0,T]  \ni t
\mapsto 
P \mathcal{Z}_t(\omega)
\in H_\gamma ) $ 
for  
$ \omega \in \Omega $
in the notation of 
item~\eqref{item:well defined} of
Lemma~\ref{lemma:ZContinuous}), 
and, e.g., 
	Da Prato \& Zabczyk~\cite[Theorem~5.10]{DaPratoZabczyk2014} therefore establishes that
	for every    
	$ t \in [0,T] $
	it holds that
	\begin{equation}
	\begin{split}
	\int_0^t e^{(t-s)A} P B \, dW_s
	=
	\Big[
	\tfrac{ \sin( \alpha \pi )}{ \pi }
	\int_0^t (t-s)^{\alpha -1} e^{(t-s)A} P \mathcal{Z}_s \, ds
	\Big]_{\P, \B( H_\gamma ) }
	.
	\end{split}
	\end{equation}
This,
the fact that
$ \mathcal{Z} \colon [0,T] \times \Omega \to H_{ \gamma } $
has continuous sample paths, 
and  
Lemma~\ref{lemma:ZContinuous} 
%
%
%
%
%
(with
$ T = T $,
$ \alpha = \alpha $,
$ \gamma = \gamma $,
$ \mathcal{Z} = 
( [0,T]  \ni t
\mapsto 
P \mathcal{Z}_t(\omega)
\in H_\gamma ) $ 
for 
$ \omega \in \Omega $
in the notation of Lemma~\ref{lemma:ZContinuous})
imply
 that  
 for every
 $ \omega\in\Omega $, 
 $ p\in[1,\infty) $ 
 it holds 
  that
 $ ( [0,T] \ni t \mapsto \int_0^t
 (t-s)^{\alpha - 1 }
e^{(t-s)A} 
 P \mathcal{Z}_s(\omega)
 \, ds \in H_\gamma ) \in \mathcal{C}( [0,T], H_\gamma) $ 
 and 
	\begin{equation}
	\begin{split}
	\E \Big[
	\sup\nolimits_{ t \in [0,T] }
	\| P O_t \|_{H_\gamma}^p
	\Big]
	&
	=
	\E\Big[
	\sup\nolimits_{ t \in [0,T] }
	\Big\|
	\tfrac{ \sin( \alpha \pi  ) }{ \pi }
	\int_0^t
	(t-s)^{\alpha - 1 }
	e^{(t-s)A} 
	P \mathcal{Z}_s
	\, ds
	\Big\|_{H_\gamma}^p
	\Big]
	\\
	&
	\leq 
	\E \Big[ 
	\sup\nolimits_{ t \in [0,T] }
	\Big(
	\int_0^t 
	(t-s)^{\alpha - 1 } 
	\|  
	e^{(t-s)A}
	P 
	\mathcal{Z}_s
	\|_{H_\gamma}
	\, ds
	\Big)^p
	\Big]
	.
	\end{split}
	\end{equation}
	H\"older's inequality and Tonelli's theorem hence prove that for every 
	$ p \in ( \nicefrac{1}{\alpha},\infty) $ 
	it holds that
	\begin{equation}
	\begin{split}
	\label{eq:proj_estimate}
	&
	\E \big[
	\sup\nolimits_{ t \in [0,T] }
	\| P O_t \|_{H_\gamma}^p
	\big]
	\leq 
	\E \Big[ 
	\sup\nolimits_{ t \in [0,T] }
	\Big(
	\int_0^t 
	(t-s)^{\alpha - 1 } 
	\| 
	P
	\mathcal{Z}_s
	\|_{H_\gamma}
	\, ds
	\Big)^p
	\Big]
	\\
	&
	\leq
	\E \Big[
	\sup \nolimits_{ t \in [0,T] }
	\Big \{
	\Big( 
	\int_0^t (t-s)^{ \frac{p(\alpha-1)}{p-1} } \, ds
	\Big)^{ p - 1 }
	\Big( 
	\int_0^t 
	\| P \mathcal{Z}_s \|_{H_\gamma}^p \, ds
	\Big)
	\Big \}
	\Big]
	\\
	&
	\leq
	\big(
	\tfrac{ t ^{ 1 + ( p(\alpha-1)/(p-1) ) } }{ 1 + ( p(\alpha-1)/(p-1) ) } 
	\big)^{p-1} 
	\int_0^T \E \big [ \| P \mathcal{Z}_s \|_{H_\gamma}^p \big ] \, ds
	\leq
	\big(
	\tfrac{ p-1} { p\alpha - 1 }
	\big)^{p-1}
	T^{ p \alpha }
	\sup\nolimits_{ s \in [0,T] } 
	\E \big [ \| P \mathcal{Z}_s \|_{H_\gamma}^p  \big]
	.
	\end{split}
	\end{equation}
	In addition, observe that
	the
	Burkholder-Davis-Gundy-type inequality in Da Prato \& Zabczyk~\cite[Lemma~7.7]{dz92}
	shows that for every  
	$ p \in ( \nicefrac{1}{\alpha},\infty) $,
	$ t \in [0,T] $ 
	it holds that
	\begin{equation}
	\begin{split}
	\| P \mathcal{Z}_t \|_{\L^p(\P; H_\gamma)}^2
	&
	=
	\Big \|
	\int_0^t
	(t-u)^{-\alpha} 
	e^{(t-u)A} P B \, dW_u
	\Big \|_{L^{ p }(\P; H_\gamma)}^2
	\\
	&
	\leq
	\tfrac{ p ( p - 1 ) }{ 2 }
	\int_0^t 
	(t-u)^{-2\alpha}
	\|  
	e^{(t-u)A} P B \|_{\HS(H,H_\gamma)}^2
	\, du  
	.
	\end{split} 
	\end{equation} 
	Tonelli's theorem therefore 
	implies that for every 
	$ p \in (\nicefrac{1}{\alpha}, \infty) $,
	$ t \in [0,T] $ 
	it holds that
	\begin{equation}
	\begin{split}
	&
	\| P \mathcal{Z}_t \|_{\L^p(\P; H_\gamma)}^2
	\leq
	\tfrac{ p ( p - 1 ) }{ 2 } 
	\int_0^t 
	(t-u)^{ - 2 \alpha }
	\| ( -A)^\gamma  e^{(t-u)A} P B \|_{ \HS(H) }^2
	\, du 
	\\
	&
	= 
	\tfrac{ p ( p - 1 ) }{ 2 }
	\int_0^t 
	(t-u)^{ - 2 \alpha }
	\| \mathbb{B} (-A)^\gamma P e^{(t-u)A} \|_{ \HS(H) }^2
	\, du
	\\
	&
	= 
	\tfrac{ p ( p - 1 ) }{ 2 } 
	\int_0^t 
	(t-u)^{ - 2 \alpha }
	\sum_{ h \in I }
	\| \mathbb{B} (-A)^\gamma e^{(t-u)A} h \|_{ H }^2
	\, du
	\\
	&
 =
	\tfrac{ p ( p - 1 ) }{ 2 } 
	\int_0^t 
	(t-u)^{ - 2 \alpha }
	\sum_{ h \in I } 
	\| \mathbb{B}   h \|_{ H }^2
	| \values_h |^{ 2 \gamma }
	e^{2(t-u) \values_h }
	\, du
	\\
	&
	=
	\tfrac{ p ( p - 1 ) }{ 2 }
	\sum\nolimits_{ h \in I  } 
	\| \mathbb{B}  h \|_{ H }^2
		\int_0^t 
	(t-u)^{ - 2 \alpha }
	e^{ 2 (t-u) \values_h }
	| \values_h |^{ 2 \gamma }
	\, du 
	\\
	&
	=
	\tfrac{ p ( p - 1 ) }{ 2 }
	\sum\nolimits_{ h \in I  }
	\| \mathbb{B} h \|_{ H }^2
	\Big( \int_0^{ 2 | \values_h | t }
	s^{ - 2  \alpha  } 
	e^{-s}
	\, ds
	\Big) 
	2^{ 2 \alpha - 1 }
	(   | \values_h | )^{  2 ( \alpha  + \gamma ) - 1 } 
 \\
 &
 \leq
 2^{ 2 \alpha - 2 } p^2  
 \Big( 
 \int_0^\infty
 s^{ - 2  \alpha } e^{-s} \, ds
 \Big)
 \sum\nolimits_{ h \in I } 
\| \mathbb{B} h \|_{ H }^2
 |\values_h|^{ 2 ( \alpha  + \gamma )- 1 }   
 .
	\end{split}
	\end{equation}
	Combining this with~\eqref{eq:proj_estimate}
	ensures that
	\begin{equation}
	\begin{split}
	&
	\big(
	\E \big[
	\sup\nolimits_{ t \in [0,T] }
	\| P O_t \|_{H_\gamma}^p
	\big]
	\big)^{\nicefrac{1}{p}}
	\\
	&
	\leq
	2^{ \alpha - 1 }
	p
	\big(
	\tfrac{ p-1} { p\alpha - 1 }
	\big)^{ \nicefrac{ ( p - 1 ) }{ p } }
	T^{  \alpha }
	\Big[
	\Big( 
	\int_0^\infty
	s^{ - 2 \alpha } e^{-s} \, ds
	\Big)
	\sum\nolimits_{ h \in I } 
	\| \mathbb{B} h \|_{ H }^2
	|\values_h|^{ 2 ( \alpha + \gamma ) - 1 } 
	\Big]^{\nicefrac{1}{2}}
	.
	\end{split} 
	\end{equation}
	The proof of Lemma~\ref{lemma:GalerkinRegularity}
	is thus completed.
\end{proof}
\begin{lemma} 
	\label{lemma:ConvergenceSpeed} 
	Let $ ( V, \left \| \cdot \right \|_V ) $
	be an $ \R $-Banach space,
	let $ ( \Omega, \F, \P ) $
	be a probability space,
	let
	$ \alpha \in (0, \infty) $, 
	and
	let 
	$ Z_n \colon \Omega \to V $,
	$ n \in \N $,
	be $ \F / \B(V) $-measurable functions
	which satisfy
	for every  
	$ p \in [1, \infty) $
	that
	$ \sup_{ n \in \N } 
	( n^\alpha \| Z_n \|_{\L^p(\P; V) }  ) < \infty $.
	Then it holds for every 
	$ \varepsilon \in (0, \infty ) $,
	$ p \in [1, \infty) $
	that
	\begin{equation} 
	\label{eq:some auxiliary eq}
	\P\big( 
	\sup\nolimits_{ n \in \N } 
	( n^{\alpha - \varepsilon } \| Z_n \|_V ) < \infty 
	\big)
	=
	1
	\qquad
	\text{and}
	\qquad
	\E \big[
	(
	\sup\nolimits_{ n \in \N }
	(n^{\alpha - \varepsilon}
	\| Z_n \|_V
	)
	)^p
	\big] < \infty 
	.
	\end{equation}
\end{lemma}
\begin{proof}[Proof of Lemma~\ref{lemma:ConvergenceSpeed}]
	%
	%
	Observe that
	for every
	$ \varepsilon, \delta \in (0, \infty) $,
	$ p \in ( \max \{ \nicefrac{1}{\varepsilon}, 1 \}, \infty) $
	it holds that
	\begin{equation}
	\begin{split}
	&
	\E [
	(
	\sup\nolimits_{ n \in \N }
	(
	n^{\alpha - \varepsilon}
	\| Z_n \|_V
	)
	)^p
	]
	=
	\E [
	\sup\nolimits_{ n \in \N }
	(
	n^{p(\alpha - \varepsilon)}
	\| Z_n \|_V^p
	)
	]
	\\
	&
	\leq
	\sum_{n=1}^\infty
	n^{ p ( \alpha - \varepsilon ) }
	\E[ \| Z_n \|_V^p ]
	\leq
	(
	\sup\nolimits_{ n \in \N } 
	( n^\alpha \| Z_n \|_{\L^p(\P; V) }  )
	)^p
	\sum_{n=1}^\infty 
	n^{- p \varepsilon}
	<
	\infty
	.
	\end{split}
	\end{equation}
	Jensen's inequality therefore 
	demonstrates that for every
	$ \varepsilon \in (0, \infty) $,
	$ p \in [1, \infty) $
	it holds that
	\begin{equation}
	\E [
	(
	\sup\nolimits_{ n \in \N }
	(n^{\alpha - \varepsilon}
	\| Z_n \|_V
	)
	)^p
	] < \infty 
	.
	\end{equation}
	This
	establishes~\eqref{eq:some auxiliary eq}.
	The proof of Lemma~\ref{lemma:ConvergenceSpeed}
	is thus completed.
\end{proof}
\begin{lemma}
	\label{lemma:PathwiseRates}
	%
Assume Setting~\ref{setting:Examples},
	let $ T \in (0, \infty ) $, 
$ \beta \in \R $,
$ \gamma \in ( - \infty, \nicefrac{1}{2} + \beta ) $,
$ B \in \HS(H, H_\beta) $,
let $ \mathcal{P}(\H) $ be the power set of $ \H $, 
let 
$ ( P_I )_{ I\in \mathcal{P} (\H) } \subseteq L(H_{ \min \{0, \gamma \} }) $ 
satisfy
for every
$ I\in \mathcal{P}(\H) $,  
$ v \in H_{ \min \{0, \gamma \} } $ 
that
$ P_I(v) =\sum_{h \in I} \langle (-A)^{ - \min \{0, \gamma \} }h ,  (-A)^{ \min \{0, \gamma \} }v \rangle_H h $,
let
$ ( \Omega, \F, \P ) $
be a probability space,
let
$ (W_t)_{t\in [0,T]} $
be an
$ \operatorname{Id}_H $-cylindrical 
 Wiener process,
and
		let $ O \colon [0,T] \times \Omega \to H_\gamma $ be a stochastic process
		with continuous sample paths
		which satisfies for every
		$ t \in [0,T] $ 
	that  
	$ [O_t]_{\P, \B(H_\gamma)} = \int_0^t e^{(t-s)A} B \, dW_s $. 	 
	Then
	\begin{equation}
	\begin{split} 
	\label{eq:statement}
	\P \big(
	\forall \, \eta \in (- \infty, 1 + 2 ( \beta - \gamma ) )
	\colon
	\sup\nolimits_{ n \in \N }
	(
	n^\eta
	\sup\nolimits_{ t \in [0,T] }
	\|   P_{\H \backslash\{ e_1, \ldots, e_n \}}  O_t \|_{H_\gamma} 
	)
	< 
	\infty 
	\big)
	=
	1
	.
	\end{split}
	\end{equation}
\end{lemma}
\begin{proof}[Proof of Lemma~\ref{lemma:PathwiseRates}]
	\sloppy 
	Throughout this proof let
	$ \mathbb{B} \in L(H) $
	satisfy for every
	$ u, v \in H $ that
	$ \langle B u, v \rangle_H
	=
	\langle u, \mathbb{B} v \rangle_H $
	and let
	$ (I_n)_{ n \in \N } \subseteq \H $
	satisfy for every $ n \in \N $ that
	$ I_n = \{ e_1, \ldots, e_n \} $. 
	Note that Lemma~\ref{lemma:GalerkinRegularity}
	(with
	$ T = T $,
	$ I = \H \backslash I_n $, 
	$ \beta = \beta $,
	$ \gamma = \gamma $,
	$ \alpha = \alpha $, 
	$ B = B $,
	$ \mathbb{B} = \mathbb{B} $,
	$ P = P_{ \H \backslash I_n } $,
	$ ( \Omega, \F, \P ) = ( \Omega, \F, \P ) $,
	$ ( W_t )_{ t \in [0,T] } = ( W_t )_{ t \in [0,T] } $,
	$ O = O $
	for  
	$ n \in \N $,
	$ \alpha \in (0,  \frac{1}{2}  -
	\max \{0, \gamma - \beta \}  ) $
	in the notation of Lemma~\ref{lemma:GalerkinRegularity})
	ensures that for every 
	$ n \in \N $, 
	$ \alpha \in ( 0,  \frac{1}{2} -
	\max \{0,  \gamma - \beta \} ) $,
	$ p \in ( \nicefrac{1}{\alpha}, \infty) $ 
	it holds that
\begin{equation}
	\begin{split}
	&
	\big( 
	\E \big[
	\sup\nolimits_{ t \in [0,T] }
	\| P_{ \H \backslash I_n } O_t \|_{H_\gamma}^p
	\big]
	\big)^{\nicefrac{1}{p}}
	\\
	&
	\leq 
		2^{ \alpha - 1 }
	\tfrac{ p ( p-1) }{ p \alpha - 1 }
	T^{  \alpha }
	\Big[
	\Big(
	\int_0^\infty
	s^{ - 2 \alpha } e^{-s} \, ds
	\Big)
	\sum\nolimits_{ h \in \H \backslash I_n } 
	\| \mathbb{B} h \|_{ H }^2
	|\values_h|^{ 2 ( \alpha + \gamma ) - 1 }  
	\Big]^{\nicefrac{1}{2}} 
	\\
	&
	\leq 
		2^{ \alpha  - 1 }
	\tfrac{ p ( p-1) }{ p \alpha - 1 }
	T^{  \alpha }
	\Big[
	\Big( 
	\int_0^\infty
	s^{- 2 \alpha } e^{-s} \, ds
	\Big)
	\sum\nolimits_{ h \in \H }
	\| \mathbb{B} h \|_{ H }^2
	| \values_h |^{2 \beta } 
	\Big]^{\nicefrac{1}{2}}
	\big( 
	\sup\nolimits_{ h \in \H \backslash I_n }
	( 
	|\values_h|^{ \alpha + \gamma -\beta - ( 1/2 ) }
	)
	\big)
	\\
	&
	=
		2^{ \alpha  - 1 }
	\tfrac{ p ( p-1) }{ p \alpha - 1 }
	T^{  \alpha }
	\Big[
	\Big( 
	\int_0^\infty
	s^{- 2 \alpha } e^{-s} \, ds
	\Big)
	\sum\nolimits_{ h \in \H }
	\| \mathbb{B} ( - A )^{ \beta } h \|_H^2
	\Big]^{\nicefrac{1}{2}} 
	\big[ \sqrt{ |c_0| } \pi (n+1) \big]^{ 2 ( \alpha + \gamma -\beta ) - 1 } 
	\\
	&
	=
	2^{ \alpha  - 1 }
	\tfrac{ p ( p-1) }{ p \alpha - 1 }
	T^{  \alpha }
	\Big[
	\Big( 
	\int_0^\infty
	s^{- 2 \alpha } e^{-s} \, ds
	\Big)
	\| \mathbb{B} ( - A )^{ \beta }  \|_{\HS(H)}^2
	\Big]^{\nicefrac{1}{2}} 
	\big[ \sqrt{ |c_0| } \pi (n+1) \big]^{ 2 ( \alpha + \gamma -\beta ) - 1 } 
		\\
	&
	=
	2^{ \alpha  - 1 }
	\tfrac{ p ( p-1) }{ p \alpha - 1 }
	T^{  \alpha }
	\Big[
	\Big( 
	\int_0^\infty
	s^{- 2 \alpha } e^{-s} \, ds
	\Big)
	\| B \|_{\HS(H, H_\beta)}^2
	\Big]^{\nicefrac{1}{2}} 
	\big[ \sqrt{ |c_0| } \pi (n+1) \big]^{ 2 ( \alpha + \gamma -\beta ) - 1 } 
	.
	\end{split}
	\end{equation}
	Jensen's inequality
	hence implies that for every  
	$ \alpha \in (0, \frac{1}{2}  - \max \{0, \gamma - \beta \} ) $, 
	$ p \in [ 1, \infty) $
	it holds  
	that
	\begin{equation}
	\begin{split}
	\sup\nolimits_{ n \in \N } 
	\Big\{ 
	n^{ 1 + 2 ( \beta - \alpha - \gamma ) } 
	\big(
	\E 
	\big[  
	\sup\nolimits_{ t \in [0,T] }
	\|  P_{ \H \backslash I_n } O_t \|_{ H_\gamma }^p
	\big]
	\big)^{ \nicefrac{1}{p} }
	\Big\}
	< \infty 
	.
	\end{split}
	\end{equation}
	Lemma~\ref{lemma:ConvergenceSpeed}
	(with
	$ V = \R $,
	$ ( \Omega, \F, \P ) =
	( \Omega, \F, \P ) $,
	$ \alpha = 1 + 2( \beta - \alpha - \gamma ) $,
	$ Z_n = \sup\nolimits_{ t \in [0,T] }
	\|  P_{ \H \backslash I_n } O_t \|_{ H_\gamma } $
	for 
	$ n \in \N $,
	$ \alpha \in (0, \frac{1}{2} - \max \{0, \gamma - \beta \} ) $
	in the notation of
	Lemma~\ref{lemma:ConvergenceSpeed})
	therefore shows that 
	for every  
	$ \alpha \in (0,  \frac{1}{2}  - \max \{0, \gamma - \beta \} ) $,
	$ \eta \in (0, 1 + 2 ( \beta - \alpha - \gamma ) ) $
	it holds that
	\begin{equation}
	\begin{split}
	\P 
	\big(
	\sup\nolimits_{ n \in \N }
	( 
	n^{\eta} 
	\sup\nolimits_{ t \in [0,T] }
	\| P_{ \H  \backslash I_n } O_t \|_{H_\gamma} 
	)
	<
	\infty 
	\big)
	=
	1
	.
	\end{split}
	\end{equation}
	This completes the proof of Lemma~\ref{lemma:PathwiseRates}.
\end{proof}
\begin{lemma}
\label{Lemma:Existence of approximation processes} 
%
Assume Setting~\ref{setting:Examples}, 
let
$ T \in (0,\infty) $,  
$ \beta \in \R $,
$ \gamma \in (- \infty, \nicefrac{1}{2} + \beta ) $,
$ B \in \HS(H, H_\beta) $, 
let
$ ( \Omega, \F, \P ) $
be a probability space with a normal filtration
$ ( \f_t )_{t \in [0,T]} $,
let
$ (W_t)_{t\in [0,T]} $
be an
$ \operatorname{Id}_H $-cylindrical 
$ ( \f_t )_{t \in [0,T]} $-Wiener process,  
let  
$ \xi \in \M
( \f_0, \mathcal{B}( H ) ) $, 
let $ \mathcal{P}(\H) $ be the power set of $ \H $, 
let $ \mathcal{P}_0(\H) = \{ \theta \in \mathcal{P}(\H) \colon \theta \text{ is a finite set} \} $, 
 let 
$ ( P_I )_{ I\in \mathcal{P}_0 (\H) } \subseteq L(H_{ \min \{0, \gamma\} }, H) $ 
satisfy
for every 
$ I\in \mathcal{P}_0(\H) $,  
$ v \in H_{\min \{0, \gamma\} } $ 
that
$ P_I(v) =\sum_{h \in I} 
\langle (-A)^{ - \min \{0, \gamma \} }h ,  (-A)^{ \min \{0, \gamma \} }v \rangle_H
h $,
and
let $ O \colon [0,T] \times \Omega \to H_\gamma $ be a stochastic process
with continuous sample paths
which satisfies for every
$ t \in [0,T] $ 
that  
$ [O_t]_{\P, \B(H_\gamma)} = \int_0^t e^{(t-s)A} B \, dW_s $. 	
Then 
there exist
$ ( \f_t )_{ t \in [0,T] } $-adapted stochastic
processes
$ X^I \colon [0,T] \times \Omega \to P_I(H) $,
$ I \in \mathcal{P}_0(\H) $, 
with continuous sample paths
such that
for every 
$ I \in \mathcal{P}_0(\H) $,
$ t \in [0,T] $ it holds that
\begin{equation}
X_t^I 
= 
e^{tA} P_I \xi + \int_0^t e^{(t-s)A} P_I F(X_s^I) \, ds + P_I O_t 
.
\end{equation}
\end{lemma}
\begin{proof}[Proof of Lemma~\ref{Lemma:Existence of approximation processes}.]
	Throughout this proof
	let
	$ \Phi \colon H_{\nicefrac{1}{2}} \to [0, \infty) $
	be the function which satisfies for every
	$ w \in H_{\nicefrac{1}{2}} $ that
	\begin{equation}
	\Phi(w) 
	=
	\tfrac{ 3 | c_1 |^2 }{ 8  | c_0 |  } 
	\bigg[   
	\sup_{ u \in H_{\nicefrac{1}{2}} \backslash \{ 0 \} }
	\tfrac{ \| u \|_{ L^\infty(\lambda; \R) } }{ \| u \|_{ H_{\nicefrac{1}{2}} } }
	+
	\sup_{ u \in H_{\nicefrac{1}{2}} \backslash \{ 0 \} }
	\tfrac{ \| u \|_{ L^4(\lambda; \R) }^2 }{ \| u \|_{ H_{\nicefrac{1}{2}} }^2 } 
	\bigg]^2
	( 
	1
	+ 
	\| w \|_{ H_{\nicefrac{1}{2}} }^2
	)^2
	,
	\end{equation}
	%
	%
	%
	%
	let
	$ \mathcal{A}_I \colon P_I(H) \to P_I(H) $,
	$ I \in \mathcal{P}_0(\H) $,
	be the linear operators
	which 
	satisfy for every
	$ I \in \mathcal{P}_0(\H) $,
	$ v \in P_I(H) $
	that
	$ \mathcal{A}_I v = A v $, 
	and 
	for every
	$ I \in \mathcal{P}_0(\H) $
	let
	$ ( \mathcal{H}_{I, s}, \langle \cdot, \cdot \rangle_{ \mathcal{H}_{I, s} },
	\left \| \cdot \right\|_{ \mathcal{H}_{I, s} } ) $,
	$ s \in \R $,
	be a family of interpolation spaces
	associated
	to $ - \mathcal{A}_I $. 
	Note that item~\eqref{item:02} 
	of Lemma~\ref{lemma:local_lip}
	proves that
	for every
	$ I \in \mathcal{P}_0(\H) $,
	$ v, w \in H_{ \nicefrac{1}{2} } $
	it holds that
	\begin{equation}
	\begin{split}
	\label{eq:Existence condition 1}
	\| P_I F(v) - P_I F(w) \|_H
	\leq  
	\| F(v) - F(w) \|_H 
	\leq
	\tfrac{ |c_1| }{\sqrt{3 } \, c_0 } 
	( 
	\| v \|_{ H_{ \nicefrac{1}{2} } } 
	+ 
	\| w \|_{ H_{ \nicefrac{1}{2} } } 
	)
	\| v - w \|_{ H_{ \nicefrac{1}{2} } }
	.
	\end{split}
	\end{equation}
	Moreover, observe that
	Corollary~\ref{corollary:AppropriateCoercivityEstimate}  
	(with $ \iota = \nicefrac{1}{2} $,
	$ v = v $, $ w = w $
	for $ v, w \in H_{ \nicefrac{1}{2} } $
	in the notation of 
	Corollary~\ref{corollary:AppropriateCoercivityEstimate})  
	shows that for every
	$ I \in \mathcal{P}_0(\H) $,
	$ v, w \in P_I (H) \subseteq H_{ \nicefrac{1}{2} } $
	it holds that
	\begin{equation}
	\begin{split}
	\label{eq:Existence condition 2}
	&
	\langle v, P_I F( v + w ) \rangle_H
	=
	\langle P_I v, F( v + w ) \rangle_H
	=
	\langle v, F( v + w ) \rangle_H 
	\\
	&
	\leq 
	\tfrac{ 3 | c_1 |^2 }{ 8  | c_0 |  } 
	\bigg[   
	\sup_{ u \in H_{\nicefrac{1}{2}} \backslash \{ 0 \} }
	\tfrac{ \| u \|_{ L^\infty(\lambda; \R) } }{ \| u \|_{ H_{\nicefrac{1}{2}} } }
	+
	\sup_{ u \in H_{\nicefrac{1}{2}} \backslash \{ 0 \} }
	\tfrac{ \| u \|_{ L^4(\lambda; \R) }^2 }{ \| u \|_{ H_{\nicefrac{1}{2}} }^2 } 
	\bigg]^2
	\big( 
	\| v \|_H^2
	+ 
	\| w \|_{ H_{\nicefrac{1}{2}} }^2
	\big) 
	\| w \|_{ H_{\nicefrac{1}{2}} }^2
	+  
	\| v \|_{H_{\nicefrac{1}{2}}}^2
	\\
	&
	\leq  
\Phi(w) 
	( 
	1
	+
	\| v \|_H^2
	) 
	+  
	\| v \|_{H_{\nicefrac{1}{2}}}^2
	<
	\infty 
	.
	\end{split}
	\end{equation}
	Combining~\eqref{eq:Existence condition 1} 
	and 
	Corollary~\ref{corollary:Existence} (with
%
	$ (H, \langle \cdot, \cdot \rangle_H,
	\left \| \cdot \right \|_H ) = 
	( P_I (H), \langle \cdot, \cdot \rangle_{H},
	\left \| \cdot \right \|_{H} ) $,
	$ \H = I $, 
	$ \values_{e_n} = - c_0 \pi^2 n^2  $,
	$ A = \mathcal{A}_I $,
	$ ( H_s )_{ s \in \R } = ( \mathcal{H}_{I, s} )_{ s \in \R } $,
	$ T = T $, 
	$ s = 0 $,
	$ C = \nicefrac{ | c_1 | }{ c_0 } $,  
	$ c = 1 $,
	$ \delta = \nicefrac{1}{2} $,
	$ \kappa = \nicefrac{1}{2} $, 
	$ F = ( P_{I}(H) \ni x \mapsto P_{I} F(x) \in P_{I}(H) ) $,
	%
	$ \Phi = ( P_I(H) \ni x \mapsto \Phi(x) \in [0, \infty) ) $,
	$ ( \Omega, \F, \P, ( \f_t )_{ t \in [0,T ] }) 
	= 
	( \Omega, \F, \P, ( \f_t )_{ t \in [0,T ] } ) $,
	$ \xi = ( \Omega \ni \omega \mapsto P_{I} \xi (\omega) \in P_{I }( H ) ) $,
	$ O = ( [0,T] \times \Omega \ni ( t, \omega ) \mapsto 
	P_{I} O_t( \omega ) \in P_{I}( H ) ) $
	%
	%
	for 
	$ I \in \mathcal{P}_0(\H) $,
%
	$ n\in\{m\in\N\colon e_m\in P_I(\H)\} $
	in the notation of Corollary~\ref{corollary:Existence})
	therefore
	completes the proof of Lemma~\ref{Lemma:Existence of approximation processes}.
\end{proof}
\begin{theorem}
	\label{theorem:existence_Burgers}
	%
	%
%
	Let
	$ \lambda
	\colon
	\mathcal{B}( (0,1) )
	\rightarrow [0, 1] $
	be the Lebesgue-Borel
	measure on
	$ (0,1) $, 
	for every measure space
	$ ( \Omega, \F,  \mu) $,
	every measurable space $ ( S, \mathcal{S} ) $,
	every set $ R $,
	and every function
	$ f \colon \Omega \to R $
	let
	$ [f]_{\mu,\mathcal{S}}
	=
	\{g \colon \Omega 
	\to S \colon 
	(\exists\, D \in \F  \colon[ \mu (D)=0
	\text{ and }
	\{\omega\in
	\Omega \colon f(\omega)\neq g(\omega)\}
	\subseteq D]) \text{ and } (\forall\, D\in\mathcal{S} \colon g^{-1}( D )\in 
	\F )\} $,
	let
	$ T, \varepsilon, c_0 \in (0,\infty) $,  
	$ c_1 \in \R $,
	$ \beta \in ( - \nicefrac{1}{4}, \infty) $,
	$ \gamma \in ( \nicefrac{1}{4}, \min \{ 1, \nicefrac{1}{2} + \beta \} ) $,
	$ (H, \langle \cdot, \cdot \rangle_H,
	\left \| \cdot \right \|_H)  
	=
	(L^2( \lambda; \R), 
	\langle \cdot, \cdot 
	\rangle_{L^2( \lambda;\R)},
	\left \| \cdot \right \|_{L^2( \lambda; \R)}
	) $,  
	let
	$ ( e_n )_{ n \in \N } \subseteq H $  
	satisfy for every  
	$ n \in \N $
	that    
	$ e_n = [ ( \sqrt{2} \sin( n \pi x ) )_{x \in (0,1) } ]_{ \lambda, \B(\R ) } $,  
	let
	$ A \colon D( A ) \subseteq H \to H $ 
	be the linear operator which satisfies
	$ D(A) = \{
	v \in H \colon \sum_{ n=1}^\infty  | n^2 \langle e_n, v \rangle_H  |^2 <  \infty
	\} $ 
	and  
	$ \forall \, v \in D(A) \colon 
	A v = - \sum_{ n=1}^\infty c_0 \pi^2 n^2 \langle e_n, v \rangle _H e_n $,
	let
	$ (H_r, \langle \cdot, \cdot \rangle_{H_r}, \left \| \cdot \right \|_{H_r} ), r\in \R $, be a family of interpolation spaces associated to $ -A $
	(cf., e.g., \cite[Section~3.7]{SellYou2002}),
	for every  
	$ v \in W^{1,2}((0,1), \R) $
	let
	$ \partial v    
	\in H $ 
	satisfy for every 
	$ \varphi  
	\in 
	\mathcal{C}_{cpt}^\infty( (0,1), \R ) $
	that
	$ \langle \partial  v, 
	[ \varphi ]_{ \lambda, \mathcal{B}(\R )} \rangle_{H}
	=
	-
	\langle v , 
	[ \varphi' ]_{ \lambda, \mathcal{B}(\R )}
	\rangle_{H} $, 
	let
	$ ( \Omega, \F, \P ) $
	be a probability space with a normal filtration
	$ ( \f_t )_{t \in [0,T]} $,
	let
	$ (W_t)_{t\in [0,T]} $
	be an
	$ \operatorname{Id}_H $-cylindrical 
	$ ( \f_t )_{t \in [0,T]} $-Wiener process, 
	let 
	$ B \in \HS(H, H_\beta) $,
	and let
	$ \xi \colon \Omega \to H_{ \gamma + \varepsilon } $ 
	be an
	$ \f_0 / \mathcal{B}( H_{ \gamma + \varepsilon } ) $-measurable function.
	Then 
	\begin{enumerate}[(i)]
		\item \label{item:Continuous extension}
		there exists a unique continuous function
			$ F \colon H_{\nicefrac{1}{8}} \to H_{ - \nicefrac{1}{2} } $ 
		which satisfies
		for every
		$ v \in H_{ \nicefrac{1}{2} } $  
		that  
		$ F(v) =
		 c_1
		v \partial v
		$
		and
		\item \label{item:main statement} \sloppy there exists an up to indistinguishability unique
	$ (\f_t )_{ t \in [0,T] } $-adapted stochastic process
	$ X \colon [0,T] \times \Omega \to H_\gamma $
	with continuous sample paths which satisfies
	for every 
	$ t \in [0,T] $
	that
	\begin{equation}
	[
	X_t
	]_{\P, \B(H_\gamma)} 
	= 
	\Big[
	e^{tA}\xi 
	+  
	\int_0^t e^{(t-s)A} F(X_s) \, ds 
	\Big]_{\P, \B(H_\gamma)}
	+
	\int_0^t e^{(t-s)A} B \, dW_s
	.
	\end{equation}
	\end{enumerate}
\end{theorem}
\begin{proof}[Proof of Theorem~\ref{theorem:existence_Burgers}]
Throughout this proof
let
$ f \colon H_{\nicefrac{1}{2} } \to H $
be the function
which satisfies for every 
$ v \in H_{ \nicefrac{1}{2} } $
that
$ f(v) = c_1 v \partial v $,
let
$ \H = \{ e_n \colon n \in \N \} $,
let $ \mathcal{P}(\H) $ be the power set of $ \H $, 
let $ \mathcal{P}_0(\H) = \{ \theta \in \mathcal{P}(\H) \colon \theta \text{ is a finite set} \} $,
let 
$ ( P_I )_{ I\in \mathcal{P} (\H) } \subseteq L(H) $ 
satisfy
for every 
$ I\in \mathcal{P}(\H) $,  
$ v \in H $ 
that
$ P_I(v) =\sum_{h \in I} \left< h ,v \right>_H h $,
let
$ \nu = \nicefrac{
( 2 - 4 \min \{ \gamma, 1 / 2 \} )
}{3} $,
$ \eta \in (0, \min \{ 2 \varepsilon, 1 + 2 ( \beta - \gamma ),
2 ( 1 - \gamma - \nu ) \} ) $,
and
let
$ (I_n)_{n \in \N} \subseteq \mathcal{P}_0(\H) $ 
satisfy
for every $ n \in \N $ that
$ I_n = \{ e_1, \ldots, e_n \} $.
Note that
item~\eqref{item: Extension of F 2} of
Corollary~\ref{corollary:Extension of Function F part2} 
(with
$ F = f $,
$ \bar F = F $
in the notation of 
Corollary~\ref{corollary:Extension of Function F part2})
establishes 
item~\eqref{item:Continuous extension}.
Next we intend to apply Bl\"omker \& Jentzen~\cite[Theorem~3.1]{BloemkerJentzen2013}
to prove item~\eqref{item:main statement}.
For this observe that
 Lemma~\ref{lemma:Justify}
 (with
 $ H = H $,
 $ \H = \H $,
 $ \values_{ e_n } = - c_0 \pi^2 n^2 $,
 $ A = A $,
 $ H_r = H_r $,
 $ T = T $,
 $ \beta = \beta $,
 $ \gamma = \gamma $,
 $ B = B $,
 $ ( \Omega, \F, \P ) = ( \Omega, \F, \P ) $,
 $ ( W_t )_{ t \in [0,T] } 
 =
 ( W_t )_{ t \in [0,T] } $
 for 
 $ n \in \N $,
 $ r \in \R $
 in the notation of 
 Lemma~\ref{lemma:Justify})
 ensures
 that there exists 
 an
$ ( \f_t )_{ t \in [0,T] } $-adapted stochastic process
$ O \colon [0,T] \times \Omega \to H_\gamma $
with continuous sample paths 
which satisfies for every 
$ t \in [0,T] $ that
\begin{equation} 
\label{eq:more O}
[ O_t ]_{\P, \B(H_\gamma)} = \int_0^t e^{(t-s)A} B \, dW_s 
.
\end{equation} 
%
%
%
Note that~\eqref{eq:more O} 
and
Lemma~\ref{Lemma:Existence of approximation processes}
(with
$ c_0 = c_0 $,
$ c_1 = c_1 $,
$ H = H $,
$ \H = \H $,
$ \values_{ e_n } = - c_0 \pi^2 n^2 $,
$ e_n = e_n $,
$ A = A $,
$ H_r = H_r $,
$ F = f $,
$ T = T $,
$ \beta = \beta $,
$ \gamma = \gamma $,
$ B = B $,
$ (\Omega, \F, \P) = (\Omega, \F, \P) $,
$ (\f_t)_{ t \in [0,T] } = (\f_t)_{ t \in [0,T] } $,
$ (W_t)_{ t \in [0,T] } = (W_t)_{ t \in [0,T] } $,
$ \xi = ( \Omega \ni \omega \mapsto \xi(\omega) \in H ) $,
$ P_I = P_I $,
$ O = O $
for $ n \in \N $, $ r \in \R $, $ I \in \mathcal{P}_0(\H) $
in the notation of
Lemma~\ref{Lemma:Existence of approximation processes})
show that
there exist 
$ ( \f_t )_{ t \in [0,T] } $-adapted stochastic processes
$ X^I \colon [0,T] \times \Omega \to P_{I} ( H ) $, 
$ I \in \mathcal{P}_0( \H ) $, 
with continuous sample paths 
which satisfy for every 
$ I \in \mathcal{P}_0( \H ) $,
$ t \in [0,T] $ that
\begin{equation}
\begin{split}
\label{eq:processEx}
X_t^I = e^{tA} P_{I} \xi + \int_0^t e^{(t-s)A} P_{I} f (  X_s^I  ) \, ds
+
P_{I} O_t  
.
\end{split} 
\end{equation}
Next let
$ \Sigma \in \F $
be the set which satisfies that
\begin{equation}
\label{eq:Sigma2}
	\Sigma =
	\big\{ \omega \in \Omega \colon 
%
		\sup\nolimits_{ n \in \N }
		(
		n^{ \eta }
		\sup\nolimits_{ t \in [0,T] }
		\|  O_t(\omega) - P_{I_n}  O_t(\omega) \|_{H_\gamma} 
		) 
		< \infty
		\big\} 
		,
	\end{equation}
	%
%
%
%
%
	let
	$ \O \colon [0,T] \times \Omega \to H_\gamma $
	be the stochastic process which satisfies
	for every $ t \in [0,T] $,
	$ \omega \in \Omega $ that
	\begin{equation} 
	\label{eq:Define versions}
	\mathcal{O}_t( \omega ) 
	= 
	\begin{cases} 
	 O_t(\omega) 
	& \colon \omega \in \Sigma
	\\
	- e^{tA} \xi ( \omega ) 
	- \int_0^t e^{(t-s)A} f(0) \, ds 
	&\colon 
	\omega \in ( \Omega \backslash \Sigma ) 
	,
	\end{cases}
	\end{equation}
	and let
	$ \mathcal{X}^I \colon [0,T] \times \Omega \to P_{I} ( H ) $,
	$ I \in \mathcal{P}_0( \H ) $,
	be the stochastic processes
	which satisfy for every 
	$ I \in \mathcal{P}_0( \H ) $,
	$ t \in [0,T] $, 
	$ \omega \in \Sigma $
	that  
	\begin{equation} 
	\label{eq:Define versions 2}
	\mathcal{X}_t^{I} (\omega ) =
	\begin{cases}
	X_t^{I}(\omega ) & \colon \omega \in \Sigma \\
	0 & \colon \omega \in ( \Omega \backslash \Sigma )
	.
	\end{cases}
	\end{equation}
	%
	%
	%
	%
%
%
Moreover, note that
the fact that
$ ( \gamma + \nu )
\in (0, 1) $
shows that
for every 
$ t \in (0,T] $ it holds that
\begin{equation}
\begin{split}
%
\| e^{tA} \|_{ L( H_{ - \nu }, H_\gamma ) } 
\leq
t^{-\gamma - \nu } 
\leq
t^{ - \gamma - \nu - ( \nicefrac{\eta}{2} ) }
T^{ \nicefrac{\eta}{2} }
.
\end{split}
\end{equation}
This ensures that
\begin{equation}
\label{eq:Ass1a}
\sup\nolimits_{ t \in (0,T] }
\big( 
t^{ \gamma + \nu + ( \nicefrac{\eta}{2} ) } 
\| e^{tA} \|_{ L( H_{ - \nu }, H_\gamma ) } 
\big)
<
\infty.
\end{equation}
In addition, observe that 
the fact that
$ ( \gamma + \nu + ( \nicefrac{\eta}{2} ) )
\in (0, 1) $
implies that
for every 
$ n \in \N $,
$ t \in [0,T] $ 
it holds
that
\begin{equation}
\begin{split}
%
&
\| P_{ \H \backslash I_n }
e^{tA} 
\|_{L(H_{ - \nu }, H_\gamma)}
=
\| 
(-A)^{ - \nicefrac{\eta}{2} } 
P_{ \H \backslash I_n } 
 (-A)^{ \gamma + \nu + (\nicefrac{\eta}{2}) } e^{tA} 
\|_{L(H)}
\\
&
\leq
\| 
(-A)^{ - \nicefrac{\eta}{2} } 
P_{ \H \backslash I_n } 
\|_{L(H)}
\| 
(-A)^{ \gamma + \nu + ( \nicefrac{\eta}{2} ) } e^{tA} 
\|_{L(H)}
\leq 
[ c_0 \pi^2 (n+1)^2 ]^{ - \nicefrac{\eta}{2} } 
t^{ - \gamma - \nu - ( \nicefrac{\eta}{2} ) }
. 
\end{split}
\end{equation}
This proves that
\begin{equation}
\begin{split}
\label{eq:Ass1b}
\sup\nolimits_{n \in \N }
\sup\nolimits_{t \in (0,T] }
\big( 
t^{ \gamma + \nu + ( \nicefrac{\eta}{2} ) }
n^{ \eta } 
\| 
e^{tA}
-
P_{ I_n }
e^{tA} 
\|_{L(H_{ - \nu }, H_\gamma)}
\big)
<
\infty.
\end{split}
\end{equation}
Next note that 
the fact that
for every $ x \in ( \nicefrac{1}{8}, \nicefrac{1}{2} ] $
it holds that
\begin{equation}  
\big( \tfrac{ ( 2 - 4 x ) }{ 3 } \big)
\in 
\big( \big[\tfrac{1}{2} - x, 
\tfrac{1}{2} \big] \cap \big( \tfrac{3}{4} - 2 x, \infty \big) \big) 
\end{equation} 
%
%
and 
Lemma~\ref{lemma:Extension of Function F part1}
(with      
$ \gamma = \min \{ \gamma, \nicefrac{1}{2} \} $,  
$ \nu = \nicefrac{ ( 2 - 4 \min \{ \gamma,  1/2 \} ) }{ 3 } $  
in the
notation of Lemma~\ref{lemma:Extension of Function F part1})
ensure that 
there exists 
$ C \in [0, \infty) $ 
such that
for every 
$ v, w \in H_{ \gamma }
\subseteq 
H_{ \min \{ \gamma, 1/2 \} } $ 
it holds that
\begin{equation}
\begin{split}
\label{eq:Ass2}
&
\| F(v) - F(w) \|_{ H_{ - \nicefrac{ ( 2 - 4 \min \{ \gamma, 1/2 \} ) }{ 3 } } }
\leq  
C 
\| v - w \|_{ H_{ \min \{ \gamma,  1/2 \} } } 
(   
1
+
\| v \|_{ H_{ \min \{ \gamma, 1/2 \} } }   
+
\| w \|_{ H_{ \min \{ \gamma, 1/2 \} } }
)   
\\
&
\leq 
C 
\big[ 
\max \{ 1, 
\| (-A)^{\min \{ 0, (1/2)  - \gamma \} } \|_{L(H)} \} \big]^2
\| v - w \|_{ H_{\gamma } }
(   
1
+
\| v \|_{ H_{ \gamma } }   
+
\| w \|_{ H_{ \gamma } }
)   
.
\end{split}
\end{equation}
%
%
%
%
This demonstrates that
there exists $ C \in \R $
such that
for every
$ v, w \in H_{ \gamma } $
it holds that
\begin{equation} 
\label{eq:Local lip}
\| F(v) - F(w) \|_{ H_{ - \nu } }
=
\|  F(v) -  F(w) \|_{ H_{ - \nicefrac{ ( 2 - 4 \min \{ \gamma, 1/2 \} ) }{ 3 } } }
\leq  
C 
\| v - w \|_{ H_{ \gamma } } 
(   
1
+
\| v \|_{ H_{ \gamma } }   
+
\| w \|_{ H_{ \gamma } }
)   
.
\end{equation} 
Furthermore,
observe that~\eqref{eq:Sigma2},
the fact that
$ \eta \in (0, 1 + 2 ( \beta - \gamma ) ) $,  
and
Lemma~\ref{lemma:PathwiseRates}
(with
$ T = T $,
$ \beta =  \beta $,
$ \gamma = \gamma $,
$ B = B $,
$ P_I = P_I $,
$ ( \Omega, \F, \P ) = ( \Omega, \F, \P ) $,
$ ( W_t )_{ t \in [0,T] } = ( W_t )_{ t \in [0,T] } $,
$ O = O $
for 
$ I \in \mathcal{P}(\H) $
in the notation of 
Lemma~\ref{lemma:PathwiseRates})
show that $ \P ( \Sigma ) = 1 $.
%
This and~\eqref{eq:Define versions}
prove that
\begin{equation} 
\label{eq:modif2}
\P ( \forall \, t \in [0,T] \colon \O_t = O_t ) = 1 
.
\end{equation}
In the next step we note that
the fact that
$ f(0) = 0 $ ensures that
for every 
$ n \in \N $,
$ \omega \in ( \Omega \backslash \Sigma ) $,
$ t \in [0,T] $ 
it holds that
\begin{equation}
\begin{split}
\label{eq:FirstRegularity}
&
\| P_{ \H \backslash I_n}  
(
\O_t(\omega)  
+
e^{tA} \xi(\omega) 
)
\|_{H_\gamma}
=
\Big\|
 P_{ \H \backslash I_n}  
\int_0^t e^{(t-s)A} f(0) \, ds 
\Big\|_{H_\gamma} 
=
0.
%
\end{split}
\end{equation}
Furthermore, observe that for every 
$ n \in \N $, 
$ \omega \in \Omega $,
$ t \in [0,T] $ 
it holds that
\begin{equation}
\| P_{ \H \backslash I_n } e^{tA} \xi(\omega) \|_{H_\gamma}
\leq
(
\sup\nolimits_{ h \in \H \backslash I_n }
| \values_h |^{-\nicefrac{\eta}{2}}
) 
\| \xi (\omega) \|_{ H_{\gamma +  ( \nicefrac{\eta}{2} )  } } 
=
[ c_0 \pi^2 (n+1)^2 ]^{ - \nicefrac{\eta}{2} }
\| \xi(\omega) \|_{ H_{\gamma + ( \nicefrac{\eta}{2} ) } } 
.
\end{equation}
Combining
this, 
\eqref{eq:Sigma2},
\eqref{eq:FirstRegularity},
and the triangle inequality
demonstrates that for every
$ \omega \in \Omega $ 
it holds that
	\begin{equation}
	\begin{split}
	\label{eq:Galerking_Speed}
	\sup\nolimits_{ n \in \N }
	(
	n^\eta
	\sup\nolimits_{ t \in [0,T] }
	\| P_{ \H \backslash I_n }
	( \O_t(\omega)
	+ e^{tA} \xi(\omega) ) \|_{H_\gamma} 
	)
	< \infty 
	.
	\end{split}
	\end{equation}
Moreover, note that~\eqref{eq:processEx},
\eqref{eq:Define versions},
and~\eqref{eq:Define versions 2}
ensure that for every 
$ I \in \mathcal{P}_0( \H ) $,
$ \omega \in \Omega $,
$ t \in [0,T] $
it holds that
\begin{equation}
\begin{split}
\label{eq:Important for extension}
\mathcal{X}_t^{ I } ( \omega ) = e^{tA} P_{I} \xi (\omega)  
+
\int_0^t e^{(t-s)A} P_{I}  f( \mathcal{X}_s^I(\omega) ) \, ds 
+
P_{I} \O_t(\omega)
.
\end{split}
\end{equation}
In addition,
observe that
the fact that $ \O \colon [0,T] \times \Omega \to H_\gamma $ has continuous sample paths 
establishes 
for every $ \omega \in \Omega $ that
\begin{equation}
\begin{split}  
\sup\nolimits_{ t \in [0,T] } 
\|  \O_t(\omega) \|_{H_\gamma} < \infty 
.
\end{split}
\end{equation}
The fact that $ \gamma < 1 $,
\eqref{eq:Important for extension},
and
Lemma~\ref{lemma:AprioriBound3}
(with
%
$ F = f $,
$ T = T $,     
$ \iota = \gamma $,
$ \gamma = \gamma $,
$ \xi = \xi(\omega) $,
$ O_t^{ I } = P_{I }\O_t(\omega) $,
$ X_t^{ I } = \mathcal{X}_t^{I }(\omega) $
for
$ n \in \N $,
$ I \in \mathcal{P}_0( \H ) $,
$ t \in [0,T] $,
$ \omega \in \Omega $
in the notation of 
Lemma~\ref{lemma:AprioriBound3})
therefore prove that
for every
$ \omega \in \Omega $ it holds that
\begin{equation}
\begin{split} 
\label{eq:unifBound}
\sup\nolimits_{ n \in \N }
\sup\nolimits_{ t \in [0,T] } 
\| \mathcal{X}_t^{ I_n } (\omega) \|_{H_\gamma} < \infty 
.
\end{split}
\end{equation}
Furthermore, note that
item~\eqref{item:Continuous extension} 
and~\eqref{eq:Important for extension}
show that 
for every 
$ I \in \mathcal{P}_0( \H ) $,
$ \omega \in \Omega $,
$ t \in [0,T] $
it holds that
\begin{equation}
\begin{split} 
\mathcal{X}_t^{ I } ( \omega ) = e^{tA} P_{I} \xi (\omega)  
+
\int_0^t e^{(t-s)A} P_{I}  F( \mathcal{X}_s^I(\omega) ) \, ds 
+
P_{I} \O_t(\omega)
.
\end{split}
\end{equation}
Combining
the fact that
$ 0 < \gamma + \nu + ( \nicefrac{\eta}{2} ) < 1 $,
\eqref{eq:Ass1a}, 
\eqref{eq:Ass1b}, 
\eqref{eq:Local lip},
\eqref{eq:Galerking_Speed}, 
and
\eqref{eq:unifBound} 
with Bl\"omker \& Jentzen~\cite[Theorem~3.1]{BloemkerJentzen2013}  
(with
$ T = T $,
$ ( \Omega, \F, \P )
= ( \Omega, \F, \P ) $,
$ V = H_\gamma $,
$ W = H_{ - \nu } $, 
$ P_n = P_{I_n} $,
$ \alpha = \gamma + \nu + ( \nicefrac{\eta}{2} ) $,
$ \gamma = \eta $,
$ S = ( (0,T] \ni s \mapsto e^{sA} \in L(H_{-\nu}, H_\gamma) ) $,
$ F = ( H_\gamma \ni v \mapsto F(v) \in H_{-\nu} ) $,
$ O_t = \O_t + e^{tA} \xi $,
$ X^n_t = \mathcal{X}^{ I_n }_t $
for 
$ t \in [0,T] $,
$ n \in \N $
in the notation of Bl\"omker \& Jentzen~\cite[Theorem~3.1]{BloemkerJentzen2013})
therefore shows that
\begin{enumerate}[(a)]
	\item \label{item:ProvesExistence} there exists a  unique   
stochastic process
$ X \colon [0,T] \times \Omega \to H_\gamma $
with continuous sample paths
which satisfies for every{\tiny }
$ t \in [0,T] $ that
\begin{equation}
\begin{split} 
\label{eq:ProvesExistence}
X_t = e^{tA} \xi + \int_0^t e^{(t-s)A} F(X_s) \, ds + \O_t
\end{split} 
\end{equation}
and
\item \label{item:B} there exists
a $ \F / \B([0,\infty) ) $-measurable
function 
$ K \colon \Omega \to [0, \infty) $
such that for every 
$ \omega \in \Omega $,
$ n \in \N $
it holds that
\begin{equation}
\sup\nolimits_{ t \in [0,T] }
\| X_t(\omega) - \X_t^{ I_n }(\omega) \|_{ H_\gamma } 
\leq
\tfrac{ 
K(\omega)
}{   n^{\eta} }
.
\end{equation}
\end{enumerate}
Observe that
the fact that
for every 
$ n \in \N $ 
it holds that
$ (\X_t^{ I_n } )_{ t \in [0,T]} $
is $ ( \f_t )_{ t \in [0,T] } $-adapted  
and
item~\eqref{item:B} 
imply that
$ ( X_t )_{ t \in [0,T] } $
is $ ( \f_t )_{ t \in [0,T] } $-adapted.
Combining~\eqref{eq:more O},
\eqref{eq:modif2}, and
item~\eqref{item:ProvesExistence}
hence 
establishes item~\eqref{item:main statement}.
This 
completes the proof of Theorem~\ref{theorem:existence_Burgers}.
\end{proof}
\subsubsection*{Acknowledgements}
This project has been partially supported through the SNSF-Research project $ 200021\_156603 $ ''Numerical 
approximations of nonlinear stochastic ordinary and partial differential equations''.
\newpage
\bibliographystyle{acm}
\bibliography{../Bib/bibfile}

\end{document}